%% file: main.tex
\documentclass[final]{siamltex}

\usepackage{epsfig,amssymb,latexsym}
\usepackage{amsfonts,psfrag,amsmath,bbm,color}
\usepackage{cancel}
\usepackage{siunitx}
\usepackage{comment}
\usepackage{mathrsfs}
\usepackage{graphicx}
\usepackage{textcomp}
\usepackage{multirow}
\usepackage{enumerate}
\usepackage{cancel}
\usepackage{algpseudocode}
\usepackage{caption}  %use \captionof{figure / table}{caption1}
\usepackage{subcaption}
\usepackage{url}
\usepackage{rotating}
\usepackage{bigints}
\usepackage{hyperref}
\usepackage{xcolor}
\usepackage{ulem}
\usepackage{mathrsfs}

\usepackage{placeins}

\def\grad{{\nabla}}
\usepackage[ruled,vlined]{algorithm2e}
\usepackage{geometry}
\vbadness=\maxdimen

\newcommand{\footremember}[2]{%
    \footnote{#2}
    \newcounter{#1}
    \setcounter{#1}{\value{footnote}}%
}
\newcommand{\footrecall}[1]{%
    \footnotemark[\value{#1}]%
} 

\usepackage{tikz}
\usetikzlibrary{shapes.geometric, arrows}
\tikzstyle{startstop} = [rectangle, rounded corners, minimum width=1cm, minimum height=1cm,text centered, draw=black]
\tikzstyle{io} = [trapezium, trapezium left angle=70, trapezium right angle=110, minimum width=1cm, minimum height=1cm, text centered, draw=black, fill=blue!30]
\tikzstyle{method} = [rectangle, rounded corners, minimum width=1cm, minimum height =1cm, text centered, draw=black]
\tikzstyle{process} = [rectangle, minimum width=1cm, minimum height=1cm, text centered, draw=black]
\tikzstyle{decision} = [diamond, minimum width=0.5cm, minimum height=0.5cm, text centered, draw=black, fill=green!30]
\tikzstyle{arrow} = [thick,->,>=stealth]
\usepackage{amscd}

\graphicspath{{./figs/}}
\input{notation-siam}

\begin{document}
\title{Two Robust, Efficient, and optimally Accurate Algorithms for parameterized stochastic navier-stokes Flow Problems}

\author{
%Jahrul Alam\footremember{mun}
%{I\MakeLowercase{nterdisciplinary} S\MakeLowercase{cientific} C\MakeLowercase{omputing} P\MakeLowercase{rogram}, M\MakeLowercase{emorial} U\MakeLowercase{niversity of} N\MakeLowercase{ewfoundland}, NL, C\MakeLowercase{anada.}}
Neethu Suma Raveendran\footremember{uabm}{D\MakeLowercase{epartment of} M\MakeLowercase{athematics}, U\MakeLowercase{niversity of} A\MakeLowercase{labama at} B\MakeLowercase{irmingham}, AL 35294, USA; P\MakeLowercase{artially supported by the} N\MakeLowercase{ational} S\MakeLowercase{cience} F\MakeLowercase{oundation grant} DMS-2213274.}%
\and
Md Abdul Aziz\footrecall{uabm}
\and Muhammad Mohebujjaman\footnote{C\MakeLowercase{orrespondence: mmohebuj@uab.edu}}\hspace{1mm}\footrecall{uabm}
 }

%\date{\today}
\maketitle

\begin{abstract}
This paper presents and analyzes two robust, efficient, and optimally accurate fully discrete finite element algorithms for computing the parameterized Navier-Stokes Equations (NSEs) flow ensemble. The timestepping algorithms are linearized, use the backward-Euler method for approximating the temporal derivative, and Ensemble Eddy Viscosity (EEV) regularized. The first algorithm is a coupled ensemble scheme, and the second algorithm is decoupled using projection splitting with grad-div stabilization. We proved the stability and convergence theorems for both algorithms. We have shown that for sufficiently large grad-div stabilization parameters, the outcomes of the projection scheme converge to the outcomes of the coupled scheme. We then combine the Stochastic Collocation Methods (SCMs) with the proposed two Uncertainty Quantification (UQ) algorithms. A series of numerical experiments are given to verify the predicted convergence rates and performance of the schemes on benchmark problems, which shows the superiority of the splitting algorithm.
\end{abstract}



\pagestyle{myheadings}
\thispagestyle{plain}

\markboth{\MakeUppercase{A penalty-projection efficient algorithm for NSE flow ensemble}}{\MakeUppercase{ N. S. Raveendran, M. A. Aziz, and M. Mohebujjaman}}

\section{Introduction}
Let $\cD\subset \mathbb{R}^d\ (d=2,3)$ be a convex polygonal or polyhedral physical domain with boundary $\partial\mathcal{D}$. A complete probability space is denoted by $(\Omega,\mathcal{F},P)$ with $\Omega$ the set of outcomes, $\mathcal{F}\subset 2^\Omega$ the $\sigma$-algebra of events, and $P:\mathcal{F}\rightarrow [0,1]$ represents a probability measure. We consider the time-dependent, dimensionless, viscoresistive, incompressible Stochastic Navier-Stokes Equations (SNSEs) for computing the 
homogeneous Newtonian fluid flow which is governed by the following non-linear stochastic PDEs:\begin{align}
    \bu_{t}+\bu\cdot\nabla \bu-\nabla\cdot\left(\nu(\bx,\omega)\nabla \bu\right)+\nabla p &=  \bif(\bx,t,\omega), \hspace{2mm}\text{in}\hspace{2mm}\mathcal{D} \times (0,T],\label{momentum}\\
\nabla\cdot \bu & = 0, \hspace{2mm}\text{in}\hspace{2mm}\mathcal{D} \times (0,T], \\
\bu(\bx,0,\omega)& = \bu^0(\bx),\hspace{2mm}\text{in}\hspace{2mm}\mathcal{D},\label{nse-initial}
\end{align}
together with appropriate boundary conditions. The simulation end time is represented by $T>0$, $\bx$ the spatial variable, and $t$ the time variable. The external force is represented by $\bif$ in the momentum equation \eqref{momentum}.
The viscosity $\nu(\bx,\omega)$ is modeled as a random field with $\omega\in\Omega$. Here, the unknown quantities are the velocity field $\bu:\Lambda\rightarrow\mathbb{R}^d$, and the modified pressure $p:\Lambda\rightarrow\mathbb{R}$, where $\Lambda:=\mathcal{D}\times (0,T]\times \Omega$.

The $L^2(\cD)$ inner product is denoted by $(\cdot,\cdot)$. For Hilbert spaces of velocity $\bX=\bH_0^1(\cD)$, pressure $Q=L^2_0(\cD)$, and stochastic space $\bW=\bL_P^2(\Omega)$, the weak form of \eqref{momentum}-\eqref{nse-initial} can be represented as: Find $\bu \in \bX \otimes \bW$ and $p \in Q \otimes \bW$ which, for almost all $t\in(0,T]$, satisfy
	\begin{align}
	\mathbb{E}[(\bu_{t},\bv)]+ \mathbb{E}[(\bu\cdot\nabla \bu,\bv)]+ \mathbb{E}[(\nu\nabla \bu,\nabla \bv)]- \mathbb{E}[(p
	,\nabla\cdot \bv)]  &  = \mathbb{E}[(\bif,\bv)],&\quad\forall \bv\in \bX \otimes \bW,\\
	\mathbb{E}[(\nabla\cdot \bu,q)]  &  =0,&\quad\forall q\in Q \otimes \bW.
	\end{align}
For the input functions, we make the following assumption $\nu(\bx,\omega)=\nu(\bx,\by(\omega))$, $\bif(\bx,t,\omega)=\bif(
\bx,t,\by(\omega))$, and $\by(\omega)=(y_1(\omega),y_2(\omega),\cdots,y_N(\omega))$ with $\mathbb{E}[\by]=\textbf{0}$, and $\mathbb{V}ar[\by]=\textbf{I}_N$.
Find $\bu\in \bX \otimes \bY$ and $p \in Q \otimes \bY$ which, for almost all $t\in(0,T]$, satisfy
\begin{align}\label{weak_formulation_final}
	\int_{\Gamma}  (\bu_{t},\bv)\rho(\by)d\by &+ \int_{\Gamma} (\bu\cdot\nabla \bu,\bv)\rho(\by)d\by + \int_{\Gamma}  (\nu(\bx,\by)\nabla \bu,\nabla \bv)\rho(\by)d\by \nonumber\\&  -\int_{\Gamma}  (p ,\nabla\cdot \bv)\rho(\by)d\by 
	 = \int_{\Gamma}  (\bif,\bv)\rho(\by)d\by, \qquad\qquad\qquad\qquad\qquad\forall \bv\in \bX \otimes \bY,\\
	&~~~\int_{\Gamma}  (\nabla\cdot \bu,q) \rho(\by)d\by    =0, \qquad \qquad \qquad \qquad ~~~~~~~~~~~~~~~~~~~~~~~\;\;\forall q\in Q \otimes \bY.
	\end{align}
  We assume affine dependence of the random variables for the viscosity as below:\vspace{-4mm}
 \begin{align}
\nu(\bx,\by)&=\nu_0(\bx)+\sum_{l=1}^N\nu_l(\bx)y_l. \end{align}

In this work, we consider the following set of $J$ (number of realizations, or the number of stochastic collocation points) time-dependent, viscoresistive and incompressible dimensionless NSEs for computing the flow ensemble simulation of homogeneous Newtonian fluids:
\begin{eqnarray}
\bu_{j,t}+\bu_j\cdot\nabla \bu_j-\nabla\cdot\left(\nu_j(\bx) \nabla \bu_j\right)+\nabla p_j &= & \bif_j(\bx,t), \hspace{2mm}\text{in}\hspace{2mm}\cD \times (0,T], \label{gov1}\\
\nabla\cdot \bu_j & =& 0, \hspace{11mm}\text{in}\hspace{2mm}\cD \times (0,T],\label{gov2}
\end{eqnarray}
where $\bu_j$, and $p_j$, denote the velocity, and pressure solutions, respectively, for each $j=1,2,\cdots,J$, corresponding to distinct kinematic viscosity $\nu_j$, and/or body forces $\bif_j$,  and/or the initial conditions $\bu_j^0$, and/or the $J$ different boundary conditions. For simplicity of our analysis, we prescribed Dirichlet boundary conditions $\bu_j(\bx,t)=\textbf{0}$  for the $j^{th}$ realization. It is assumed that $\nu_j(\bx)\in L^\infty(\cD)$, and $\nu_j(\bx)\ge\nu_{j,\min}>0$, where $\nu_{j,\min}=\min\limits_{\bx\in\cD}\nu_j(\bx)$.

Input data, e.g. initial and boundary conditions, viscosities, body forces, etc. have a significant effect on simulations of complex dynamical systems. The involvement of uncertainty in their measurements reduces the fidelity of the final solutions. For a robust and high fidelity solution, computation of ensemble average solution is popular in many applications such as surface data assimilation \cite{fujita2007surface}, magnetohydrodynamics \cite{jiang2018efficient}, porous media flow \cite{jiang2021artificial}, weather forecasting \cite{L05,LP08}, spectral methods \cite{LK10}, sensitivity analyses \cite{MX06}, and hydrology \cite{GG11}.

To reduce the immense computational cost for the above ensemble system, we propose a decoupled (penalty-projection based splitting \cite{AKMR15,erkmen2020second, linke2017connection, mohebujjaman2024efficient}) scheme together with the breakthrough idea presented in \cite{JL14}. Thus, we consider a uniform time-step size $\Delta t$ and let $t_n=n\Delta t$ for $n=0, 1, \cdots$., (suppress the spatial discretization momentarily), then computing the $J$ solutions independently, takes the following form: For $j$=1,...,$J$,\\
Step 1: Find $\bhu_j^{n+1}$:
\begin{align}
\frac{\bhu_j^{n+1}}{\Delta t}+<\bhu>^n\cdot\nabla \bhu_j^{n+1}-\nabla\cdot\left(\Bar{\nu}\nabla \bhu_j^{n+1}\right)&-\gamma\nabla(\nabla\cdot\bhu_j^{n+1})-\nabla\cdot\left( 2\nu_T(\bhu^{'}_{h},t^n)\nabla \bhu_{j,h}^{n+1}\right)\nonumber\\&= \bif_{j}(t^{n+1})+\frac{\tilde{\bu}_j^n}{\Delta t}-\bhu_j^{'n}\cdot\nabla \bhu_j^n+\nabla\cdot\left(\nu_j^{'}\nabla \bhu_j^{n}\right),\label{scheme1}\\\bhu_j^{n+1}\big|_{\partial\cD}&=0.\label{incom1}
\end{align}
Step 2: Find $\tilde{\bu}_j^{n+1}$, and $\hp_j^{n+1}$:
\begin{align}
    \frac{\tilde{\bu}_j^{n+1}}{\Delta t}+\nabla \hp_j^{n+1}&=\frac{\bhu^{n+1}_j}{\Delta t},\\
    \nabla\cdot\tilde{u}_j^{n+1}&=0,\\
\tilde{u}_j^{n+1}\cdot\bnh\big|_{\partial\cD}&=0.\label{normal-eqn}
\end{align}
Where $\bhu_j^n$, and $\hp_j^{n}$ denote approximations of $\bu_j(\cdot,t^n)$, and $p_j(\cdot,t^n)$, respectively, and the grad-div stabilization parameter $\gamma>0$. The ensemble mean and fluctuation about the mean are defined as follows:
\begin{align*}
<\bhu>^n:&=\frac{1}{J}\sum\limits_{j=1}^{J}\bhu_j^n, \hspace{2mm} \bhu_j^{'n}:=\bhu_j^n-<\bhu>^n,\\ \Bar{\nu}:&=\frac{1}{J}\sum\limits_{j=1}^{J}\nu_j,\hspace{4mm}\nu_j^{'}:=\nu_j-\Bar{\nu}.   
\end{align*}
The EEV term, which is $O(\Delta t)$, is defined using mixing length phenomenology, following \cite{jiang2015numerical},  and is given by \begin{align} \nu_T(\bhu^{'}\hspace{-1mm},t^n):=\mu\Delta t(\hl^{n})^2,\hspace{0.5mm}\text{ where}\hspace{1mm}(\hl^{n})^2=\sum_{j=1}^J|\bhu_j^{'n}|^2 \label{eddy-viscosity}
\end{align} is a scalar quantity, and $|\cdot|$ denotes length of a vector. The EEV term helps to reduce the numerical instability of the model for convection-dominated flows that are not resolved for some particular meshes. The idea of EEV is taken from turbulence modeling and is widely used \cite{jiang2015higher,jiang2015analysis,jiang2015numerical,Mohebujjaman2022High, mohebujjaman2024efficient, MR17,mohebujjaman2022efficient}. However, the analysis of the algorithms equipped with EEV is scarce.  
At each time-step, the above identical sub-problems can be solved simultaneously and  each of them shares a system matrix (which is independent of $j$) with all the $J$ realizations and solves the following system of equations $A[\bx_1|\bx_2|\cdots|\bx_J]=[\bb_1|\bb_2|\cdots|\bb_J]$.

Therefore, a massive computer memory is saved and the global system matrix assembly, factorization (if a direct solver is used), and preconditioner building (in particular for the iterative solver) are needed only once per time step. Moreover, the algorithm can take advantage of the block linear solvers. This idea in \cite{JL14} has been implemented for the solution of the heat equation with uncertain temperature-dependent conductivity \cite{Unconditionally2021Fiordilino}, NSEs simulations \cite{jiang2015higher,jiang2017second,jiang2015numerical,neda2016ensemble}, magnetohydrodynamics (MHD) \cite{jiang2018efficient,MR17,Mohebujjaman2022High,mohebujjaman2022efficient,mohebujjaman2024efficient}, parameterized flow problems \cite{GJW18, GJW17}, and turbulence modeling \cite{JKL15}.  An equivalent efficient and accurate penalty-projection EEV algorithm for the stochastic MHD flow is analyzed and test in \cite{mohebujjaman2024efficient}. Since the NSEs are the basis for simulating flows of computational modeling, and other multi-physics flow problems, it is important propose and analyze a robust, efficient and accurate penalty-projection algorithm for the SNSEs. 

Thus, using a finite element spatial discretization, we investigate the new decoupled ensemble scheme in a fully discrete setting. The efficient ensemble scheme is stable and convergent.  The rest of the report is organized as follows: To follow a smooth analysis, we provide necessary notations and mathematical preliminaries in Section \ref{notation-prelims}. In Section \ref{fully-discrete-scheme}, we present and analyze a fully discrete, linearized, efficient, and EEV based coupled algorithm corresponding to \eqref{momentum}-\eqref{nse-initial}, and provide it's stability and convergent theorems rigorously. We then propose a more efficient EEV and penalty-projection based ensemble algorithm in Section \ref{penalty-projection}, where we prove the stability and convergence theorems of the proposed penalty-projection algorithm. We show that for a large penalty parameter, the penalty-projection scheme converges to the coupled scheme proposed in Section \ref{fully-discrete-scheme}. In Section \ref{scm}, we provide a brief introduction of the SCMs. A series of numerical experiments are given in Section \ref{numerical-experiment}, which support the theory, combine SCMs with the proposed schemes and implement them on several benchmark problems.

\section{Notation and preliminaries}\label{notation-prelims}

Let $\cD\subset \mathbb{R}^d\ (d=2,3)$ be a convex polygonal or polyhedral domain in $\mathbb{R}^d(d=2,3)$ with boundary $\partial\cD$. The usual $L^2(\cD)$ norm and inner product are denoted by $\|.\|$ and $(.,.)$, respectively. Similarly, the $L^p(\cD)$ norms and the Sobolev $W_p^k(\cD)$ norms are $\|.\|_{L^p}$ and $\|.\|_{W_p^k}$, respectively for $k\in\mathbb{N},\hspace{1mm}1\le p\le \infty$. Sobolev space $W_2^k(\cD)$ is represented by $H^k(\cD)$ with norm $\|.\|_k$. The vector-valued spaces are $$\bL^p(\cD)=(L^p(\cD))^d, \hspace{1mm}\text{and}\hspace{1mm}\bH^k(\cD)=(H^k(\cD))^d.$$
For $\bX$ being a normed function space in $\cD$, $L^p(0,T;\bX)$ is the space of all functions defined on $(0,T]\times\cD$ for which the following norm
\begin{align*}
    \|\bu\|_{L^p(0,T;\bX)}=\lp\int_0^T\|\bu\|_{\bX}^pdt\rp^\frac{1}{p},\hspace{2mm}p\in[1,\infty)
\end{align*}
is finite. For $p=\infty$, the usual modification is used in the definition of this space. The natural function spaces for our problem are
\begin{align*}
    \bX:&=\bH_0^1(\cD)=\{\bv\in \bL^2(\cD) :\nabla \bv\in L^2(\cD)^{d\times d}, \bv=0 \hspace{2mm} \mbox{on}\hspace{2mm}   \partial \cD\},\\
    \bY:&=\{\bv\in \bL^2(\cD),\bv\cdot\hat{n}\big|_{\partial\cD}=0\},\\
    Q:&=L_0^2(\cD)=\{ q\in L^2(\cD): \int_\cD q\hspace{1mm}d\bx=0\}.
\end{align*}
Recall the Poincare inequality holds in $X$: There exists $C$ depending only on $\cD$ satisfying for all $\bphi\in X$,
\[
\| \bphi \| \le C \| \nabla \bphi \|.
\]
The divergence free velocity space is given by
$$\bV:=\{\bv\in \bX:(\nabla\cdot \bv, q)=0, \forall q\in Q\}.$$
We define the skew symmetric trilinear form $b^*:\bX\times \bX\times \bX\rightarrow \mathbb{R}$ by
	\[
	b^*(\bu,\bv,\bw):=\frac12(\bu\cdot\nabla \bv,\bw)-\frac12(\bu\cdot\nabla \bw,\bv). 
	\]
 By the divergence theorem \cite{jiang2015higher}, it can be shown \begin{align}
     b^*(\bu,\bv,\bw)=(\bu\cdot\nabla \bv,\bw)+\frac12(\nabla\cdot\bu,\bv\cdot\bw).\label{trilinear-identitiy}
 \end{align}Recall from \cite{L08, lee2011error, linke2017connection} that for any $\bu,\bv,\bw\in 
		\bX$
	\begin{align}
		b^*(\bu,\bv,\bw)&\leq C(\cD)\|\nabla \bu\|\|\nabla \bv\|\|\nabla \bw\|,\label{nonlinearbound}
	\end{align}	
 and additionally, if $\bv\in \bL^\infty(\Omega)$, and $\nabla\bv\in\bL^3(\Omega)$, then 
 \begin{align}
    b^*(\bu,\bv,\bw)\leq C(\cD)\|\bu\|\left(\|\nabla\bv\|_{L^3}+\|\bv\|_{L^\infty}\right)\|\nabla\bw\|. \label{nonlinearbound3}
 \end{align}The following basic inequalities will be used:
\begin{align}
\|\bu\cdot\nabla\bv\|&\le\||\bu|\nabla\bv\|,\label{basic-ineq}\\
\|\nabla\cdot\bu\|_{L^\infty}&\le C\|\nabla\bu\|_{L^\infty}.\label{basic-ineq-infinity}
\end{align}
The conforming finite element spaces are denoted by $\bX_h\subset \bX$ and  $Q_h\subset Q$, and we assume a regular triangulation $\tau_h(\cD)$, where $h$ is the maximum triangle diameter.   We assume that $(\bX_h,Q_h)$ satisfies the usual discrete inf-sup condition
\begin{eqnarray}
\inf_{q_h\in Q_h}\sup_{\bv_h\in \bX_h}\frac{(q_h,\grad\cdot \bv_h)}{\|q_h\|\|\grad \bv_h\|}\geq\beta>0,\label{infsup}
\end{eqnarray}
where $\beta$ is independent of $h$. We assume that there exists a finite element space  $\bY_h\subset\bY$. The space of discretely divergence free functions is defined as 
\begin{align*}
    \bV_h:=\{\bv_h\in \bX_h:(\nabla\cdot \bv_h,q_h)=0,\hspace{2mm}\forall q_h\in Q_h\}.
\end{align*}
For simplicity of our analysis, we will use Scott-Vogelius (SV) finite element pair $(\bX_h, Q_h)=((P_k)^d, P_{k-1}^{disc})$,  which satisfies the \textit{inf-sup} condition when the mesh is created as a barycenter refinement of a regular mesh, and the polynomial degree $k\ge d$  \cite{arnold1992quadratic,Z05}. Our analysis can be extended without difficulty to any inf-sup stable element choice, {\color{black} however, there will be additional terms that appear in the convergence analysis if non-divergence-free elements are chosen.}

We have the following approximation properties in $(\bX_h,Q_h)$: \cite{BS08}
\begin{align}
\inf_{\bv_h\in \bX_h}\|\bu-\bv_h\|&\leq Ch^{k+1}|\bu|_{k+1},\hspace{2mm}\bu\in \bH^{k+1}(\cD),\label{AppPro1}\\
 \inf_{\bv_h\in \bX_h}\|\grad (\bu-\bv_h)\|&\leq Ch^{k}|\bu|_{k+1},\hspace{5mm}\bu\in \bH^{k+1}(\cD),\label{AppPro2}\\
\inf_{q_h\in Q_h}\|p-q_h\|&\leq Ch^k|p|_k,\hspace{10mm}p\in H^k(\cD),
\end{align}
where $|\cdot|_r$ denotes the $H^r$ or $\bH^r$ seminorm.

We will assume the mesh is sufficiently regular for the inverse inequality to hold. The following lemma for the discrete Gr\"onwall inequality was given in \cite{HR90}.
\begin{lemma}\label{dgl}
		Let $\Delta t$, $\mathcal{E}$, $a_n$, $b_n$, $c_n$, $d_n$ be non-negative numbers for $n=1,\cdots, M$ such that
		$$a_M+\Delta t \sum_{n=1}^Mb_n\leq \Delta t\sum_{n=1}^{M-1}{d_na_n}+\Delta 
		t\sum_{n=1}^Mc_n+\mathcal{E}\hspace{3mm}\mbox{for}\hspace{2mm}M\in\mathbb{N},$$
		then for all $\Delta t> 0,$
		$$a_M+\Delta t\sum_{n=1}^Mb_n\leq \mbox{exp}\left(\Delta t\sum_{n=1}^{M-1}d_n\right)\lp\Delta 
		t\sum_{n=1}^Mc_n+\mathcal{E}\rp\hspace{2mm}\mbox{for}\hspace{2mm}M\in\mathbb{N}.$$
	\end{lemma}
 \section{Efficient Coupled EEV (Coupled-EEV) algorithm for SNSEs}\label{fully-discrete-scheme}
In this section, we propose a velocity and pressure coupled, fully discrete, efficient, linear extrapolated, EEV stabilized, backward-Euler finite element timestepping algorithm for the parameterized SNSEs. The algorithm is efficient because it is presented in a way so that at each time-step, for all the realizations, the system matrix remains the same but with different right-hand-side vectors. Therefore, it allows to save a huge computational time and computer memory. The Coupled-EEV scheme is presented in Algorithm \ref{coupled-alg}.  We provide the stability and convergence theorems of the Coupled-EEV scheme and the proofs of these theorems are given in Appendix \ref{appendix-stability}-\ref{appendix-convergence}.\\
\begin{algorithm}[H]\label{coupled-alg}
  \caption{Coupled-EEV scheme} Given time-step $\Delta t>0$, end time $T>0$, initial conditions $\bu_j^0\in  \bH^1(\cD)$ and $\bif_{j}\in$ $ L^2\left( 0,T;\bH^{-1}(\cD)\right)$ for $j=1,2,\cdots\hspace{-0.35mm},J$. Set $M=T/\Delta t$ and for $n=1,\cdots\hspace{-0.35mm},M-1$, compute: Find $(\bu_{j,h}^{n+1}, p_{j,h}^{n+1})\in \bX_h\times Q_h$ satisfying, for all $(\bchi_{h},q_{h})\in \bX_h\times Q_h$:
 \begin{align}
&\Big(\frac{\bu_{j,h}^{n+1}-\bu_{j,h}^n}{\Delta t},\bchi_{h}\Big)+b^*\big(\hspace{-1mm}<\bu_h>^n, \bu_{j,h}^{n+1},\bchi_{h}\big)+\big(\Bar{\nu}\nabla \bu_{j,h}^{n+1},\nabla\bchi_{h}\big)-(p_{j,h}^{n+1},\nabla\cdot\bchi_{h})\nonumber\\&+\left( 2\nu_T(\bu^{'}_{h},t^n)\nabla \bu_{j,h}^{n+1},\nabla\bchi_h\right)= \big(\bif_{j}(t^{n+1}),\bchi_{h}\big)-b^*(\bu_{j,h}^{'n}, \bu_{j,h}^n,\bchi_{h})-\big(\nu_j^{'}\nabla \bu_{j,h}^{n},\nabla\bchi_{h}\big),
\label{couple-eqn-1}\\\nonumber\\&\big(\nabla\cdot\bu_{j,h}^{n+1},q_{h}\big)=0,\label{couple-incompressibility}
\end{align}
\end{algorithm}
\noindent where $\bu_j^n$, and $p_j^{n}$ denote approximations of $\bu_j(\cdot,t^n)$, and $p_j(\cdot,t^n)$, respectively, and EEV is defined as \begin{align} \nu_T(\bu_h^{'},t^n):=\mu\Delta t(l^{n})^2,\hspace{0.5mm}\text{ where}\hspace{1mm}(l^{n})^2=\sum_{j=1}^J|\bu_j^{'n}|^2. \label{eddy-viscosity2}
\end{align} To simplify the notation, denote $\alpha_j:=\Bar{\nu}_{\min}-\|\nu_j^{'}\|_\infty$, for $j=1,2,\cdots\hspace{-0.35mm}, J$, where $\Bar{\nu}_{\min}:=\min\limits_{\bx\in\cD}\Bar{\nu}(\bx)$. We assume that the data does not have outlier and observations are close enough to the mean so that $\alpha_j>0$ holds. 
\begin{theorem}[Stability]
Suppose $\bif_{\hspace{0.5mm}j}\in L^2(0,T;\bH^{-1}(\cD))$, and $\bu_{j,h}^0\in\bH^1(\cD)$ for all $j=1,2,
\cdots, J$, then the solutions of Algorithm \ref{coupled-alg} are stable: Given $\alpha_j>0$ and $\mu>\frac12$, if $\Delta t<\frac{C\alpha_j}{\|\nabla\cdot\bu_{j,h}^{'n}\|^2_{L^\infty}}$ then 
\begin{align}
\|\bu_{j,h}^{M}\|^2+\alpha_j\Delta t\sum_{n=0}^{M}\|\nabla \bu_{j,h}^{n}\|^2\le \|\bu_{j,h}^{0}\|^2+\Bar{\nu}_{\min}\Delta t\|\nabla \bu_{j,h}^{0}\|^2+\frac{2\Delta t}{\alpha_j}\sum_{n=1}^{M}\|\bif_{\hspace{0.5mm}j}(t^{n})\|_{-1}^2.\label{stability-couple-alg}
\end{align}
\end{theorem}
\begin{proof}
	See the Apendix \ref{appendix-stability}.
\end{proof}
\begin{theorem}[Convergence] Suppose $(\bu_j,p_j)$ satisfying \eqref{momentum}-\eqref{nse-initial} and the following regularity assumptions for $m=\max\{3,k+1\}$
\begin{align*}
   \bu_j&\in L^2(0,T;\bH^{m}(\cD)^d)\cap L^\infty(0,T; \bH^{m}(\cD)^d)\\
   \bu_{j,t}&\in L^\infty(0,T;\bH^2(\cD)^d)\cap L^2(0,T;\bH^1(\cD)^d)\\
   \bu_{j,tt}&\in L^2(0,T;\bL^2(\cD)^d)
\end{align*}
with $k\ge 2$, then the ensemble solution of the Algorithm \ref{coupled-alg} converges to the true ensemble solution: For $\alpha_j>0$ and $\mu>\frac12$, if $\Delta t<\frac{C\alpha_j}{\|\nabla\cdot\bu_{j,h}^{'n}\|^2_{L^\infty}}$ then, the following holds
\begin{align}
    \|<\bu>(T)-<\bu_h>^M\|^2+\alpha_{\min}\Delta t\sum_{n=1}^M\Big\|\nabla\Big(\hspace{-1.1mm}<\bu>(t^n)-<\bu_h>^n\hspace{-1.1mm}\Big)\Big\|^2\nonumber\\\le C\Big(h^{2k}+\Delta t^2+h^{2-d}\Delta t^2+h^{2k-1}\Delta t\Big).\label{convergence-error}
\end{align}
\end{theorem}
\begin{proof}
	See the proof in the Apendix \ref{appendix-convergence}.
\end{proof}

\begin{lemma}\label{lemma1}
Assume the true solution $\bu_j\in L^\infty(0,T;\bH^2(\cD))$. We also assume there exists a constant $C_*$ which is independent of $h$, $\Delta t$, and $\gamma$ such that for sufficiently small $h$ and $\Delta t$, the solution of the Algorithm \ref{coupled-alg} satisfies
\begin{align*}
    \max_{1\le n\le M}\Big(\|\nabla \bu_{j,h}^n\|_{L^3}+\|\bu_{j,h}^n\|_{L^\infty}\Big)&\le C_*,\hspace{2mm}\text{for all}\hspace{2mm}j=1,2,\cdots,J.
\end{align*}
\end{lemma}
\begin{proof}
    Using triangle inequality, we write
\begin{align}
    \|\nabla \bu_{j,h}^n\|_{L^3}+\|\bu_{j,h}^n\|_{L^\infty}&\le \|\nabla(\bu_{j,h}^n-\bu_j(t^n))\|_{L^3}\nonumber\\&+\|\nabla\bu_j(t^n)\|_{L^3}+\|\bu_{j,h}^n-\bu_j(t^n)\|_{L^\infty}+\|\bu_j(t^n)\|_{L^\infty}.\label{lemma-trianlge}
\end{align}
Apply Sobolev embedding theorem on the first, and second terms, and Agmon’s \cite{Robinson2016Three-Dimensional} inequality on the third, and fourth terms in the right-hand-side of \eqref{lemma-trianlge}, to obtain
\begin{align}
    \|\nabla \bu_{j,h}^n\|_{L^3}+\|\bu_{j,h}^n\|_{L^\infty}&\le C\|\nabla(\bu_{j,h}^n-\bu_j(t^n))\|^{\frac12}\|\nabla^2(\bu_{j,h}^n-\bu_j(t^n))\|^{\frac12}\nonumber\\&+C\|\bu_j(t^n)\|_{H^1}^\frac12\|\bu_j(t^n)\|_{H^2}^\frac12.
\end{align}
 Apply the regularity assumption of the true solution and
discrete inverse inequality, to obtain
\begin{align}
    \|\nabla \bu_{j,h}^n\|_{L^3}+\|\bu_{j,h}^n\|_{L^\infty}&\le Ch^{-\frac12}\|\nabla(\bu_{j,h}^n-\bu_j(t^n))\|+C.
\end{align}
Consider the $(P_k,P_{k-1})$ element for the pair $(\bu_{j,h},p_{j,h})$, and use the error bounds in \eqref{error-bounds-j-level}, gives 
\begin{align*}
    \|\nabla \bu_{j,h}^n\|_{L^3}+\|\bu_{j,h}^n\|_\infty&\le Ch^{-\frac12}\left(\frac{h^k}{\Delta t^{\frac12}}+\Delta t^{\frac12}+h^{1-\frac{d}{2}}\Delta t^{\frac12}+h^{k-\frac12}\right)+C.
\end{align*}
$$C\Big(h^{2k}+\Delta t^2+h^{2-d}\Delta t^2+h^{2k-1}\Delta t\Big).$$
Choose $\Delta t$ so that
\begin{align*}
    \frac{h^{k-\frac12}}{\Delta t^\frac12}\le\frac{1}{C},\;
    \frac{\Delta t^\frac12}{h^\frac12}\le \frac{1}{C},\;\text{and}\;
    h^{\frac{1-d}{2}}\Delta t^\frac12\le\frac{1}{C},
\end{align*}
which gives
\begin{align*}
    \|\nabla \bu_{j,h}^n\|_{L^3}+\|\bu_{j,h}^n\|_\infty\le 3+C,
\end{align*}
with time-step restriction $O(h^{2k-1})\le\Delta t\le O(h^{d-1})$. Therefore, for $C_*:=3+C$, completes the proof.
\end{proof}

\section{Efficient Stabilized Penalty Projection EEV (SPP-EEV) algorithm for SNSEs}
\label{penalty-projection}
Now, we present and analyze a more efficient, fully discrete, and decoupled penalty-projection time stepping scheme with grad-div and EEV stabilization for computing NSE flow ensemble. The splitting error of the algorithm diminishes for large grad-div stabilization parameter values. We then connect the SPP-EEV scheme to the Coupled-EEV scheme by showing that for the large penalty parameters, the outcomes of the SPP-EEV converge to the Coupled-EEV scheme's outcomes. In the SPP-EEV scheme, we avoid solving a difficult saddle-point problem at each time-step by using two steps, where we require two easier linear solves. In Step 1, we solve a $1\times 1$ block system for the velocity with the Dirichlet boundary condition but without satisfying the incompressibility condition. In Step 2, we solve a $2\times 2$ saddle-point system (without satisfying the original boundary condition) which requires an easier linear solve since the non-linear term is absent, and provides symmetric positive definite system matrices at each time-step. Moreover, each of the steps in SPP-EEV scheme is designed technically so that at each time-step, the system matrix remains the same for all the realizations but with different right-hand-side vectors. Thus, in both steps, the advantage of reusing the matrix factorization or the block linear solvers can be taken. Therefore, together with all these features, the SPP-EEV is supposed to be an efficient and accurate ensemble scheme for the uncertainty quantification of SNSEs flows. The SPP-EEV scheme is given in Algorithm \ref{SPP-FEM}.

\begin{algorithm}[H]\label{SPP-FEM}
  \caption{SPP-EEV scheme} Given time-step $\Delta t>0$, end time $T>0$, initial conditions $\bhu_j^0=\tilde{\bu}_j^0\in \bY_h\cap \bH^2(\cD)$ and $\bif_{j}\in$ $ L^2\left( 0,T;\bH^{-1}(\cD)\right)$ for $j=1,2,\cdots\hspace{-0.35mm},J$. Set $M=T/\Delta t$ and for $n=1,\cdots\hspace{-0.35mm},M-1$, compute:\\
 Step 1: Find $\bhu_{j,h}^{n+1}\in \bX_h$ satisfying for all $\bchi_{h}\in \bX_h$:
 \begin{align}
&\Big(\frac{\bhu_{j,h}^{n+1}-\tilde{\bu}_{j,h}^n}{\Delta t},\bchi_{h}\Big)+b^*\big(\hspace{-1mm}<\hspace{-1mm}\bhu_h\hspace{-1mm}>^n, \bhu_{j,h}^{n+1},\bchi_{h}\big)+\big(\Bar{\nu}\nabla \bhu_{j,h}^{n+1},\nabla\bchi_{h}\big)+\gamma\big(\nabla\cdot\bhu_{j,h}^{n+1},\nabla\cdot \bchi_{h}\big)\nonumber\\&+\Big(2\nu_T(\hat{\bu}^{'}_{h},t^n)\nabla \bhu_{j,h}^{n+1},\nabla \bchi_{h}\Big)= \big(\bif_{j}(t^{n+1}),\bchi_{h}\big)-b^*\big(\bhu_{j,h}^{'n}, \bhu_{j,h},\bchi_{h}\big)-\big(\nu_j^{'}\nabla \bhu_{j,h}^{n},\nabla\bchi_{h}\big).\label{spp-step-1}
\end{align}
Step 2: Find $\left(\tilde{\bu}_{j,h}^{n+1},\hp_{j,h}^{n+1}\right)\in \bY_h\times Q_h$ satisfying for all $(\bv_{h},q_{h})\in \bY_h\times Q_h$:
\begin{align}
    \Big(\frac{\tilde{\bu}_{j,h}^{n+1}-\bhu_{j,h}^{n+1}}{\Delta t},\bv_{h}\Big)-\big( \hp_{j,h}^{n+1},\nabla\cdot\bv_{h}\big)&=0,\label{spp-step-2-1}\\
    \big(\nabla\cdot\tilde{\bu}_{j,h}^{n+1},q_{h}\big)&=0.\label{spp-step-2-2}
\end{align}
\end{algorithm}
 Since $\bX_h\subset\bY_h$, we can choose $\bv_{h}=\bchi_{h}$ in \eqref{spp-step-2-1}, and combine them with equation \eqref{spp-step-1}, to get

\begin{align}
\Big(\frac{\bhu_{j,h}^{n+1}-\bhu_{j,h}^n}{\Delta t}, &\;\bchi_{h}\Big)+b^*\big(\hspace{-1mm}<\bhu_h>^n, \bhu_{j,h}^{n+1},\bchi_{h}\big)+\big(\Bar{\nu}\nabla \bhu_{j,h}^{n+1},\nabla \bchi_{h}\big)+\gamma\big(\nabla\cdot\bhu_{j,h}^{n+1},\nabla\cdot\bchi_{h}\big)\nonumber\\&+\Big(2\nu_T(\hat{\bu}^{'}_{h},t^n)\nabla \bhu_{j,h}^{n+1},\nabla \bchi_{h}\Big)-\big(\hp_{j,h}^{n},\nabla\cdot\bchi_{h}\big)= \big(\bif_{j}(t^{n+1}),\bchi_{h}\big)\nonumber\\&-b^*\big(\bhu_{j,h}^{'n}, \bhu_{j,h}^n,\bchi_{h}\big)-\big(\nu_j^{'}\nabla \bhu_{j,h}^{n},\nabla\bchi_{h}\big).\label{hatweak1}
\end{align}

\subsection{Stability Analysis}\label{stability-analysis}
We now prove stability and well-posedness for the Algorithm \ref{SPP-FEM}. 

\begin{lemma}(Unconditional Stability)
Let $\big(\bhu_{j,h}^{n+1},\hp_{j,h}^{n+1}\big)$ be the solution of Algorithm \ref{SPP-FEM} and $\textbf{f}_{j}\in L^2\left(0,T;\bH^{-1}(\cD)\right)$, and $\bhu_{j,h}^0\in \bH^1(\cD)$. Then for all $\Delta t>0$, if $\alpha_j>\frac{C}{h}>0$, and $\mu>\frac{\alpha_jh^2}{2\Delta t}$, we have the following stability bound:
\begin{align}
    \|\bhu_{j,h}^{M}\|^2+\left(\alpha_j-\frac{C}{\alpha_jh^2}\right)\Delta t\sum_{n=1}^{M}\|\nabla\bhu_{j,h}^n\|^2+2\gamma\Delta t\sum_{n=1}^{M}\|\nabla\cdot\bhu_{j,h}^{n}\|^2\nonumber\\\le\|\bhu_{j,h}^0\|^2+\Bar{\nu}_{\min}\Delta t\|\nabla \bhu_{j,h}^{0}\|^2+\frac{\Delta t}{\alpha_j}\sum_{n=1}^{M}\|\bif_{j}(t^{n})\|_{-1}^2.
\end{align}
\end{lemma}

\begin{proof}
Taking $\bchi_{h}=\bhu_{j,h}^{n+1}$ in \eqref{spp-step-1}, to obtain
\begin{align}
    \Bigg(\frac{\bhu_{j,h}^{n+1}-\btu_{j,h}^n}{\Delta t}, \bhu_{j,h}^{n+1}\Bigg)+\|\Bar{\nu}^\frac{1}{2}\nabla \bhu_{j,h}^{n+1}\|^2+\gamma\|\nabla\cdot\bhu_{j,h}^{n+1}\|^2+\left(2\nu_T(\hat{\bu}^{'}_{h},t^n)\nabla \bhu_{j,h}^{n+1},\nabla \bhu_{j,h}^{n+1}\right)\nonumber\\=\Big(\bif_{j}(t^{n+1}),\bhu_{j,h}^{n+1}\Big)-b^*\Big(\bhu_{j,h}^{'n}, \bhu_{j,h}^n,\bhu_{j,h}^{n+1}\Big)-\Big(\nu_j^{'}\nabla \bhu_{j,h}^{n},\nabla\bhu_{j,h}^{n+1}\Big).
\end{align}
Using polarization identity and
$(2\nu_T(\bhu^{'}_{h},t^n)\nabla \bhu_{j,h}^{n+1},\nabla \bhu_{j,h}^{n+1})=2\mu\Delta t\|\hl^{n}\nabla \bu_{j,h}^{n+1}\|^2$, we get 
\begin{align}
    \frac{1}{2\Delta t}\Big(\|\bhu_{j,h}^{n+1}\|^2-\|\btu_{j,h}^n\|^2+\|\bhu_{j,h}^{n+1}-\btu_{j,h}^n\|^2\Big)+\|\Bar{\nu}^\frac{1}{2}\nabla \bhu_{j,h}^{n+1}\|^2+\gamma\|\nabla\cdot\bhu_{j,h}^{n+1}\|^2\nonumber\\+2\mu\Delta t\|\hl^{n}\nabla \bhu_{j,h}^{n+1}\|^2=\Big(\bif_{j}(t^{n+1}),\bhu_{j,h}^{n+1}\Big)-b^*\Big(\bhu_{j,h}^{'n}, \bhu_{j,h}^n,\bhu_{j,h}^{n+1}\Big)-\Big(\nu_j^{'}\nabla \bhu_{j,h}^{n},\nabla\bhu_{j,h}^{n+1}\Big).\label{pol11}
\end{align}
Applying Cauchy-Schwarz and Young's inequalities on the forcing term, yields
  \begin{align*}
      (\bif_{j}(t^{n+1}),\bhu_{j,h}^{n+1})\le\|\bif_{j}(t^{n+1})\|_{-1}\|\nabla\bhu_{j,h}^{n+1}\|\le \frac{\alpha_j}{4}\|\nabla\bhu_{j,h}^{n+1}\|^2+\frac{1}{\alpha_j}\|\bif_{j}(t^{n+1})\|_{-1}^2.
  \end{align*}
We rewrite the trilinear form in \eqref{pol11}, use identity \eqref{trilinear-identitiy}, Cauchy-Schwarz, H\"older's, Poincar\'e, and \eqref{basic-ineq}-\eqref{basic-ineq-infinity} inequalities, to have
\begin{align}
    -b^*\left(\bhu_{j,h}^{'n}, \bhu_{j,h}^n,\bhu_{j,h}^{n+1}\right)&=b^*\left(\bhu_{j,h}^{'n}, \bhu_{j,h}^{n+1},\bhu_{j,h}^n\right)=\left(\bhu_{j,h}^{'n}\cdot\nabla \bhu_{j,h}^{n+1}, \bhu_{j,h}^n\right)+\frac12\left(\nabla\cdot\bhu_{j,h}^{'n},\bhu_{j,h}^{n+1}\cdot \bhu_{j,h}^n\right)\nonumber\\&\le\|\bhu_{j,h}^{'n}\cdot\nabla
\bhu_{j,h}^{n+1}\|\|\bhu_{j,h}^n\|+\frac12\|\nabla\cdot\bhu_{j,h}^{'n}\|_{L^\infty}\|\bhu_{j,h}^{n+1}\|\|\bhu_{j,h}^n\|\nonumber\\&\le C\||\bhu_{j,h}^{'n}|\nabla
\bhu_{j,h}^{n+1}\|\|\nabla\bhu_{j,h}^n\|+C\|\nabla\bhu_{j,h}^{'n}\|_{L^\infty}\|\nabla\bhu_{j,h}^{n+1}\|\|\nabla\bhu_{j,h}^n\|. \label{fluc-bound-first-hat}
\end{align}
Using \eqref{eddy-viscosity}, Young's, discrete inverse inequalities, and Assumption \ref{assump1} in \eqref{fluc-bound-first-hat}, gives
  \begin{align}
      -b^*\Big(\bhu_{j,h}^{'n}, 
			\bhu_{j,h}^n,\bhu_{j,h}^{n+1}\Big)
   &\le\frac{\alpha_j}{4}\|\nabla\bhu_{j,h}^{n+1}\|^2+C\|\hl^{n}\nabla 			\bhu_{j,h}^{n+1}\|\|\nabla\bhu_{j,h}^n\|+\frac{C}{\alpha_jh^2}\|\bhu_{j,h}^{'n}\|_{L^\infty}^2\|\nabla\bhu_{j,h}^{n}\|^2\nonumber\\&\le\frac{\alpha_j}{4}\|\nabla\bhu_{j,h}^{n+1}\|^2+C\|\hl^{n}\nabla 			\bhu_{j,h}^{n+1}\|\|\nabla\bhu_{j,h}^n\|+\frac{CC_*}{\alpha_jh^2}\|\nabla\bhu_{j,h}^{n}\|^2\nonumber\\
   &\le\frac{\alpha_j}{4}\|\nabla\bhu_{j,h}^{n+1}\|^2+\alpha_jh^2\|\hl^{n}\nabla 
\bhu_{j,h}^{n+1}\|^2+\frac{C}{\alpha_jh^2}\|\nabla\bhu_{j,h}^{n}\|^2.\label{fluc-bound-2}
  \end{align}
  Use of H\"older's and Young's inequalities, we have
\begin{align*}
    -(\nu_j^{'}\nabla \bhu_{j,h}^{n},\nabla\bhu_{j,h}^{n+1})\le\|\nu_j^{'}\|_{\infty}\|\nabla \bhu_{j,h}^{n}\|\|\nabla\bhu_{j,h}^{n+1}\|\le\frac{\|\nu_j^{'}\|_{\infty}}{2}\|\nabla \bhu_{j,h}^{n}\|^2+\frac{\|\nu_j^{'}\|_{\infty}}{2}\|\nabla \bhu_{j,h}^{n+1}\|^2.
\end{align*}
Define $\Bar{\nu}_{\min}:=\min\limits_{\bx\in\cD}\Bar{\nu}(\bx)$, using the above bounds, and reducing the equation \eqref{pol11}, becomes
\begin{align}
    \frac{1}{2\Delta t}\Big(\|\bhu_{j,h}^{n+1}\|^2-\|\btu_{j,h}^n\|^2+\|\bhu_{j,h}^{n+1}-\btu_{j,h}^n\|^2\Big)+\frac{\nu_{\min}}{2}\|\nabla \bhu_{j,h}^{n+1}\|^2+\gamma\|\nabla\cdot\bhu_{j,h}^{n+1}\|^2\nonumber\\+\left(2\mu\Delta t-\alpha_jh^2\right)\|\hl^{n}\nabla \bhu_{j,h}^{n+1}\|^2\le\left(\frac{\|\nu_j^{'}\|_{\infty}}{2}+\frac{C}{\alpha_jh^2}\right)\|\nabla\bhu_{j,h}^n\|^2+\frac{1}{\alpha_j}\|\bif_{j}(t^{n+1})\|_{-1}^2.
\end{align}
Choose $\mu>\frac{\alpha_jh^2}{2\Delta t}$, drop non-negative terms from left, and rearrange
\begin{align}
    \frac{1}{2\Delta t}\Big(\|\bhu_{j,h}^{n+1}\|^2-\|\btu_{j,h}^n\|^2\Big)+\frac{\Bar{\nu}_{\min}}{2}\left(\|\nabla \bhu_{j,h}^{n+1}\|^2-\|\nabla \bhu_{j,h}^{n}\|^2\right)\nonumber\\+\left(\frac{\alpha_j}{2}-\frac{C}{\alpha_jh^2}\right)\|\nabla\bhu_{j,h}^n\|^2+\gamma\|\nabla\cdot\bhu_{j,h}^{n+1}\|^2\le \frac{1}{2\alpha_j}\|\bif_{j}(t^{n+1})\|_{-1}^2.\label{estimate-b4-tild}
\end{align}
Now choose $\bv_{h}=\btu_{j,h}^{n+1}$ in \eqref{spp-step-2-1}, and $q_{h}=\hp_{j,h}^{n+1}$ in \eqref{spp-step-2-2}, then apply Cauchy-Schwarz and Young’s inequalities, to obtain
\begin{align*}
    \|\btu_{j,h}^{n+1}\|^2\le\|\bhu_{j,h}^{n+1}\|^2,
\end{align*}
for all $n=0,1,2,\cdots,M-1$. Plugging this estimate into \eqref{estimate-b4-tild}, results in 
\begin{align}
    \frac{1}{2\Delta t}\Big(\|\bhu_{j,h}^{n+1}\|^2-\|\bhu_{j,h}^n\|^2\Big)+\frac{\Bar{\nu}_{\min}}{2}\Big(\|\nabla \bhu_{j,h}^{n+1}\|^2-\|\nabla \bhu_{j,h}^{n}\|^2\Big)\nonumber\\+\left(\frac{\alpha_j}{2}-\frac{C}{\alpha_jh^2}\right)\|\nabla\bhu_{j,h}^n\|^2+\gamma\|\nabla\cdot\bhu_{j,h}^{n+1}\|^2\le \frac{1}{2\alpha_j}\|\bif_{j}(t^{n+1})\|_{-1}^2.
\end{align}
Multiplying both sides by $2\Delta t$, summing over the time steps $n=0,1,2,\cdots,M-1$, and assuming $\alpha_j>\frac{C}{h}>0$ for all $j=1,2,\cdots,J$, completes the proof.
\end{proof}

We now prove the penalty projection based Algorithm \ref{SPP-FEM} converges to the coupled Algorithm \ref{coupled-alg} as $\gamma\rightarrow\infty$. Thus, we need to define the space $\bR_h:=\bV_h^\perp\subset\bX_h$ to be the orthogonal complement of $\bV_h$ with respect to the $\bH^1(\cD)$ norm.

\begin{lemma}\label{CR-lemma}
Let the finite element pair $(\bX_h,Q_h)\subset(\bX,Q)$ satisfy the \textit{inf-sup condition} \eqref{infsup} and the divergence-free property, i.e., $\nabla\cdot\bX_h\subset Q_h$. Then there exists a constant $C_R$ independent of $h$ such that $$\|\nabla\bv_h\|\le C_R\|\nabla\cdot\bv_h\|,\hspace{3mm}\forall\bv_h\in \bR_h.$$
\end{lemma}
\begin{proof}
See \cite{GR86, linke2017connection}
\end{proof}

\begin{assumption}\label{assump1}
We assume there exists a constant $C_*$ which is independent of $h$, and $\Delta t$, such that for sufficiently small $h$ for a fixed mesh and fixed $\Delta t$ as $\gamma\rightarrow\infty$, the solution of the Algorithm \ref{SPP-FEM} satisfies
\begin{align}
    \max_{1\le n\le M}\|\bhu_{j,h}^n\|_{L^\infty}\le C_*,\hspace{2mm}\text{for all}\hspace{2mm}j=1,2,\cdots,J.
\end{align}
\end{assumption}
The Assumption \ref{assump1} is proved later in Lemma \ref{uniform-boundedness-lemma-proof}. The use of the Assumption \ref{assump1} in the convergence analysis is followed by \cite{mohebujjaman2024decoupled}. We define $\alpha_{\min}:=\min\limits_{1\le j\le J}\alpha_j$.

\begin{theorem}[Convergence]\label{gamma-convergence}
Let $(\bu_{j,h}^{n+1}
,p_{j,h}^{n+1})$, and $(\bhu_{j,h}^{n+1},\tilde{\bu}_{j,h}^{n+1}
,\hp_{j,h}^{n+1})$ are the solutions to the Algorithm \ref{coupled-alg}, and Algorithm \ref{SPP-FEM}, respectively, for $n=0,1,\cdots,M-1$. We then have for a given $\gamma>0$ and $\mu>1$:
\begin{align}
    \Big(\Delta t\sum_{n=1}^M\|\nabla\hspace{-1mm}\lab\bu_h\rab^n&-\nabla\hspace{-1mm}\lab\bhu_h\rab^n\|^2\Big)^{\frac12}\le\frac{C}{\gamma} exp\lp\frac{C}{\alpha_{\min}} \left(1+\frac{\Delta t}{h^3}\right)\rp\lp  \Delta t\sum_{n=0}^{M-1}\sum_{j=1}^J\|p_{j,h}^{n+1}-\hp_{j,h}^n\|^2\rp^\frac12\nonumber\\&\times\Bigg[1+\frac{1}{\alpha_{\min}}exp\left(\frac{C}{\alpha_{\min}h^2}+\frac{C}{\Delta t}\right)\left(\frac{1}{\alpha_{\min}^2\Delta t}+\frac{1}{\Delta t}+\Delta t\right)\Bigg]^\frac12.\label{convergence-thm}
\end{align}
\end{theorem}

\begin{remark}
 The above theorem states the first order convergence of the penalty-projection algorithm to the Algorithm \ref{coupled-alg} as $\gamma\rightarrow\infty$ for a fixed mesh and time-step size.
\end{remark}
\begin{proof}
Denote $\be_{j}^{n+1}:=\bu_{j,h}^{n+1}-\bhu_{j,h}^{n+1}$ and use the following $H^1$-orthogonal decomposition of the error:
$$\be_{j}^{n+1}:=\be_{j,0}^{n+1}+\be_{j,\bR}^{n+1},$$
with $\be_{j,0}^{n+1}\in\bV_h$, and $\be_{j,\bR}^{n+1}\in\bR_h$, for $n=0,1,\cdots,M-1$.\\
\textbf{Step 1:} Estimate of $\be_{j,\bR}^{n+1}$: Subtracting the equation \eqref{hatweak1} from \eqref{couple-eqn-1}, produces to
\begin{align}
    \frac{1}{\Delta t}\Big(\be_j^{n+1}&-\be_j^n,\bchi_{h}\Big)+\left(\Bar{\nu}\nabla \be_j^{n+1},\nabla \bchi_{h}\right)+\gamma\Big(\nabla\cdot\be_{j,\bR}^{n+1},\nabla\cdot\bchi_{h}\Big)+b^*\Big(\hspace{-1.5mm}\lab\bhu_h\rab^n,\be^{n+1}_j,\bchi_{h}\Big)\nonumber\\&+b^*\Big(\hspace{-1.5mm}\lab\be\rab^n,\bu_{j,h}^{n+1},\bchi_{h}\Big)-\Big(p_{j,h}^{n+1}-\hp_{j,h}^n,\nabla\cdot\bchi_{h}\Big)+2\mu\Delta t\Big((\hl^n)^2\nabla\be_j^{n+1},\nabla\bchi_{h}\Big)\nonumber\\&+2\mu\Delta t\Big(\big\{(l^n)^2-(\hl^n)^2\big\}\nabla\bu_{j,h}^{n+1},\nabla\bchi_{h}\Big)=-b^*\Big(\bhu^{'n}_{j,h},\be^n_j,\bchi_{h}\Big)-b^*\Big(\be^{'n}_{j},\bu_{j,h}^n,\bchi_{h}\Big)\nonumber\\&-\Big(\nu_j^{'}\nabla \be_{j}^{n},\nabla\bchi_{h}\Big).\label{step-1-eqn-1}
\end{align}
Take $\bchi_{h}=\be_j^{n+1}$ in \eqref{step-1-eqn-1} which yields $b^*\Big(\hspace{-1.5mm}\lab\bhu_h\rab^n,\be^{n+1}_j,\bchi_{h}\Big)=0$, and use polarization identity, to get
\begin{align}
    \frac{1}{2\Delta t}\Big(\|\be_j^{n+1}\|^2-\|\be_j^n\|^2+\|\be_j^{n+1}-\be_j^n\|^2\Big)+\|\Bar{\nu}^\frac12\nabla \be_j^{n+1}\|^2+\gamma\|\nabla\cdot\be_{j,\bR}^{n+1}\|^2\nonumber\\+b^*\Big(\hspace{-1.5mm}\lab\be\rab^n,\bu_{j,h}^{n+1},\be_j^{n+1}\Big)-\Big(p_{j,h}^{n+1}-\hp_{j,h}^n,\nabla\cdot\be_{j,\bR}^{n+1}\Big)+2\mu\Delta t\|\hl^n\nabla\be_j^{n+1}\|^2\nonumber\\+2\mu\Delta t\Big(\big\{(l^n)^2-(\hl^n)^2\big\}\nabla\bu_{j,h}^{n+1},\nabla\be_j^{n+1}\Big)=-b^*\Big(\bhu^{'n}_{j,h},\be^n_j,\be_j^{n+1}\Big)\nonumber\\-b^*\Big(\be^{'n}_{j},\bu_{j,h}^n,\be_j^{n+1}\Big)-\nu_j^{'}\Big(\nabla \be_{j}^{n},\nabla\be_j^{n+1}\Big).\label{pol-1}
\end{align}
Now, we find the bound of the terms in \eqref{pol-1} first. Similar as \eqref{fluc-bound-2}, we rearrange, use identity in \eqref{trilinear-identitiy}, Cauchy-Schwarz, H\"older's, Poincar\'e,  and \eqref{basic-ineq}-\eqref{basic-ineq-infinity} inequalities, in the following nonlinear term, to get
\begin{align}
b^*\Big(&\bhu^{'n}_{j,h},\be^n_j,\be_j^{n+1}\Big)= b^*\Big(\bhu^{'n}_{j,h},\be_j^{n+1},\be_j^{n+1}-\be^n_j\Big)\nonumber\\&=\left(\bhu^{'n}_{j,h}\cdot\nabla\be_j^{n+1}, \be_j^{n+1}-\be^n_j\right)+\frac12\left(\nabla\cdot\bhu^{'n}_{j,h},\be_j^{n+1}\cdot(\be_j^{n+1}-\be_j^n)\right)\nonumber\\&\le\|\bhu^{'n}_{j,h}\cdot\nabla\be_j^{n+1}\|\|\be_j^{n+1}-\be^n_j\|+\frac12\|\nabla\cdot\bhu^{'n}_{j,h}\|_{L^\infty}\|\be_j^{n+1}\|\|\be_j^{n+1}-\be_j^n\|\nonumber\\&\le \||\bhu^{'n}_{j,h}|\nabla\be_j^{n+1}\|\|\be_j^{n+1}-\be^n_j\|+C\|\nabla\cdot\bhu^{'n}_{j,h}\|_{L^\infty}\|\nabla\be_j^{n+1}\|\|\be_j^{n+1}-\be_j^n\|\nonumber\\&\le\|\hl^n\nabla\be_j^{n+1}\|\|\be_j^{n+1}-\be^n_j\|+C\|\nabla\cdot\bhu^{'n}_{j,h}\|_{L^\infty}\|\nabla\be_j^{n+1}\|\|\be_j^{n+1}-\be_j^n\|\nonumber\\&\le \frac{\alpha_j}{8}\|\nabla\be_j^{n+1}\|^2+\Delta t\|\hl^n\nabla\be_j^{n+1}\|^2+\left(\frac{1}{4\Delta t}+\frac{C}{\alpha_j}\|\nabla\cdot\bhu_{j,h}^{'n}\|_{L^\infty}^2\right)\|\be_j^{n+1}-\be_j^n\|^2.\label{before-young-error}
\end{align}
Applying H\"older's and Young’s inequalities, we have
\begin{align*}
    -\Big(\nu_j^{'}\nabla \be_{j}^{n},\nabla\be_j^{n+1}\Big)&\le \frac{\|\nu_j^{'}\|_\infty}{2}\Big(\|\nabla \be_{j}^{n}\|^2+\|\nabla \be_j^{n+1}\|^2\Big).
\end{align*}
Applying Cauchy-Schwarz and Young’s inequalities, we have
\begin{align*}
    \Big(p_{j,h}^{n+1}-\hp_{j,h}^n,\nabla\cdot\be_{j,\bR}^{n+1}\Big)&\le\frac{1}{2\gamma}\|p_{j,h}^{n+1}-\hp_{j,h}^n\|^2+\frac{\gamma}{2}\|\nabla\cdot\be_{j,\bR}^{n+1}\|^2.
\end{align*}
Use the trilinear bound in \eqref{nonlinearbound3}, estimate in Lemma \ref{lemma1}, and Young's inequalities, provides
\begin{align*}
b^*\Big(\hspace{-1.5mm}\lab\be\rab^n,\bu_{j,h}^{n+1},\be_j^{n+1}\Big)&\le C \|\hspace{-1.mm}\lab\be\rab^n\hspace{-1.mm}\|\left(\|\nabla\bu_{j,h}^{n+1}\|_{L^3}+\|\bu_{j,h}^{n+1}\|_{L^\infty}\right)\|\be_j^{n+1}\|\\&\le CC_*\|\hspace{-1.mm}\lab\be\rab^n\hspace{-1.mm}\|\|\nabla\be_j^{n+1}\|\\
    &\le \frac{\alpha_j}{8}\|\nabla\be_j^{n+1}\|^2+\frac{C}{\alpha_j}\|\hspace{-1.mm}\lab\be\rab^n\hspace{-1.mm}\|^2,\\
    b^*\Big(\be^{'n}_{j},\bu_{j,h}^n,\be_j^{n+1}\Big)&\le C\|\be^{'n}_{j}\|\left(\|\nabla\bu_{j,h}^n\|_{L^3}+\|\bu_{j,h}^n\|_{L^\infty}\right)\|\be_j^{n+1}\|\\&\le CC_* \|\be^{'n}_{j}\|\|\nabla\be_j^{n+1}\|\\&\le \frac{\alpha_j}{8}\|\nabla\be_j^{n+1}\|^2+\frac{C}{\alpha_j}\|\be^{'n}_{j}\|^2.
\end{align*}
For the second non-linear term, we apply H\"older’s and triangle inequalities, stability estimate of Algorithm \ref{coupled-alg}, uniform boundedness in Lemma \ref{lemma1} and in Assumption \ref{assump1}, Agmon’s \cite{Robinson2016Three-Dimensional}, discrete inverse, and Young's inequalities,  to get
\begin{align}
    2\mu\Delta t\Big(\big\{(l^n)^2-&(\hl^n)^2\big\}\nabla\bu_{j,h}^{n+1},\nabla\be_{j}^{n+1}\Big)\nonumber\\
    &\le 2\mu\Delta t\|(l^n)^2-(\hl^n)^2\|_{L^\infty}\|\nabla\bu_{j,h}^{n+1}\|\|\nabla\be_j^{n+1}\|\nonumber\\
    &=2\mu\Delta t \|\sum_{i=1}^J\left(|\bu_{i,h}^{'n}|^2-|\bhu_{i,h}^{'n}|^2\right)\|_{L^\infty}\|\nabla\bu_{j,h}^{n+1}\|\|\nabla\be_j^{n+1}\|\nonumber\\
    &\le 2\mu\Delta t \sum_{i=1}^J\|(\bu_{i,h}^{'n}-\bhu_{i,h}^{'n})\cdot(\bu_{i,h}^{'n}+\bhu_{i,h}^{'n})\|_{L^\infty}\|\nabla\bu_{j,h}^{n+1}\|\|\nabla\be_j^{n+1}\|\nonumber\\&\le 2\mu\Delta t \sum_{i=1}^J\|\bu_{i,h}^{'n}-\bhu_{i,h}^{'n}\|_{L^\infty}\|\bu_{i,h}^{'n}+\bhu_{i,h}^{'n}\|_{L^\infty}\|\nabla\bu_{j,h}^{n+1}\|\|\nabla\be_j^{n+1}\|\nonumber\\
    &\le C\Delta t^{\frac{1}{2}} \sum_{i=1}^J\|\be_i^{'n}\|_{L^\infty}\left(\|\bu_{i,h}^{'n}\|_{L^\infty}+\|\bhu_{i,h}^{'n}\|_{L^\infty}\right)\|\nabla\be_j^{n+1}\|\nonumber\\
    &\le C\Delta t^{\frac{1}{2}}\sum_{i=1}^J\|\be_i^n\|_{L^\infty}\|\nabla\be_j^{n+1}\|\le C\Delta t^\frac{1}{2}h^{-\frac{3}{2}}\sum_{i=1}^J\|\be_i^n\|\|\nabla\be_j^{n+1}\|\nonumber\\
    &\le\frac{\alpha_j}{8}\|\nabla\be_j^{n+1}\|^2+\frac{C\Delta t}{\alpha_jh^3}\sum_{i=1}^J\|\be_i^n\|^2.\label{mixing-length-difference}
\end{align}
Using the above estimates in \eqref{pol-1}, and reducing, produces
\begin{align}
    \frac{1}{2\Delta t}\Big(\|\be_j^{n+1}\|^2-\|\be_j^n\|^2\Big)+\left(\frac{1}{4\Delta t}-\frac{C}{\alpha_j}\|\nabla\cdot\bhu_{j,h}^{'n}\|_{L^\infty}^2\right)\|\be_j^{n+1}-\be_j^n\|^2+\frac{\Bar{\nu}_{\min}}{2}\|\nabla \be_j^{n+1}\|^2\nonumber\\+\frac{\gamma}{2}\|\nabla\cdot\be_{j,\bR}^{n+1}\|^2+(2\mu-1)\Delta t\|\hl^n\nabla\be_j^{n+1}\|^2\le \frac{\|\nu_j^{'}\|_\infty}{2}\|\nabla \be_{j}^{n}\|^2+\frac{1}{2\gamma}\|p_{j,h}^{n+1}-\hp_{j,h}^n\|^2\nonumber\\+\frac{C}{\alpha_j}\Big(\|\hspace{-1.mm}\lab\be\rab^n\hspace{-1.mm}\|^2+\|\be^{'n}_{j}\|^2\Big)+\frac{C\Delta t}{\alpha_jh^3}\sum_{i=1}^J\|\be_i^n\|^2.\label{upperbound1}
\end{align}
Choose the tuning parameter $\mu>\frac{1}{2}$, and time-step size  $\Delta t<\frac{C\alpha_j}{\max\limits_{1\le n\le M}\{\|\nabla\cdot\bhu_{j,h}^{'n}\|_{L^\infty}^2\}}$ and drop non-negative terms from left-hand-side, and rearrange
\begin{align}
    \frac{1}{2\Delta t}\Big(\|\be_j^{n+1}\|^2-\|\be_j^n\|^2\Big)+\frac{\Bar{\nu}_{\min}}{2}\left(\|\nabla \be_j^{n+1}\|^2-\|\nabla \be_j^{n}\|^2\right)+\frac{\alpha_j}{2}\|\nabla \be_j^{n}\|^2+\frac{\gamma}{2}\|\nabla\cdot\be_{j,\bR}^{n+1}\|^2\nonumber\\\le \frac{1}{2\gamma}\|p_{j,h}^{n+1}-\hp_{j,h}^n\|^2+\frac{C}{\alpha_j}\Big(\|\hspace{-1.mm}\lab\be\rab^n\hspace{-1.mm}\|^2+\|\be^{'n}_{j}\|^2\Big)+\frac{C\Delta t}{\alpha_jh^3}\sum_{i=1}^J\|\be_i^n\|^2.
\end{align}
Using triangle, and Young's inequalities, then multiply both sides by $2\Delta t$, and sum over the time steps $n=0,1,\cdots,M-1$, to obtain
\begin{align}
     \|\be_j^{M}\|^2+\alpha_j\Delta t\sum_{n=1}^M\|\nabla\be_j^n\|^2+\gamma\Delta t\sum_{n=1}^{M}\|\nabla\cdot\be_{j,\bR}^{n}\|^2\le\frac{\Delta t}{\gamma}\sum_{n=0}^{M-1}\|p_{j,h}^{n+1}-\hp_{j,h}^n\|^2\nonumber\\+\left(\frac{C}{J^2\alpha_j}\Delta t+\frac{C\Delta t^2}{\alpha_jh^3}\right)\sum_{n=1}^{M-1}\sum_{j=1}^J\|\be_j^n\|^2+\frac{C}{\alpha_j}\Delta t\sum_{n=1}^{M-1}\|\be_j^n\|^2.
\end{align} 
 Summing over $j=1,2,\cdots, J$, we have
\begin{align}
\sum_{j=1}^J\|\be_j^{M}\|^2+\alpha_{\min}\Delta t\sum_{n=1}^M\sum_{j=1}^J\|\nabla\be_j^n\|^2+\gamma\Delta t\sum_{n=1}^{M}\sum_{j=1}^J\|\nabla\cdot\be_{j,\bR}^{n}\|^2\nonumber\\\le\frac{\Delta t}{\gamma}\sum_{n=0}^{M-1}\sum_{j=1}^J\|p_{j,h}^{n+1}-\hp_{j,h}^n\|^2+\Delta t\sum_{n=1}^{M-1}\frac{C}{\alpha_{\min}}\left(1+\frac{1}{J}+\frac{J\Delta t}{h^3}\right)\sum_{j=1}^J\|\be_j^n\|^2.
\end{align} 
Apply discrete Gr\"onwall inequality given in Lemma \ref{dgl}, to get 
\begin{align}
\sum_{j=1}^J\|\be_j^{M}\|^2+\alpha_{\min}\Delta t\sum_{n=1}^M\sum_{j=1}^J\|\nabla\be_j^n\|^2+\gamma\Delta t\sum_{n=1}^{M}\sum_{j=1}^J\|\nabla\cdot\be_{j,\bR}^{n}\|^2\nonumber\\\le \frac{\Delta t}{\gamma} exp\lp \frac{CT}{\alpha_{\min}}\left(1+\frac{\Delta t}{h^3}\right)\rp\sum_{n=0}^{M-1}\sum_{j=1}^J\|p_{j,h}^{n+1}-\hp_{j,h}^n\|^2.\label{after-gronwall}
\end{align}
Using Lemma \ref{CR-lemma} with \eqref{after-gronwall} yields the following bound
\begin{align}
    \Delta t\sum_{n=1}^{M}\sum_{j=1}^J\|\nabla\be_{j,\bR}^{n}\|^2\le C_R^2\Delta t\sum_{n=1}^{M}\sum_{j=1}^J\|\nabla\cdot\be_{j,\bR}^{n}\|^2\nonumber\\\le\frac{C_R^2}{\gamma^2} exp\lp \frac{C}{\alpha_{\min}}\left(1+\frac{\Delta t}{h^3}\right)\rp\lp\Delta t\sum_{n=0}^{M-1}\sum_{j=1}^J\|p_{j,h}^{n+1}-\hp_{j,h}^n\|^2\rp.\label{step-1-bound}
\end{align}
\textbf{Step 2:} Estimate of $\be_{j,0}^{n}$: To find a bound on $\Delta t\sum\limits_{n=1}^{M}\sum\limits_{j=1}^J\|\nabla\be_{j,0}^{n}\|^2,$ take $\bchi_{h}=\be_{j,0}^{n+1}$ in \eqref{step-1-eqn-1}, which yields 
\begin{align}
    \frac{1}{\Delta t}\Big(\be_j^{n+1}-\be_j^n,\be_{j,0}^{n+1}\Big)+\|\Bar{\nu}^\frac12\nabla \be_{j,0}^{n+1}\|^2+b^*\Big(\hspace{-1.5mm}\lab\bhu_h\rab^n,\be^{n+1}_{j,\bR},\be_{j,0}^{n+1}\Big)+b^*\Big(\hspace{-1.5mm}\lab\be\rab^n,\bu_{j,h}^{n+1},\be_{j,0}^{n+1}\Big)\nonumber\\+2\mu\Delta t\Big((\hl^n)^2\nabla\be_j^{n+1},\nabla\be_{j,0}^{n+1}\Big)+2\mu\Delta t\Big(\big\{(l^n)^2-(\hl^n)^2\big\}\nabla\bu_{j,h}^{n+1},\nabla\be_{j,0}^{n+1}\Big)\nonumber\\=-b^*\Big(\bhu^{'n}_{j,h},\be^n_j,\be_{j,0}^{n+1}\Big)-b^*\Big(\be^{'n}_{j},\bu_{j,h}^n,\be_{j,0}^{n+1}\Big)-\Big(\nu_j^{'}\nabla \be_{j,0}^{n},\nabla\be_{j,0}^{n+1}\Big).\label{step-2-eqn-1}
\end{align}
Using the bound in \eqref{nonlinearbound} to the first trilinear form,  and the bound in \eqref{nonlinearbound3} to the second and fourth trilinear forms of 
		\eqref{step-2-eqn-1}, to obtain
\begin{align}
    \frac{1}{\Delta t}\Big(\be_j^{n+1}-\be_j^n,\be_{j,0}^{n+1}\Big)+\|\Bar{\nu}^\frac12\nabla \be_{j,0}^{n+1}\|^2+2\mu\Delta t\|\hl^n\nabla\be_{j,0}^{n+1}\|^2\nonumber\\\le C\|\nabla\hspace{-1mm}\lab\bhu_h\rab^n\hspace{-1mm}\|\|\nabla\be^{n+1}_{j,\bR}\|\|\nabla\be_{j,0}^{n+1}\|+C\|\hspace{-1.mm}\lab\be\rab^n\hspace{-1.mm}\|\left(\|\nabla\bu_{j,h}^{n+1}\|_{L^3}+\|\bu_{j,h}^{n+1}\|_{L^\infty}\right)\|\nabla\be_{j,0}^{n+1}\|\nonumber\\-2\mu\Delta t\Big((\hl^n)^2\nabla\be_{j,\bR}^{n+1},\nabla\be_{j,0}^{n+1}\Big)-2\mu\Delta t\Big(\big\{(l^n)^2-(\hl^n)^2\big\}\nabla\bu_{j,h}^{n+1},\nabla\be_{j,0}^{n+1}\Big)\nonumber\\-b^*\Big(\bhu^{'n}_{j,h},\be^n_j,\be_{j,0}^{n+1}\Big)+C\|\be^{'n}_{j}\|\left(\|\nabla\bu_{j,h}^n\|_{L^3}+\|\bu_{j,h}^n\|_{L^\infty}\right)\|\nabla\be_{j,0}^{n+1}\|-\Big(\nu_j^{'}\nabla \be_{j,0}^{n},\nabla\be_{j,0}^{n+1}\Big).
\end{align}
\begin{comment}
\begin{align*}
    b\Big(\bhw^{'n}_{j,h},\be^n_j,\be_{j,0}^{n+1}\Big)=b\Big(\bhw^{'n}_{j,h},\be^n_j,\be_{j}^{n+1}\Big)-b\Big(\bhw^{'n}_{j,h},\be^n_j,\be_{j,\bR}^{n+1}\Big)=-b\Big(\bhw^{'n}_{j,h},\be_{j,0}^{n+1},\be^n_j\Big)
\end{align*}
\end{comment}
Similar as \eqref{fluc-bound-2}, we rearrange, use identity in \eqref{trilinear-identitiy}, Cauchy-Schwarz, H\"older's, Poincar\'e,  and \eqref{basic-ineq} inequalities, in the following trilinear form, to get
\begin{align}    -b^*\Big(\bhu^{'n}_{j,h},\be^n_j,\be_{j,0}^{n+1}\Big)&=b^*\Big(\bhu^{'n}_{j,h},\be_{j,0}^{n+1},\be^n_j\Big)\nonumber\\&=\left(\bhu^{'n}_{j,h}\cdot\nabla\be_{j,0}^{n+1},\be^n_j\right)+\frac12\left(\nabla\cdot\bhu^{'n}_{j,h},\be_{j,0}^{n+1}\cdot\be^n_j\right)\nonumber\\&\le\|\bhu^{'n}_{j,h}\cdot\nabla\be_{j,0}^{n+1}\|\|\be^n_j\|+\frac12\|\nabla\cdot\bhu^{'n}_{j,h}\|_{L^\infty}\|\be_{j,0}^{n+1}\|\|\be_j^n\|\nonumber\\&\le \|\hl^n\nabla\be_{j,0}^{n+1}\|\|\be_j^n\|+C\|\nabla\cdot\bhu^{'n}_{j,h}\|_{L^\infty}\|\nabla\be_{j,0}^{n+1}\|\|\be_j^n\|\nonumber\\&\le\Delta t\|\hl^n\nabla\be_{j,0}^{n+1}\|+\frac{\alpha_j}{10}\|\nabla\be_{j,0}^{n+1}\|^2+\left(\frac{1}{4\Delta t}+\frac{C}{\alpha_j}\|\nabla\cdot\bhu^{'n}_{j,h}\|_{L^\infty}^2\right)\|\be_j^n\|^2.\label{before-young-e0}
		\end{align}
Using H\"older’s and Young's inequalities, gives
\begin{align*}
    -\Big(\nu_j^{'}\nabla \be_{j,0}^{n},\nabla\be_{j,0}^{n+1}\Big)\le\frac{\|\nu_j^{'}\|_{\infty}}{2}\left(\|\nabla \be_{j,0}^{n}\|^2+\|\nabla\be_{j,0}^{n+1}\|^2\right).
\end{align*}
Using the above bound, stability estimate, and Lemma \ref{lemma1}, reducing and rearranging, we have
\begin{align}
    \frac{1}{\Delta t}\Big(\be_j^{n+1}-\be_j^n,\be_{j,0}^{n+1}\Big)+\|\Bar{\nu}^\frac12\nabla \be_{j,0}^{n+1}\|^2+\left(2\mu-1\right)\Delta t\|\hl^n\nabla\be_{j,0}^{n+1}\|^2\nonumber\\\le \frac{C}{(\Bar{\nu}_{\min}\Delta t)^\frac12}\|\nabla\be^{n+1}_{j,\bR}\|\|\nabla\be_{j,0}^{n+1}\|+CC_*\|\hspace{-1.mm}\lab\be\rab^n\hspace{-1.mm}\|\|\nabla\be_{j,0}^{n+1}\|+\left(\frac{1}{4\Delta t}+\frac{C}{\alpha_j}\|\nabla\cdot\bhu^{'n}_{j,h}\|_{L^\infty}^2\right)\|\be_j^n\|^2\nonumber\\+\frac{\alpha_j}{10}\|\nabla\be_{j,0}^{n+1}\|^2+2\mu\Delta t\Big|\Big((\hl^n)^2\nabla\be_{j,\bR}^{n+1},\nabla\be_{j,0}^{n+1}\Big)\Big|+2\mu\Delta t\Big|\Big(\big\{(l^n)^2-(\hl^n)^2\big\}\nabla\bu_{j,h}^{n+1},\nabla\be_{j,0}^{n+1}\Big)\Big|\nonumber\\+CC_*\|\be^{'n}_{j}\|\|\nabla\be_{j,0}^{n+1}\|+\frac{\|\nu_j^{'}\|_{\infty}}{2}\left(\|\nabla \be_{j,0}^{n}\|^2+\|\nabla\be_{j,0}^{n+1}\|^2\right).\label{before-time-derivative-estimate}
\end{align}
To evaluate the time-derivative term, we use polarization identity, Cauchy-Schwarz, Young's and Poincar\'e's inequalities 
\begin{align*}
    \frac{1}{\Delta t}\Big(\be_j^{n+1}-\be_j^n,\be_{j,0}^{n+1}\Big)&=\frac{1}{\Delta t}\Big(\be_j^{n+1}-\be_j^n,\be_{j}^{n+1}-\be_{j,\bR}^{n+1}\Big)\\&=\frac{1}{2\Delta t}\Big(\|\be_j^{n+1}-\be_j^n\|^2+\|\be_j^{n+1}\|^2-\|\be_j^n\|^2\Big)-\frac{1}{\Delta t}\Big(\be_j^{n+1}-\be_j^n,\be_{j,\bR}^{n+1}\Big)\\&\ge \frac{1}{2\Delta t}\Big(\|\be_j^{n+1}\|^2-\|\be_j^n\|^2\Big)-\frac{C}{\Delta t}\|\nabla\be_{j,\bR}^{n+1}\|^2.
\end{align*}
Plugging the above estimate into \eqref{before-time-derivative-estimate} and using Cauchy-Schwarz's, and Young's inequalities again, yields
\begin{align}
    \frac{1}{2\Delta t}\Big(\|\be_j^{n+1}\|^2-\|\be_j^n\|^2\Big)+\|\Bar{\nu}^\frac12\nabla \be_{j,0}^{n+1}\|^2+\left(2\mu-1\right)\Delta t\|\hl^n\nabla\be_{j,0}^{n+1}\|^2\nonumber\\\le \left(\frac{C}{\alpha_j^2\Delta t}+\frac{C}{\Delta t}\right)\|\nabla\be^{n+1}_{j,\bR}\|^2+\frac{C}{\alpha_j}\|\hspace{-1.mm}\lab\be\rab^n\hspace{-1.mm}\|^2+2\mu\Delta t\Big|\Big((\hl^n)^2\nabla\be_{j,\bR}^{n+1},\nabla\be_{j,0}^{n+1}\Big)\Big|\nonumber\\+\left(\frac{1}{4\Delta t}+\frac{C}{\alpha_j}\|\nabla\cdot\bhu^{'n}_{j,h}\|_{L^\infty}^2\right)\|\be_j^n\|^2+\frac{2\alpha_j}{5}\|\nabla\be_{j,0}^{n+1}\|^2+\frac{C}{\alpha_j}\|\be^{'n}_{j}\|^2\nonumber\\+2\mu\Delta t\Big|\Big(\big\{(l^n)^2-(\hl^n)^2\big\}\nabla\bu_{j,h}^{n+1},\nabla\be_{j,0}^{n+1}\Big)\Big|+\frac{\|\nu_j^{'}\|_{\infty}}{2}\Big(\|\nabla \be_{j,0}^{n}\|^2+\|\nabla\be_{j,0}^{n+1}\|^2\Big).\label{before-non-linear-bounds}
\end{align}
We now use Cauchy-Schwarz, and  Young's inequalities, uniform boundedness in Assumption \ref{assump1} (which holds true for sufficiently large $\gamma$), and the stability estimate, to obtain  
\begin{align*}
    2\mu\Delta t\Big|\Big((\hl^n)^2\nabla\be_{j,\bR}^{n+1},\nabla\be_{j,0}^{n+1}\Big)\Big|&\le 2\mu\Delta t\|\hl^n\nabla\be_{j,\bR}^{n+1}\|\|\hl^n\nabla\be_{j,0}^{n+1}\|\\
    &\le\mu\Delta t\|\hl^n\nabla\be_{j,\bR}^{n+1}\|^2+\mu\Delta t\|\hl^n\nabla\be_{j,0}^{n+1}\|^2\\
    &\le\mu\Delta t\|\hl^n\|_{L^\infty}^2\|\nabla\be_{j,\bR}^{n+1}\|^2+\mu\Delta t\|\hl^n\nabla\be_{j,0}^{n+1}\|^2\\
    &\le C\Delta t\|\nabla\be_{j,\bR}^{n+1}\|^2+\mu\Delta t\|\hl^n\nabla\be_{j,0}^{n+1}\|^2.
\end{align*}
We follow the same treatment as in \eqref{mixing-length-difference}, and get
\begin{align*}
    2\mu\Delta t\Big(\big\{(l^n)^2-(\hl^n)^2\big\}\nabla\bu_{j,h}^{n+1},\nabla\be_{j,0}^{n+1}\Big)\le\frac{\alpha_j}{10}\|\nabla\be_{j,0}^{n+1}\|^2+\frac{C\Delta t}{\alpha_jh^3}\sum_{i=1}^J\|\be_i^n\|^2.
\end{align*}
Use the above estimates in \eqref{before-non-linear-bounds}, assume $\mu\ge 1$ to drop non-negative term from left-hand-side, use triangle and Young's inequalities and reduce, then the equation \eqref{before-non-linear-bounds} becomes
\begin{align}
    \frac{1}{2\Delta t}\Big(\|\be_j^{n+1}\|^2-\|\be_j^n\|^2\Big)+\frac{\Bar{\nu}_{\min}}{2}\|\nabla \be_{j,0}^{n+1}\|^2\le C\left(\frac{1}{\alpha_j^2\Delta t}+\frac{1}{\Delta t}+\Delta t\right)\|\nabla\be^{n+1}_{j,\bR}\|^2\nonumber\\+\Big(\frac{C}{\alpha_jJ^2}+\frac{C\Delta t}{\alpha_jh^3}\Big)\sum_{i=1}^J\|\be_i^n\|^2+\left(\frac{C}{\alpha_jJ^2}+\frac{1}{4\Delta t}+\frac{C}{\alpha_j}\|\nabla\cdot\bhu^{'n}_{j,h}\|_{L^\infty}^2\right)\|\be_j^n\|^2+\frac{\|\nu_j^{'}\|_{\infty}}{2}\|\nabla \be_{j,0}^{n}\|^2.\label{e-equ-bounded}
\end{align}
Use \eqref{basic-ineq-infinity}, discrete inverse inequality, Assumption \ref{assump1}, and rearranging
\begin{align}
    \frac{1}{2\Delta t}\Big(\|\be_j^{n+1}\|^2-\|\be_j^n\|^2\Big)+\frac{\Bar{\nu}_{\min}}{2}\Big(\|\nabla \be_{j,0}^{n+1}\|^2-\|\nabla \be_{j,0}^{n}\|^2\Big)\nonumber\\+\frac{\alpha_j}{2}\|\nabla \be_{j,0}^{n}\|^2\le C\left(\frac{1}{\alpha_j^2\Delta t}+\frac{1}{\Delta t}+\Delta t\right)\|\nabla\be^{n+1}_{j,\bR}\|^2\nonumber\\+\Big(\frac{C}{\alpha_jJ^2}+\frac{C\Delta t}{\alpha_jh^3}\Big)\sum_{i=1}^J\|\be_i^n\|^2+\left(\frac{C}{\alpha_jJ^2}+\frac{1}{4\Delta t}+\frac{C}{\alpha_jh^2}\right)\|\be_j^n\|^2.
\end{align}
Multiply both sides by $2\Delta t$, and summing over the time-step $n=0,1,\cdots,M-1$, results in
\begin{align}
    \|\be_j^M\|^2+\Bar{\nu}_{\min}\Delta t\|\nabla\be_{j,0}^M\|^2+\alpha_j\Delta t\sum_{n=1}^{M-1}\|\nabla \be_{j,0}^{n}\|^2\le C\Delta t\left(\frac{1}{\alpha_j^2\Delta t}+\frac{1}{\Delta t}+\Delta t\right)\sum_{n=1}^M\|\nabla\be^{n}_{j,\bR}\|^2\nonumber\\+C\Delta t\Big(\frac{1}{\alpha_jJ^2}+\frac{\Delta t}{\alpha_jh^3}\Big)\sum_{n=1}^{M-1}\sum_{i=1}^J\|\be_i^n\|^2+C\Delta t\left(\frac{1}{\alpha_jJ^2}+\frac{1}{\Delta t}+\frac{1}{\alpha_jh^2}\right)\sum_{n=1}^{M-1}\|\be_j^n\|^2.
\end{align}
Now, simplifying, and summing over $j=1,2,\cdots,J$, we have
\begin{align}
    \sum_{j=1}^J\|\be_j^M\|^2+\Delta t\sum_{n=1}^{M}\alpha_{\min}\sum_{j=1}^J\|\nabla \be_{j,0}^{n}\|^2\le \Delta t \sum_{n=1}^MC\left(\frac{1}{\alpha_{\min}^2\Delta t}+\frac{1}{\Delta t}+\Delta t\right)\sum_{j=1}^J\|\nabla\be^{n}_{j,\bR}\|^2\nonumber\\+\frac{C\Delta t}{\alpha_{\min}}\Big(1+\frac{\alpha_{\min}}{\Delta t}+\frac{1}{h^2}+\frac{J\Delta t}{h^3}\Big)\sum_{n=1}^{M-1}\sum_{i=1}^J\|\be_i^n\|^2.
\end{align}
Apply the version of the discrete Gr\"onwall inequality given in Lemma \ref{dgl}
\begin{align}
    \sum_{j=1}^J\|\be_j^M\|^2+\alpha_{\min}\Delta t\sum_{n=1}^{M}\sum_{j=1}^J\|\nabla \be_{j,0}^{n}\|^2\nonumber\\\le C exp \left(\frac{CT}{\alpha_{\min}}\Big(1+\frac{\alpha_{\min}}{\Delta t}+\frac{1}{h^2}+\frac{\Delta t}{h^3}\Big)\right)\left(1+\frac{1}{\alpha_{\min}^2}+\Delta t^2\right)\sum_{n=1}^M\sum_{j=1}^J\|\nabla\be^{n}_{j,\bR}\|^2,\label{step-2-after-gronwell}
\end{align}
and use the estimate  \eqref{step-1-bound} in \eqref{step-2-after-gronwell}, to get
\begin{comment}
\begin{align}
    \Delta t\sum_{n=1}^{M}\sum_{j=1}^J\|\nabla \be_{j,0}^{n}\|^2\le\frac{C}{\alpha_{min}}  exp \left(\frac{CC_*^2}{J^2\alpha_{min}}+\frac{C\Delta t}{h^3\alpha_{min}}+\frac{T}{2}\right)\left(\frac{C_*^2}{\alpha_{min}}+\frac{1}{\Delta t}+\Delta t\right)\nonumber\\\times\frac{C_R^2}{\gamma^2} exp\lp CT\left(\frac{C_*^2}{\alpha_{min}}+\frac{\Delta t}{h^3\alpha_{min}}\right)\rp\lp  \Delta t\sum_{n=0}^{M-1}\sum_{j=1}^J\Big(\|q_{j,h}^{n+1}-\hq_{j,h}^n\|^2+\|\lambda_{j,h}^{n+1}-\hlam_{j,h}^n\|^2\Big)\rp.
\end{align}
\end{comment}
\begin{align}
    \Delta t\sum_{n=1}^{M}\sum_{j=1}^J\|\nabla \be_{j,0}^{n}\|^2\le \frac{C}{\gamma^2\alpha_{\min}} exp\lp\frac{C}{\alpha_{\min}} \left(1+\frac{\alpha_{\min}}{\Delta t}+\frac{1}{h^2}+\frac{J\Delta t}{h^3}\right)\rp \nonumber\\\times\left(1+\frac{1}{\alpha_{\min}^2}+\Delta t^2\right)\lp\sum_{n=0}^{M-1}\sum_{j=1}^J\|p_{j,h}^{n+1}-\hp_{j,h}^n\|^2\rp.
\end{align}
Using triangle and Young's inequalities
\begin{align}
    \Delta t\sum_{n=1}^{M}\|\nabla \hspace{-1.mm}\lab\be_{0}\rab^{n}\|^2\le\frac{2\Delta t}{J^2}\sum_{n=1}^{M}\sum_{j=1}^J\|\nabla \be_{j,0}^{n}\|^2\nonumber\\\le \frac{C}{\gamma^2\alpha_{\min}} exp\lp\frac{C}{\alpha_{\min}} \left(1+\frac{\alpha_{\min}}{\Delta t}+\frac{1}{h^2}+\frac{J\Delta t}{h^3}\right)\rp\left(1+\frac{1}{\alpha_{\min}^2}+\Delta t^2\right)\lp\sum_{n=0}^{M-1}\sum_{j=1}^J\|p_{j,h}^{n+1}-\hp_{j,h}^n\|^2\rp,\label{step-2-bound-ensemble}
\end{align}
and
\begin{align}
    \Delta t\sum_{n=1}^{M}\|\nabla\hspace{-1mm}\lab\be_{\bR}\rab^{n}\|^2\le\frac{2\Delta t}{J^2}\sum_{n=1}^{M}\sum_{j=1}^J\|\nabla\be_{j,\bR}^{n}\|^2\nonumber\\\le\frac{C}{\gamma^2} exp\lp\frac{C}{\alpha_{\min}} \left(1+\frac{\Delta t}{h^3}\right)\rp\lp  \Delta t\sum_{n=0}^{M-1}\sum_{j=1}^J\|p_{j,h}^{n+1}-\hp_{j,h}^n\|^2\rp.
\end{align}
Finally, apply triangle and Young's inequalities on $$\|\nabla\hspace{-1mm}\lab\bu_h\rab^n-\nabla\hspace{-1mm}\lab\bhu_h\rab^n\|^2$$ to obtain the desire result.

\end{proof}
\begin{comment}
\begin{remark}\label{remark-1}
 We assume $\|\bhv^n_{j,h}\|_{\infty},\;\|\bhw^n_{j,h}\|_{\infty}\le K^*,\;1\le j\le J$ for some positive constant $K^*>0$. Then,
\begin{align*}
    \|l_{\bhv,h}^n\|_{\infty}=\|\sqrt{\sum_{j=1}^J|\bhv_{j,h}^{'n}}|^2\|_{\infty}&=\|\bhv_{j^*,h}^{'n}\|_{\infty}\;(\text{for some}\; j^*, 1\le j^*\le J)\le \|\bhv_{j^*,h}^n\|_{\infty}+\|\hspace{-1mm}<\bhv_h>^n\hspace{-1mm}\|_{\infty}\le 2K^*.
\end{align*}
Hence, $\|(l_{\bhw,h}^n)^2\|_{\infty}=\|l_{\bhw,h}^n\|^2_{\infty}$. Similarly, $\|(l_{\bhv,h}^n)^2\|_{\infty}=\|l_{\bhv,h}^n\|^2_{\infty}$.
\end{remark}
\end{comment}

We prove the following Lemma by strong mathematical induction.
\begin{lemma}\label{uniform-boundedness-lemma-proof}
If $\gamma\rightarrow\infty$ then there exists a constant $C_*$ which is independent of $h$, and $\Delta t$, such that for sufficiently small $h$ for a fixed mesh and fixed $\Delta t$, the solution of the Algorithm \ref{SPP-FEM} satisfies
\begin{align}
    \max_{0\le n\le M}\|\bhu_{j,h}^n\|_{L^\infty}\le C_*,\hspace{2mm}\text{for all}\hspace{2mm}j=1,2,\cdots,J.
\end{align}
\end{lemma}
\begin{proof}
Basic step: $\bhu_{j,h}^0=I_h(\bu_j^{true}(0,\bx)),$ where $I_h$ is an appropriate interpolation operator. Due to the regularity assumption of $\bu_j^{true}(0,\bx)$, we have $\|\bhu_{j,h}^0\|_{L^\infty}\le C_*$, for some constant $C_*>0$.\\
Inductive step: Assume for some $N\in\mathbb{N}$ and $N<M$, $\|\bhu_{j,h}^n\|_{L^\infty}\le C_*$ holds true for $n=0,1,\cdots,N$. Then, using triangle inequality and Lemma \ref{lemma1}, we have
\begin{align*}
\|\bhu_{j,h}^{N+1}\|_{L^\infty}\le\|\bhu_{j,h}^{N+1}-\bu_{j,h}^{N+1}\|_{L^\infty}+C_*.
\end{align*}
Using Agmon’s inequality \cite{Robinson2016Three-Dimensional}, and discrete inverse inequality, yields
\begin{align}
    \|\bhu_{j,h}^{N+1}\|_{L^\infty}\le Ch^{-\frac32}\|\bhu_{j,h}^{N+1}-\bu_{j,h}^{N+1}\|+C_*.
\end{align}
Next, using equation \eqref{after-gronwall}
\begin{align}
    &\|\bhu_{j,h}^{N+1}\|_{L^\infty}\le C_*+\frac{C}{h^{\frac32}\gamma^{\frac12} } exp\lp \frac{CC_*^2}{\alpha_{\min}}+\frac{C\Delta t}{h^3\alpha_{\min}}\rp\lp  \Delta t\sum_{n=0}^{N}\sum_{j=1}^J\|p_{j,h}^{n+1}-\hp_{j,h}^n\|^2\rp^{\frac12}\hspace{-3mm}.
\end{align}
For a fixed mesh, and time-step size, as $\gamma\rightarrow \infty$, yields $\|\bhu_{j,h}^{N+1}\|_{L^\infty}\le C_*$. Hence, by the principle of strong mathematical induction, $\|\bhu_{j,h}^{n}\|_{L^\infty}\le C_*$ holds true for $0\le n\le M$.
\end{proof}
\section{SCM}\label{scm} In this paper, sparse grid algorithm \cite{FTW2008} is consider as SCM in which for a given time $t^n$ and a set of sample points $\{\by_j\}_{j=1}^{J}\subset\bGamma$, we approximate the exact solution of \eqref{momentum}-\eqref{nse-initial} by solving a discrete scheme (which can be either Coupled-EEV or SPP-EEV). Then, for a basis $\{\phi_l\}_{l=1}^{N_p}$ of dimension $N_p$ for the space $L_\rho^2(\bGamma)$, a discrete approximation is constructed with coefficients $c_l(t^n,\bx)$ of the form \begin{align*}
\bu_h^{sc}(t^n\hspace{-0.5ex},\bx,\by)=\sum_{l=1}^{N_{p}}c_l(t^n\hspace{-0.5ex},\bx)\phi_l(\by),
\end{align*}
which is essentially an interpolant. In the sparse grid algorithm, we consider Leja and Clenshaw--Curtis points as the interpolation points that come with the associated weights $\{w^j\}_{j=1}^{N_{sc}}$. SCM were recently developed for the UQ of the Quantity of Interest (QoI), $\Theta$, which can be the lift, drag, and energy and provide statistical information about QoI, that is, $$\mathbb{E}[\Theta(\bu(t^n))]=\int_{\Gamma} \Theta(\bu(t^n),\by)\rho(\by)dy\approx\sum_{j=1}^{N_{sc}}w^j\Theta(\bu_{j,h}^n).$$
 SCM are highly efficient compared to the standard MC method for large-scale problems with large-dimensional random inputs because in this case, the rate of convergence of MC generates unaffordable computational cost. A full outline of the SCM-SPP-EEV is given in Algorithm \ref{SCM-SPP-algo}.
 \begin{algorithm}
	\caption{SCM-SPP-EEV}\label{SCM-SPP-algo}
	\begin{algorithmic}
		%\Procedure OFFLINE construction %\Comment{The g.c.d. of a and b}
  \Procedure Sparse grid algorithm
		\State \textbf{Initialization:} Mesh, FE functions, $T$, $M$, $\{\by_{j}\}_{j=1}^{J}$, $\{w_j\}_{j=1}^{J}$\\
  \hspace{5mm}\textbf{Pre-compute:} $ \{\nu_{j}\}_{j=1}^{J}$,$ \{\bu_{j,h}^0\}_{j=1}^{J}$\\
  \hspace{5mm}\textbf{for} $n = 0, \ldots, M-1$ \textbf{do}\\
		\hspace{10mm}\textbf{for} $j = 1, \ldots, J$ \textbf{do}\\
		\hspace{15mm} To compute $\bhu^{n+1}_{j,h}$, solve \eqref{spp-step-1}-\eqref{spp-step-2-2}\\
  \hspace{15mm} Calculate $\Theta(\bhu^{n+1}_{j,h})$\\
  \hspace{10mm}\textbf{end for}\\
  \hspace{9mm} Estimate $\mathbb{E}[\Theta(\bu(t^{n+1}))] \approx \sum\limits_{j = 1}^{J} w^{j} \Theta(\bhu_{j,h}^{n+1})$\\
  \hspace{5mm}\textbf{end for}
		\EndProcedure
	\end{algorithmic}
\end{algorithm}

Instead of computing $\bhu^{n+1}_{j,h}$ by solving \eqref{spp-step-1}-\eqref{spp-step-2-2} in Algorithm \eqref{SCM-SPP-algo}, if we compute $\bu^{n+1}_{j,h}$ by solving \eqref{couple-eqn-1}-\eqref{couple-incompressibility} and follow the rest of the steps, which will lead to the SCM-Coupled-EEV algorithm.
\section{Numerical Experiments}\label{numerical-experiment}
In this section, we present a series of numerical tests that verify the predicted convergence rates and show the performance of the scheme on some benchmark problems. In all experiments, we consider $\bx=(x_1,x_2)$ for 2D problems, pointwise-divergence free $(\mathbb{P}_2,\mathbb{P}_1^{disc})$ Scott-Vogelius element for the Coupled-EEV scheme, and $(\mathbb{P}_2,\mathbb{P}_1)$ Taylor-hood element for the SPP-EEV scheme on a barycenter refined triangular meshes.

\subsection{Convergence Rates Verification}

In the first experiment, we verify the theoretically found convergence rates beginning with the following analytical solution:
\[ {\bu}=\left(\begin{array}{c} \cos x_2+(1+e^t)\sin x_2 \\ \sin x_1+(1+e^t)\cos x_1 \end{array} \right),\;\;\text{and}\;\; \ p =\sin(x_1+x_2)(1+e^t),
\]
on domain $\cD=[0,1]^2$. Then, we introduce noise as $\bu_j=(1+k_j\epsilon)\bu$, and $p_j=(1+k_j\epsilon)p$, where $\epsilon$ is a perturbation parameter, $k_j:=(-1)^{j+1}4\lceil j/2\rceil/J$, $j=1,2,\cdots, J$, and $J=20$. We consider $\epsilon=0.01$, this will introduce $10\%$ noise in the initial condition, boundary condition, and the forcing functions. The forcing function $\bif_j$ is computed using the above synthetic data into \eqref{gov1}. We assume the viscosity $\nu$ is a continuous uniform random variable, and consider three random samples of size $J$ as $\nu\sim
\mathcal{U}(0.009, 0.011)$ with $E[\nu]=0.01$,  $\nu\sim
\mathcal{U}(0.0009, 0.0011)$ with $E[\nu]=0.001$, and $\nu\sim
\mathcal{U}(0.00009, 0.00011)$ with $E[\nu]=0.0001$. The initial and boundary conditions are $\bu_{j,h}^0=\bu_{j}(\bx,0)$, and $\bu_{j,h}|_{\partial\cD}=\bu_j$, respectively.

\subsubsection{SPP-EEV scheme converges to the Coupled-EEV scheme as $\gamma\rightarrow\infty$}

We define the velocity, and pressure errors as $<\hspace{-1mm}\hat{\be}_{\bu}\hspace{-1mm}>:=<\hspace{-1mm}\bu_h\hspace{-1mm}>-<\hspace{-1mm}\bhu_{h}\hspace{-1mm}>$, and $<\hspace{-1mm}\hat{e}_p\hspace{-1mm}>:=<\hspace{-1mm}\mathscr{P}_h\hspace{-1mm}>-<\hspace{-1mm}\hp_{h,\gamma}\hspace{-1mm}>$, respectively, where $$\mathscr{P}_{j,h}:=p_{j,h}-\frac{1}{Area(\cD)}\int_{\cD}p_{j,h} d\cD,$$ and $$\hp^n_{j,h,\gamma}:=\hp_{j,h}^{n-1}-\frac{1}{Area(\cD)}\int_{\cD}\hp_{j,h}^{n-1}d\cD-\gamma\nabla\cdot\bhu_{j,h}^{n}.$$ That is, these errors are the difference between the outcomes of the coupled and projection schemes.

We consider the simulation end time $T=1$, time-step size $\Delta t=T/10$, and $h=1/32$. Starting with $\gamma=0$, we successively increase $\gamma$ from 1e-2 by a factor of 10, record the errors in velocity and pressure, and compute the convergence rates, and finally present them in Table \ref{convergence-tab}. We observe that as $\gamma$ increases, the convergence rates asymptotically converge to 1, which is in excellent agreement with the theoretically predicted convergence rates in terms of $\gamma$ presented in \eqref{convergence-thm}.
\begin{table}[!ht] 
\begin{center}
\small{\begin{tabular}{|c|c|c|c|c|c|c|c|c|}\hline
\multicolumn{9}{|c|}{Fixed $T=1$, $\Delta t=T/10$, $h=1/32$}\\\hline
$\hspace{-1mm}\epsilon=0.01\hspace{-1mm}$&\multicolumn{4}{c|}{$\mathbb{E}[\nu]=0.01$}&\multicolumn{4}{c|}{$\mathbb{E}[\nu]=0.001$}\\\hline
$\gamma$ & $\|\hspace{-1mm}<\hspace{-1mm}\hat{\be}_{\bu}\hspace{-1mm}>\hspace{-1mm}\|_{2,1}$ & rate &$\|\hspace{-1mm}<\hspace{-1mm}\hat{e}_p\hspace{-1mm}>\hspace{-1mm}\|_{2,0}$& rate &$\|\hspace{-1mm}<\hspace{-1mm}\hat{\be}_{\bu}\hspace{-1mm}>\hspace{-1mm}\|_{2,1}$ & rate  &$\|\hspace{-1mm}<\hspace{-1mm}\hat{e}_p\hspace{-1mm}>\hspace{-1mm}\|_{2,0}$& rate\\\hline
$0$ & 3.9912e-0& & 6.7215e-1& & 5.1823e-0 & &6.7359e-1&  \\\hline
1e-2 & 3.6882e-0 & 0.03   &6.6291e-1   &0.01& 4.6428e-0 & 0.05   &6.6413e-1   &0.01\\\hline
1e-1 & 2.7593e-0 & 0.13   &5.9839e-1   &0.04& 3.4584e-0 & 0.13   &5.9905e-1   &0.04\\\hline
1e-0 & 9.3147e-1 & 0.47   &3.1922e-1   &0.27& 1.0567e-0 & 0.51   &3.1895e-1   &0.27\\\hline
1e+1 & 1.5728e-1 & 0.77   &5.2694e-2   &0.78& 1.9040e-1 & 0.74   &5.2358e-2   &0.78\\\hline
1e+2 & 1.7306e-2 & 0.96   &5.5839e-3   &0.97&2.1293e-2 & 0.95   &5.5397e-3   &0.98\\\hline
1e+3 & 1.7479e-3 & 1.00   &6.0452e-4   &1.00& 2.1537e-3 & 1.00   &6.0788e-4   &0.96\\\hline
\end{tabular}}
\end{center}
\caption{\footnotesize SPP-EEV scheme converges to the Coupled-EEV scheme as $\gamma$ increases with $J=20$, and $\mu=1$.}\label{convergence-tab}
\end{table}

\subsubsection{Spatial and temporal convergence of the SPP-EEV scheme}
Now, we define the error between the solution of SPP-EEV scheme and the exact solution as $<\hspace{-1mm}\be_{\bu}\hspace{-1mm}>:=<\hspace{-1mm}\bu^{true}\hspace{-1mm}>-<\hspace{-1mm}\bhu_h\hspace{-1mm}>$. The upper bound of this error is the same as it in \eqref{convergence-error} for large $\gamma$, which can be shown by using triangle inequality. To observe spatial convergence, we keep temporal error small enough and thus we fix a very short simulation end time $T=0.001$. We successively reduce the mesh width $h$ by a factor of 1/2, run the simulations, and record the errors, and convergence rates in Table \ref{sp-convergence-ep-0.01}. For $\gamma=$1e+6, we observe the second order convergence rates for all the three samples, which support our theoretical finding for the $(\mathbb{P}_2,\mathbb{P}_1)$ element.
\begin{table}[!ht] 
		\begin{center}
			\small\begin{tabular}{|c|c|c|c|c|c|c|}\hline
				\multicolumn{7}{|c|}{Spatial convergence (fixed $T=0.001$, $\Delta t =T/8$)}\\\hline
				$\hspace{-1mm}\epsilon=0.01\hspace{-1mm}$&\multicolumn{2}{c|}{\hspace{-1mm}$\mathbb{E}[\nu]=0.01$ }&\multicolumn{2}{c|}{$\mathbb{E}[\nu]=0.001$}&\multicolumn{2}{c|}{$\mathbb{E}[\nu]=0.0001$}\\\hline
				$h$ & $\|\hspace{-1mm}<\hspace{-1mm}e_{\bu}\hspace{-1mm}>\hspace{-1mm}\|_{2,1}$ & rate   
				 & 
				$\|\hspace{-1mm}<\hspace{-1mm}e_{\bu}\hspace{-1mm}>\hspace{-1mm}\|_{2,1}$ & rate & 
				$\|\hspace{-1mm}<\hspace{-1mm}e_{\bu}\hspace{-1mm}>\hspace{-1mm}\|_{2,1}$ & rate  
				 \\ \hline
$1/2$ & 4.5123e-4 &&4.5128e-4& &4.5128e-4 & \\\hline
$1/4$ & 1.1568e-4 & 1.96 &1.1569e-4 & 1.96  & 1.1569e-4 & 1.96\\\hline
$1/8$ & 2.9134e-5 & 1.99 &2.9138e-5 & 1.99  & 2.9139e-5 & 1.99\\\hline
$1/16$ & 7.3380e-6 & 1.99 &7.3495e-6 & 1.99& 7.3516e-6 & 1.99\\\hline
$1/32$ & 1.8636e-6 & 1.98 &1.8938e-6 & 1.96 & 1.9133e-6 & 1.94\\\hline 

			\end{tabular}
		\end{center}
		\caption{\footnotesize Spatial errors and convergence rates of SPP-EEV scheme with $\epsilon=0.01$, $J=20$, $\mu=1$, and $\gamma=$1e+6.}\label{sp-convergence-ep-0.01} 
	\end{table}
 
 On the other hand, to observe temporal convergence, we keep fixed mesh size $h=1/64$, and simulation end time $T=1$. We run the simulations with various time-step sizes $\Delta t$ beginning with $T/2$ and successively reduce it by a factor of 1/2, record the errors, compute the convergence rates, and present them in Table \ref{tm-convergence-ep-0.01}. We observe the convergence rates approximately equal to 1. Since the backward-Euler formula is used to approximate the time derivative in the proposed SPP-EEV scheme, the found temporal convergence rate is optimal and in excellent agreement with the theory for all the three samples.
\begin{table}[!ht] 
		\begin{center}
			\small\begin{tabular}{|c|c|c|c|c|c|c|}\hline
				\multicolumn{7}{|c|}{Temporal convergence (fixed $T=1$, $h =1/64$)}\\\hline
$\hspace{-1mm}\epsilon=0.01\hspace{-1mm}$&\multicolumn2{c|}{$\mathbb{E}[\nu]=0.01$}&\multicolumn2{c|}{$\mathbb{E}[\nu]=0.001$}&\multicolumn{2}{c|}{$\mathbb{E}[\nu]=0.0001$}\\\hline
				$\Delta t$ & $\|\hspace{-1mm}<\hspace{-1mm}e_{\bu}\hspace{-1mm}>\hspace{-1mm}\|_{2,1}$ & rate   
				&$\|\hspace{-1mm}<\hspace{-1mm}e_{\bu}\hspace{-1mm}>\hspace{-1mm}\|_{2,1}$  & rate & 
				$\|\hspace{-1mm}<\hspace{-1mm}e_{\bu}\hspace{-1mm}>\hspace{-1mm}\|_{2,1}$ & rate   
				 \\ \hline
$T/2$ &9.9272e-2  & & 3.1647e-1&  &8.1904e-1&\\\hline
$T/4$ & 4.3909e-2 & 1.18    &  1.3968e-1 & 1.18 &3.6142e-1 & 1.18\\\hline
$T/8$ &  2.0572e-2 & 1.09    &   6.5374e-2 & 1.10 &1.6932e-1 & 1.09\\\hline
$T/16$ &  9.9601e-3 & 1.05    &  3.1629e-2 & 1.05 &8.2043e-2 & 1.05\\\hline
$T/32$ &  4.9022e-3 & 1.02    &   1.5580e-2 & 1.02 &4.0510e-2 & 1.02\\\hline
$T/64$ &  2.4372e-3 & 1.01    &   7.7401e-3 & 1.01 &2.0186e-2 & 1.00\\\hline
$T/128$ &  1.2218e-3 & 1.00   &  3.8780e-3 & 1.00 &1.0102e-2 & 1.00\\\hline
$T/256$ &  6.1134e-4 & 1.00   &   1.9548e-3 & 0.99 &5.0963e-3 & 0.99\\\hline
			\end{tabular}
		\end{center}
		\caption{\footnotesize Temporal errors and convergence rates of SPP-EEV scheme for $\bu$ and $p$ with $\epsilon=0.01,J=20$, $\gamma=$ 1e+5.}\label{tm-convergence-ep-0.01} 
	\end{table}

%\subsection{Computational time}
\subsection{Taylor Green-vortex (TGV) Problem \cite{taylor1937mechanism}} 

We consider the following closed form exact solution of \eqref{momentum}-\eqref{nse-initial}:
\[	{\bu}=\left(\hspace{-2mm}\begin{array}{c} \sin x_1\cos x_2\;e^{-2\nu t} \\ -\cos x_1\sin x_2\;e^{-2\nu t} \end{array} \hspace{-2mm}\right)\hspace{-1mm},\;\text{and}\; \ p =\frac14(\cos(2x_1)+\cos(2x_2))e^{-4\nu t},
	\]
together with the domain
$\cD=[0,L]\times[0,L]$, and $\bif=\textbf{0}$. The time-dependent TGV problem shows decaying vortex as time grows. In this section, we consider the stochastic NSE~\eqref{momentum}-\eqref{nse-initial} with a random viscosity $\nu({\bx},{\by})$, where ${\by}=(y_1,y_2,\cdots,y_N)\in\Gamma\subset\mathbb{R}^N$ be a finite $N\in\mathbb{N}$ dimensional vector \cite{babuvska2007stochastic,gunzburger2019evolve} distributed according to a joint probability density function in some parameter space $\bGamma=\prod\limits_{l=1}^N\Gamma_l$ with $\mathbb{E}[\by]=\textbf{0}$, and $\mathbb{V}ar[\by]=\textbf{I}_N$. We also consider $\mathbb{E}[\nu](\bx)=\frac{c}{1000}$ for a suitable $c>0$, $\mathbb{C}ov[\nu]({\bx},{\bx^{'}})=\frac{1}{1000^2}exp\left(-\frac{({\bx}-{\bx^{'}})^2}{l^2}\right)$, $L=\pi$ is the characteristic length, and $l$ is the correlation length. 
Then, the viscosity random field can be represented by:
\begin{align}
\nu(\bx, \by_j)&=\frac{1}{1000}\psi(\bx, \by_j),
\end{align}where the Karhunen-Lo\'eve expansion as below:
\begin{align}
\psi(\bx, \by_j)=c+\left(\frac{\sqrt{\pi}l}{2}\right)^{\frac12}y_{j,1}(\omega)&+\sum_{k=1}^{q}\sqrt{\xi_k}\bigg(\sin\left(\frac{k\pi x_1}{L}\right)\sin\left(\frac{k\pi x_2}{L}\right)y_{j,2k}(\omega)\nonumber\\&+\cos\left(\frac{k\pi x_1}{L}\right)\cos\left(\frac{k\pi x_2}{L}\right)y_{j,2k+1}(\omega)\bigg), \label{eq:var-vis}
\end{align}
in which the infinite series is truncated up to the first $q$ terms. The uncorrelated random variables $y_{j,k}$ have eigenvalues are equal to
$$\sqrt{\xi_k}=(\sqrt{\pi}l)^{\frac12}exp\left(-\frac{(k\pi l)^2}{8}\right).$$
For our test problem, we consider the random variables $y_{j,k}(\omega)\in[-\sqrt{3},\sqrt{3}]$, the correlation length $l=0.01$, $N=5$, $c=1$, $q=2$, $k=1,2,\cdots, N$, $J=11$ stochastic collocation points, and $j=1,2,\cdots,J$. We consider the Clenshaw--Curtis sparse grid as the SCM, and generated it via the software package TASMANIAN \cite{stoyanov2015tasmanian,doecode_6305} with 5D stochastic collocation points and their corresponding weights. An unstructured bary-centered refined triangular mesh that provides 45,087 degrees of freedom (dof) is considered.

We consider the boundary condition $\bu_{j,h}|_{\partial\cD}=\bu$ for the SCM-Coupled-EEV, and the initial condition $\bu_{j,h}^0=\bu(\bx,0)$, for $j=1,2,\cdots, J$, and run the simulations using the both SCM-Coupled-EEV and SCM-SPP-EEV methods until $T=20$ with the time-step size $\Delta t=0.1$, $\mu=1$, and $\gamma=$1e+4. We represent the approximate velocity (shown as speed) solution produced by the SCM-SPP-EEV method in Fig. \ref{SCM-SPP-EEV-vel-pres} (a) at time $t=1$.

To compare the SCM-SPP-EEV and SCM-Coupled-EEV methods, we plot their decaying Energy vs. Time graphs in Fig. \ref{SCM-SPP-EEV-vel-pres} (b) from their outcomes. For both methods, the energy at time $t=t^n$ is computed as the weighted average of $\frac12\|\hspace{-1mm}\lab\bu_h\rab^n(\bx,\by_j)\|^2$ for all stochastic collocation points. We observe an excellent agreement between the energy plots from the SCM-Coupled-EEV scheme's solution and the SCM-SPP-EEV method's solution, which supports the theory. %For the TGV problem, the effect of the EEV on the solution is found insignificant, thus we omit the Energy vs. Time plot for various $\mu$.
\begin{figure}
	\centering
	\subfloat[]{\includegraphics[width=0.4\textwidth,height=0.30\textwidth]{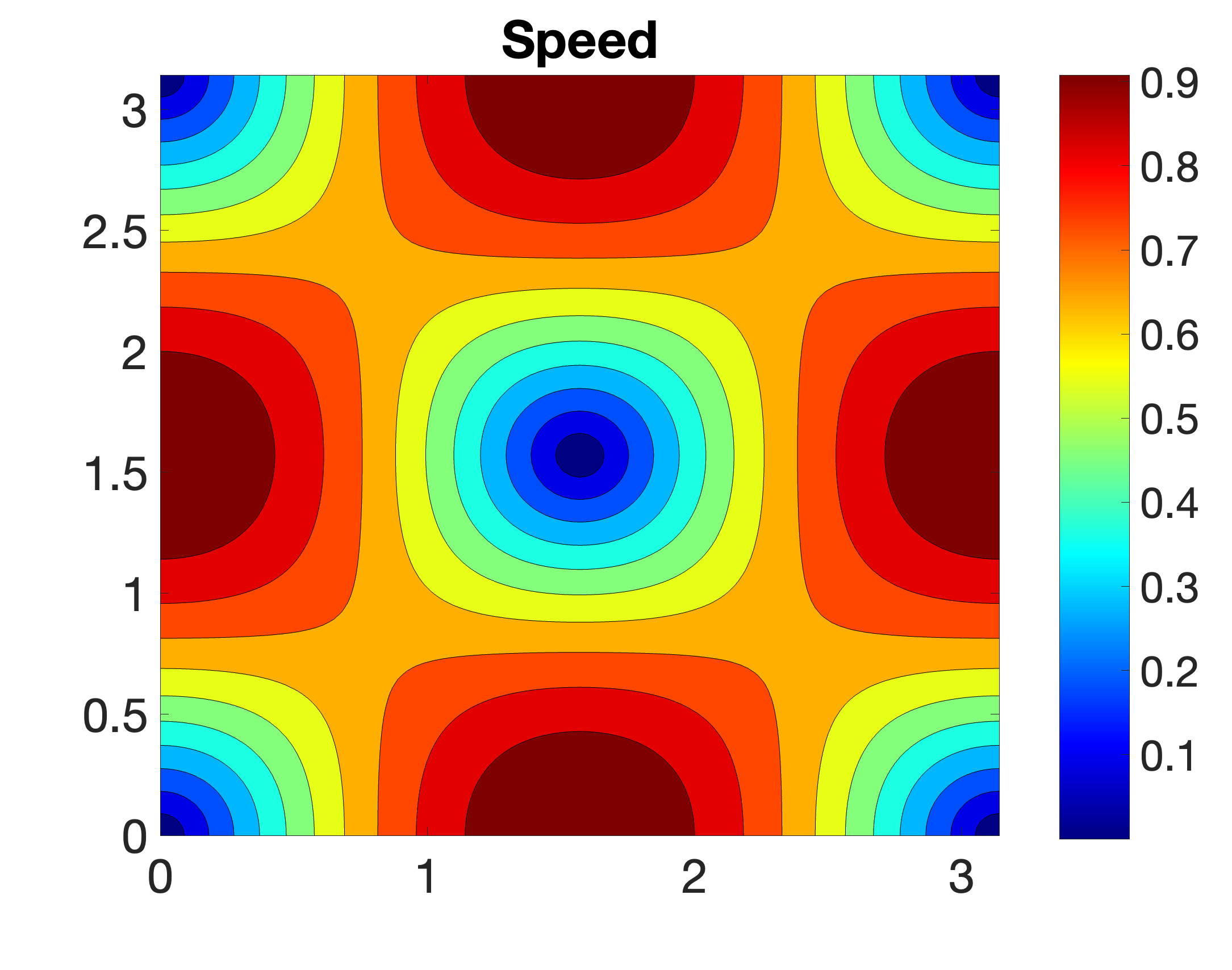}}
	\subfloat[]{\includegraphics[width=0.4\textwidth,height=0.30\textwidth]{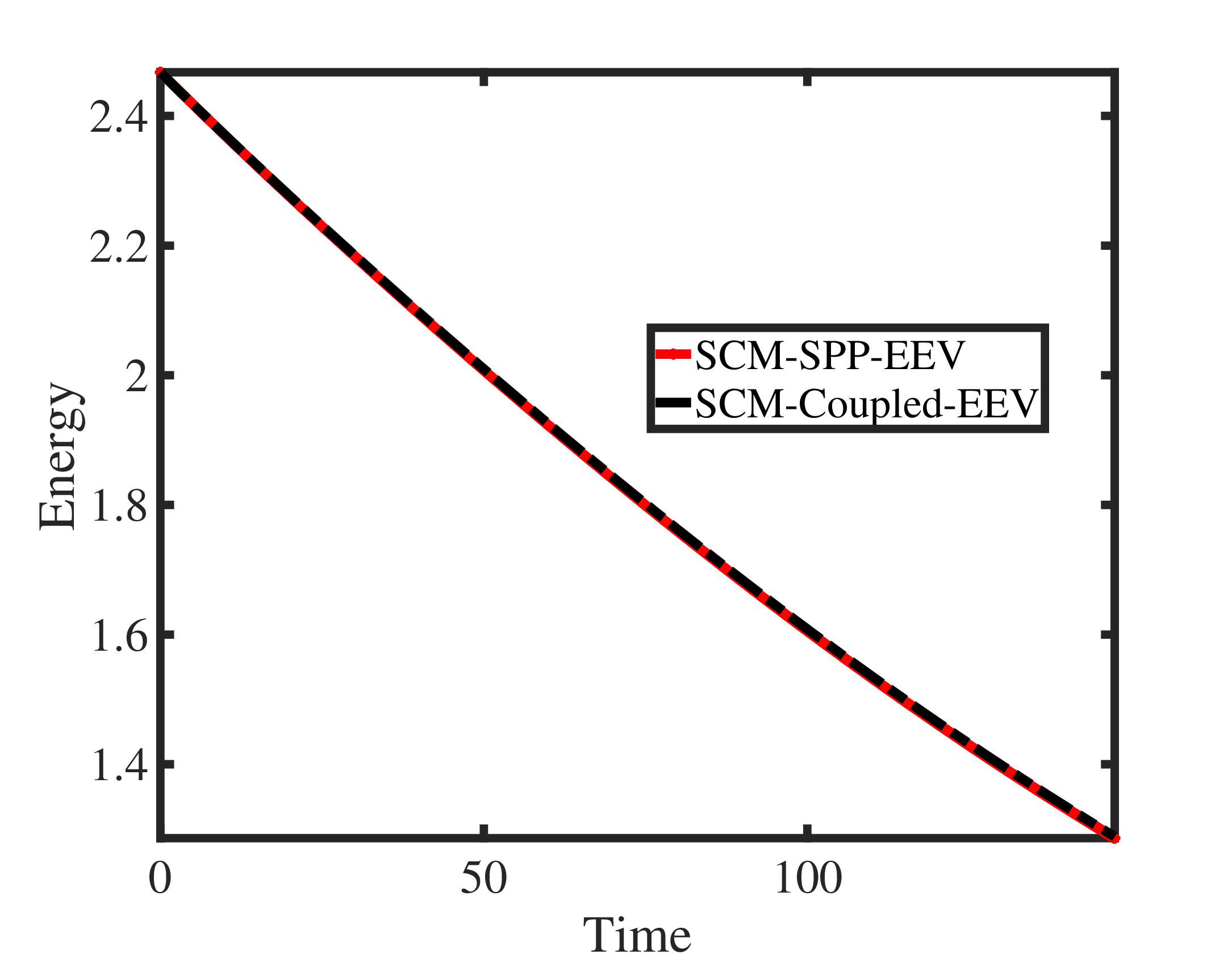}}\vspace{-4ex}
	\caption{Variable 5D random viscosity in TGV problem  for $\mathbb{E}[\nu]=0.001$: (a) Ensemble average of velocity (shown as speed contour) solution of SCM-SPP-EEV method at $t=1$, and (b) plot of Energy vs. Time for the both SCM-SPP-EEV and SCM-Coupled-EEV methods.}\vspace{-4ex}	\label{SCM-SPP-EEV-vel-pres}
\end{figure}

\subsection{Channel flow over a unit square step}  This experiment considers a benchmark channel flow over a unit square problem \cite{linke2017connection}. The dimension of the rectangular channel is $40\times 10$ unit$^2$, and the step is 5 units away from the inlet. The following parabolic noisy flow is considered
\begin{align*}
    \bu_{j,h}|_{inlet, outlet}=(1+k_j\epsilon){{\frac{x_2(x_2-10)}{25}}\choose{0}},
\end{align*} as inflow and outflow, where $k_j:=\frac{2j-1-J}{\lfloor \frac{J}{2}\rfloor }$, for $j=1,2,\cdots,J$, $J=11$, and $\epsilon=0.01$. No-slip boundary condition is applied to the domain walls and step for the SCM-Coupled-EEV scheme. In Step 1 of the SCM-SPP-EEV scheme, we enforce the no-slip boundary condition, and in Step 2, we set weakly, the normal velocity component vanishes on the boundary. The external force $\bif=\textbf{0}$ is considered. We start with the following initial condition
\begin{align*}
    \bu_{j,h}(\bx,0)=(1+k_j\epsilon){{\frac{x_2(x_2-10)}{25}}\choose{0}}.
\end{align*}
The random viscosity is modelled as: \begin{align}
\nu({\bx}, {\by_j})=\frac{1}{600}\psi({\bx}, {\by_j}),
\end{align}
with $\mathbb{E}[\nu](\bx)=\frac{c}{600}$, $L=40$, $y_{j}(\omega)\in[-\sqrt{3},\sqrt{3}]$, $l=0.01$, $d=5$, $c=1$, and $q=2$. The time-step size $\Delta t=0.1$, $\gamma=$ 1e+4, $\mu=1$, and run the simulation until $T=40$ using the both methods independently. The flow shows a recirculating vortex detaches behind the step \cite{layton2008numerical} as in Fig. \ref{channel-comparison} (a), which is the outcomes of the SCM-SPP-EEV method. To compare the SCM-Coupled-EEV and SCM-SPP-EEV methods, we plot the Energy vs. Time graphs in Fig. \ref{channel-comparison} (b) and found an excellent agreement between them.
\begin{figure} [ht]
		\centering	
  \subfloat[]{\includegraphics[width=0.49\textwidth,height=0.17\textwidth]{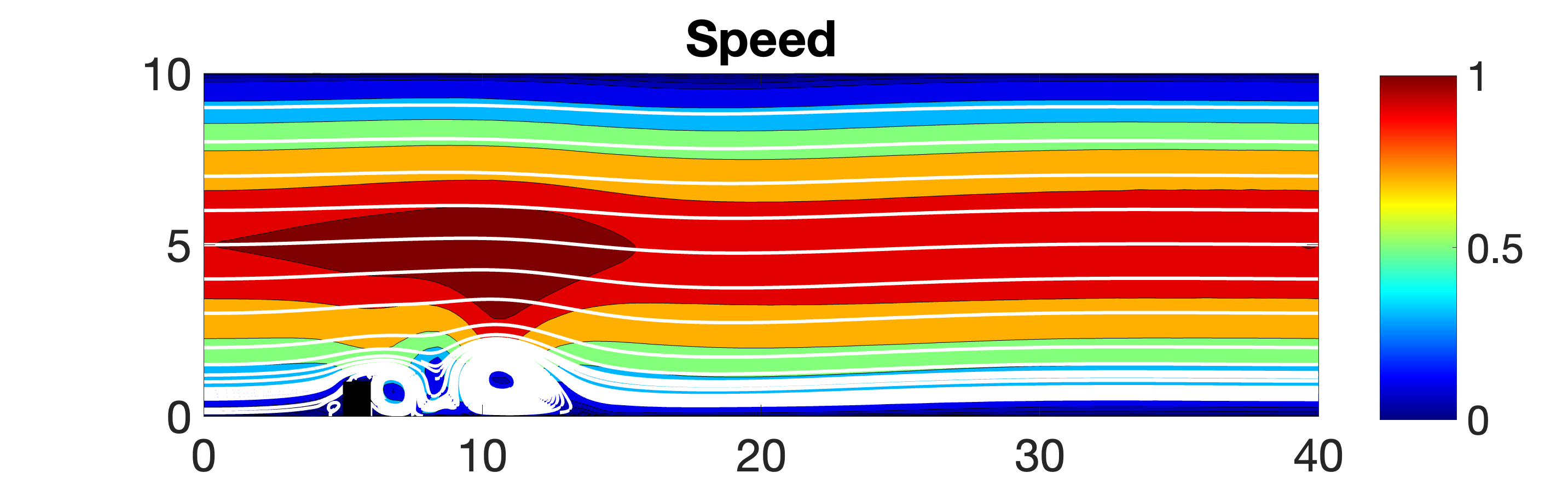}}	
  \subfloat[]{\includegraphics[width=0.3\textwidth,height=0.19\textwidth]{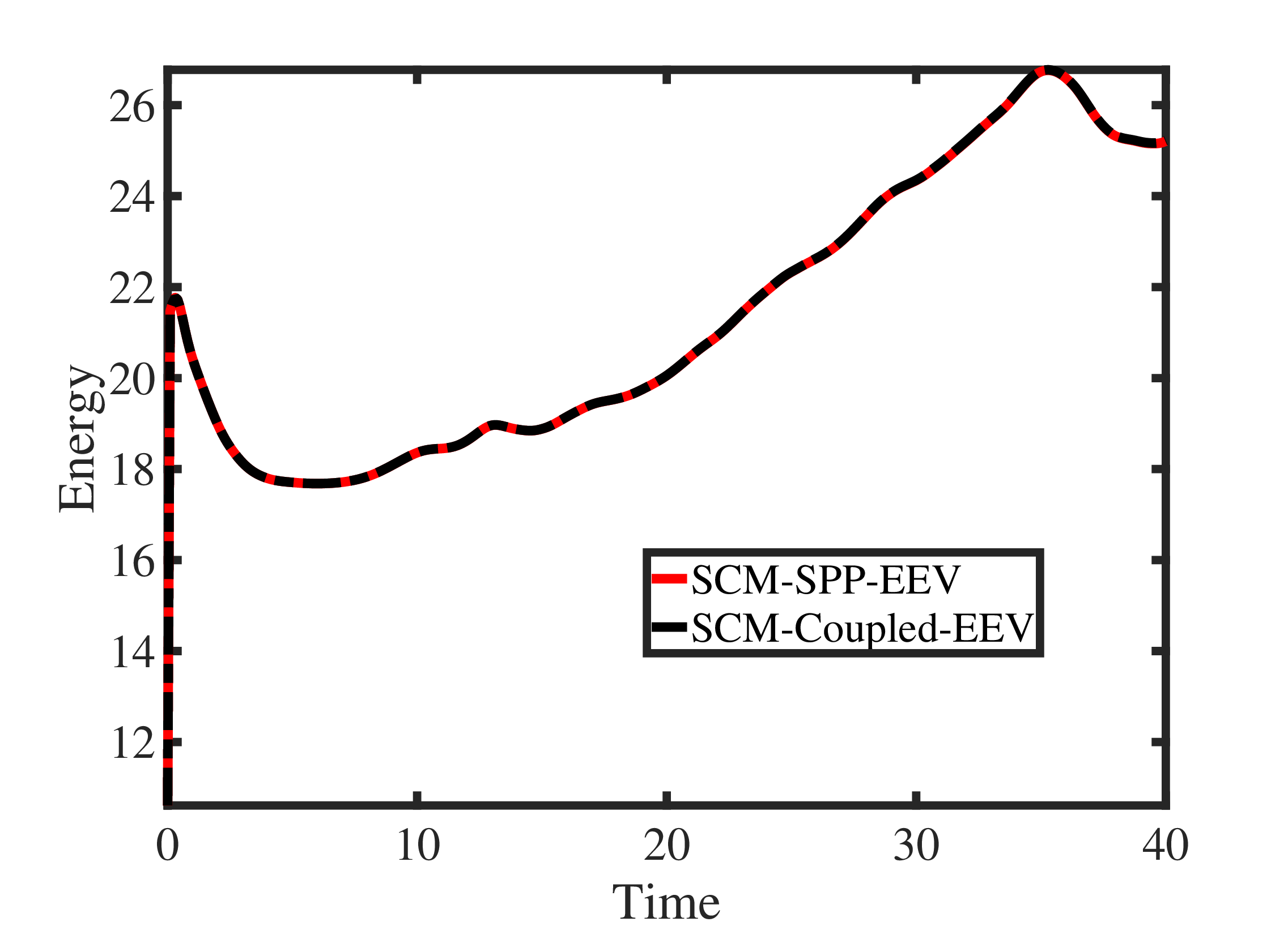}}
		\caption{\footnotesize{Variable 5D random viscosity in a flow over a step problem: (a) Ensemble average of velocity solution (shown as streamlines over the speed contour) of SCM-SPP-EEV method at $t=40$, (b) plot of Energy vs. Time for the both SCM-SPP-EEV and SCM-Coupled-EEV methods.}}\label{channel-comparison}
	\end{figure}

\subsection{Regularized Lid-driven Cavity (RLDC) Problem}\label{RLDC}
We now consider a 2D benchmark regularized lid-driven cavity problem \cite{balajewicz2013low,fick2018stabilized,lee2019study, mohebujjaman2022efficient} with a domain $\cD=(-1,1)^2$. No-slip boundary conditions are applied to all sides except on the top wall (lid) of the cavity where we impose the following noise involved boundary condition:
\begin{align*}
\bu_{j,h}|_{lid}=\left (1+k_j\epsilon\right){{(1-x_1^2)^2}\choose{0}},
\end{align*}
so that the velocity of the boundary preserve the continuity. In this case, we model the random viscosity as:
\begin{align}
\nu({\bx}, {\by_j})=\frac{2}{15000}\psi({\bx}, {\by_j}),
\end{align}
with $\mathbb{E}[\nu](\bx)=\frac{2c}{15000}$, $L=2$, $y_{j}(\omega)\in[-\sqrt{3},\sqrt{3}]$, $l=0.01$, $d=5$, $c=1$, and $q=2$. We conducted the simulation with an end time $T=600$ and a step size $\Delta t=5$. We considered the external force $\bif=\textbf{0}$. A perturbation parameter $\epsilon=0.01$ was applied in the boundary condition. The eddy-viscosity coefficient $\mu=1$ were considered. The unstructured triangular mesh used had 364,920 dof. We represent the velocity solution of the SCM-SPP-EEV method (with $\gamma=$ 1e+4) in Fig. \ref{RLDC_energy_curves} (a) at $t=600$ and the plot of Energy vs. Time of both the SCM-Coupled-EEV, and SCM-SPP-EEV methods in Fig. \ref{RLDC_energy_curves} (b). We observe an excellent agreement of the energy plots over the time period [0, 600].
\begin{figure}[h!] 
	\begin{center}    
          \subfloat[]
		{\includegraphics[width = 0.4\textwidth, height=0.30\textwidth]{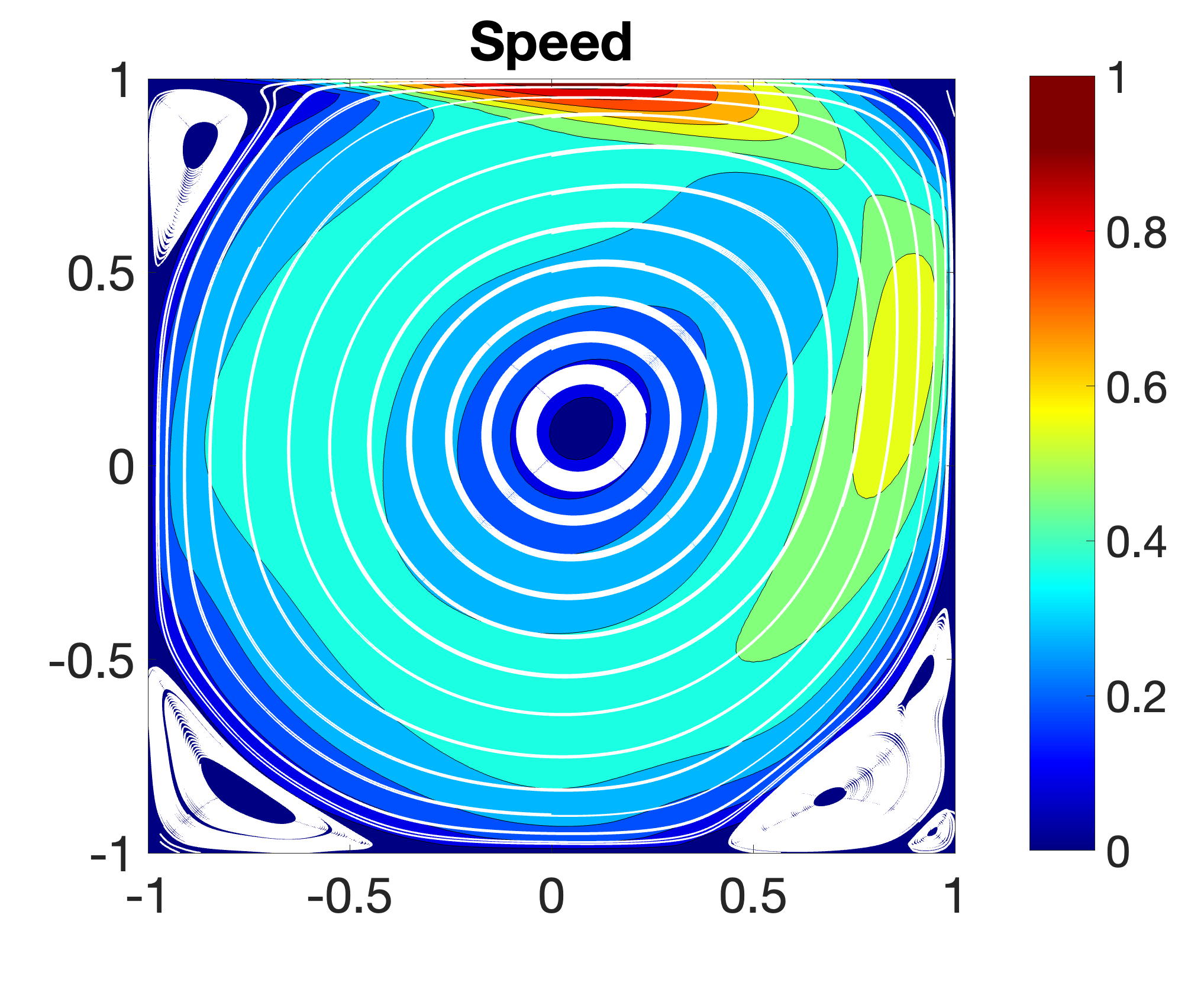}} 
  \subfloat[]{
        \includegraphics[width = 0.4\textwidth, height=0.30\textwidth]{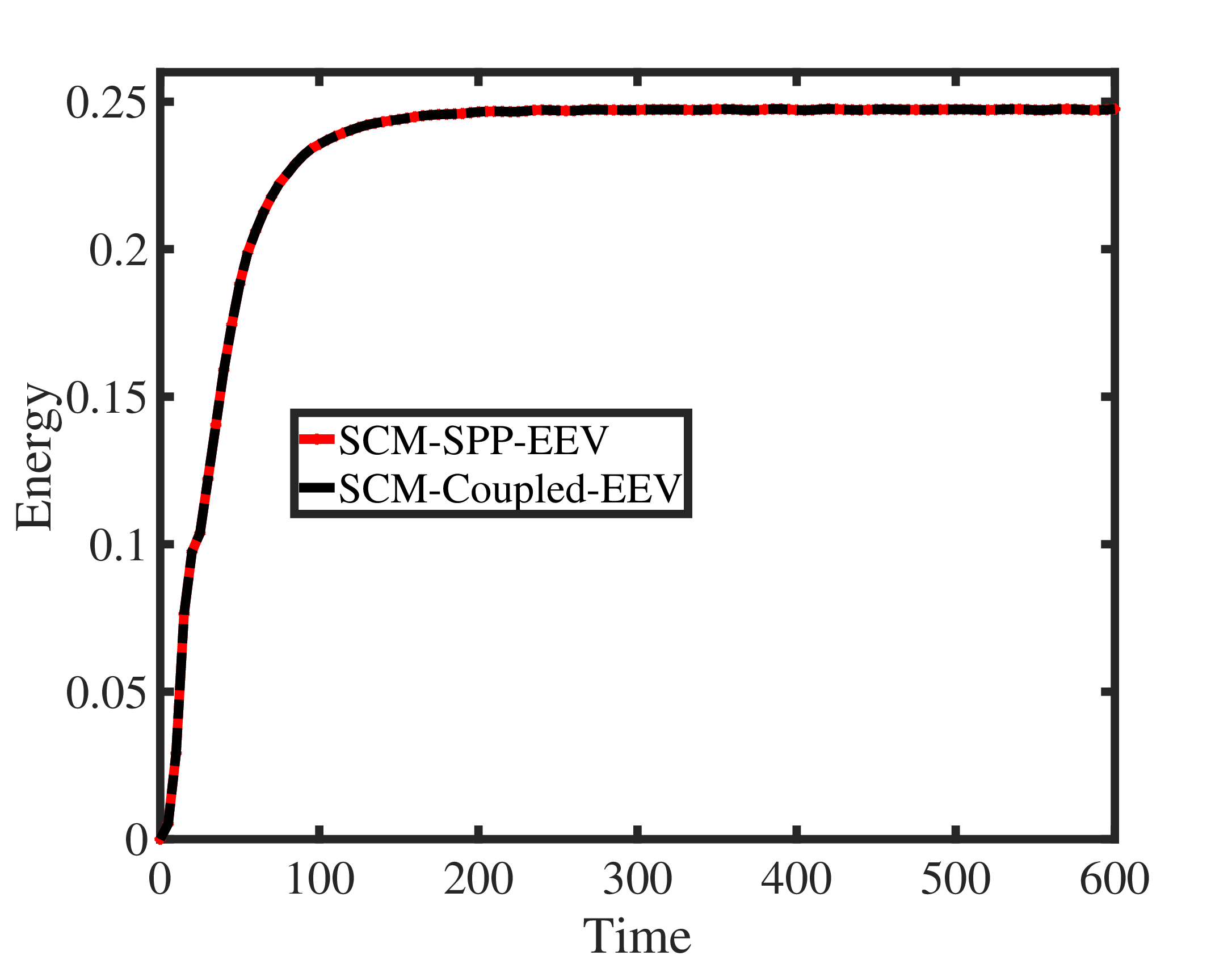}}
	\end{center}
	\caption{Variable 5D random viscosity in a RLDC problem with $\mathbb{E}[Re]=15,000$: (a) Velocity solution (shown as streamlines over the speed contour) of SCM-SPP-EEV method at $t=600$, (b) Energy vs. Time plot for both Coupled-EEV, and SCM-SPP-EEV (with $\gamma=$ 1e+4) methods.}\label{RLDC_energy_curves}
\end{figure}

\subsection{Effect of EEV on Convection Dominated Problem}

The EEV based algorithms for highly ill-conditioned complex problems, e.g., RLDC problem with high Reynolds number, are more stable than those without it \cite{jiang2015numerical, mohebujjaman2024efficient}. To observe this, we consider the RLDC problem discussed in Section \ref{RLDC} with the same continuous and discrete model input data. We plot the Energy vs. Time graphs in Fig. \ref{RLDC-energy:mu-varies} using SCM-Coupled-EEV scheme for several values of the coefficient $\mu$ of the EEV term starting from $\mu=0$ (which is the case for without EEV algorithms). 

It is observed that the flow ensemble algorithm without EEV (setting $\mu=0$ in the SCM-Coupled-EEV scheme) blows up at around $t=90$, however, for $\mu>0$, flows remain stable until $T=600$. We also notice that as $\mu$ grows, in this case, the solution converges to the case $\mu=1$. Therefore, among the ensemble algorithms for the parameterized stochastic complex flow problems, the EEV based schemes outperform. The penalty-projection algorithm that stabilizes with EEV and grad-div terms performs better than the coupled schemes.

\begin{figure}[h!] 
	\begin{center}
		\textbf{}\par\medskip            
		\includegraphics[width = 0.4\textwidth, height=0.30\textwidth]{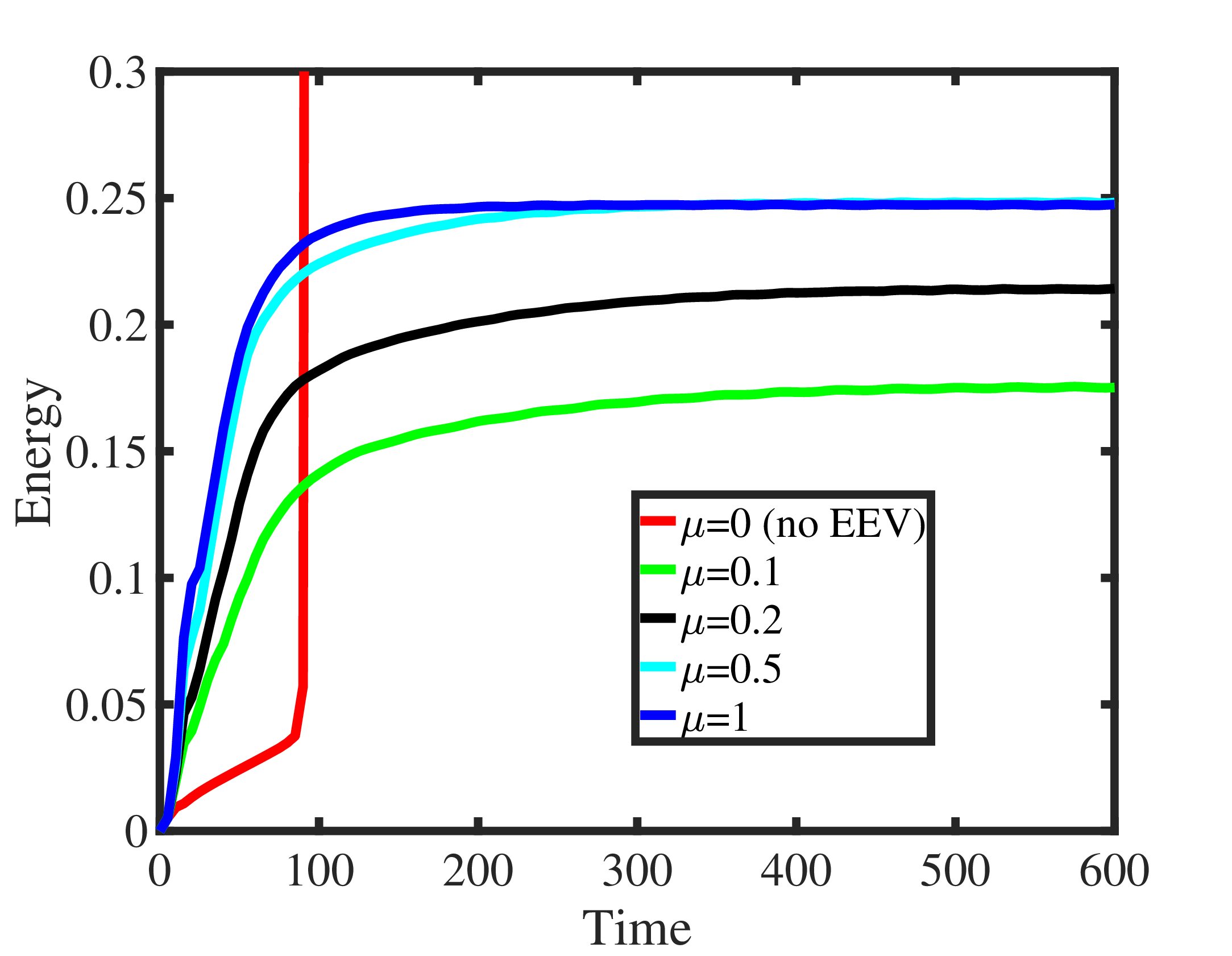}
	\end{center}
	\caption{Variable 5D random viscosity in a RLDC problem with $\mathbb{E}[Re]=15,000$: Energy vs. Time as $\mu$ varies. Solution blows up for $\mu=0$.}\label{RLDC-energy:mu-varies}
\end{figure}

\section{Conclusion}
In this work, we propose and analyze a linear extrapolated fully discrete efficient coupled EEV-based timestepping algorithm for the uncertainty quantification of SNSEs flow problems. We proved the stability and convergence of the coupled scheme rigorously. We then connect this coupled scheme to a more efficient penalty-projection and EEV-based ensemble timestepping algorithm. We proved that the penalty-projection based algorithm is stable and converges to the coupled scheme for large grad-div stabilization penalty parameter values. In future work, we plan to propose a second-order temporal accurate efficient penalty-projection timestepping algorithm for the UQ of NSE flow problems and analyze and test the schemes on benchmark problems.
\appendix
\section{Stability Proof of BESCoupled Algorithm}\label{appendix-stability}
\begin{proof}
	Taking $\bchi_{h}=\bu_{j,h}^{n+1}\in\bX_h$ and $q_{h}=p^{n+1}_{j,h}\in Q_h$ in \eqref{couple-eqn-1} and \eqref{couple-incompressibility}, respectively, using $b^*\big(\hspace{-1mm}<\bu_h>^n, \bu_{j,h}^{n+1},\bu_{j,h}^{n+1}\big)=0$, we obtain
	\begin{align}
		&\Big(\frac{\bu_{j,h}^{n+1}-\bu_{j,h}^n}{\Delta t},\bu_{j,h}^{n+1}\Big)+\|\Bar{\nu}^{\frac12}\nabla \bu_{j,h}^{n+1}\|^2+\left(2\nu_T(\bu^{'}_{h},t^n)\nabla \bu_{j,h}^{n+1},\nabla \bu_{j,h}^{n+1}\right)\nonumber\\&= (\bif_{j}(t^{n+1}),\bu_{j,h}^{n+1})-(\bu_{j,h}^{'n}\cdot\nabla \bu_{j,h}^n,\bu_{j,h}^{n+1})-\nu_j^{'}(\nabla \bu_{j,h}^{n},\nabla\bu_{j,h}^{n+1}).
	\end{align}
	Using polarization identity and $\left(2\nu_T(\bu^{'}_{h},t^n)\nabla \bu_{j,h}^{n+1},\nabla \bu_{j,h}^{n+1}\right)=2\mu\Delta t\|l^n\nabla \bu_{j,h}^{n+1}\|^2$, we have
	\begin{align}
		&\frac{1}{2\Delta t}\Big(\|\bu_{j,h}^{n+1}\|^2-\|\bu_{j,h}^n\|^2+\|\bu_{j,h}^{n+1}-\bu_{j,h}^n\|^2\Big)+\|\Bar{\nu}^\frac12\nabla \bu_{j,h}^{n+1}\|^2+2\mu\Delta t\|l^n\nabla \bu_{j,h}^{n+1}\|^2\nonumber\\&= (\bif_{j}(t^{n+1}),\bu_{j,h}^{n+1})-b^*(\bu_{j,h}^{'n}, \bu_{j,h}^n,\bu_{j,h}^{n+1})-(\nu_j^{'}\nabla \bu_{j,h}^{n},\nabla\bu_{j,h}^{n+1}).\label{before-nl-bound}
	\end{align}
  Applying Cauchy-Schwarz and Young's inequalities on the forcing term, yields
  \begin{align*}
      (\bif_{j}(t^{n+1}),\bu_{j,h}^{n+1})\le\|\bif_{j}(t^{n+1})\|_{-1}\|\nabla\bu_{j,h}^{n+1}\|\le \frac{\alpha_j}{4}\|\nabla\bu_{j,h}^{n+1}\|^2+\frac{1}{\alpha_j}\|\bif_{j}(t^{n+1})\|_{-1}^2.
  \end{align*}
  We rewrite the trilinear form in \eqref{before-nl-bound}, use identity \eqref{trilinear-identitiy}, Cauchy-Schwarz, H\"older's, Poincar\'e, and \eqref{basic-ineq}, \eqref{eddy-viscosity2}, and Young's inequalities, to have
\begin{align*}
b^*(\bu_{j,h}^{'n}, &\bu_{j,h}^n,\bu_{j,h}^{n+1})=b^*(\bu_{j,h}^{'n},\bu_{j,h}^{n+1} ,\bu_{j,h}^{n+1}-\bu_{j,h}^n)\nonumber\\&=(\bu_{j,h}^{'n}\cdot\nabla\bu_{j,h}^{n+1} ,\bu_{j,h}^{n+1}-\bu_{j,h}^n)+\frac12\left(\nabla\cdot\bu_{j,h}^{'n},\bu_{j,h}^{n+1}\cdot(\bu_{j,h}^{n+1}-\bu_{j,h}^n)\right)\nonumber\\
   &\le\|\bu_{j,h}^{'n}\cdot\nabla
\bu_{j,h}^{n+1}\|\|\bu_{j,h}^{n+1}-\bu_{j,h}^n\|+\frac12\|\nabla\cdot\bu_{j,h}^{'n}\|_{L^\infty}\|\bu_{j,h}^{n+1}\|\|\bu_{j,h}^{n+1}-\bu_{j,h}^n\|\nonumber\\&\le \||\bu_{j,h}^{'n}|\nabla
\bu_{j,h}^{n+1}\|\|\bu_{j,h}^{n+1}-\bu_{j,h}^n\|+C\|\nabla\cdot\bu_{j,h}^{'n}\|_{L^\infty}\|\nabla\bu_{j,h}^{n+1}\|\|\bu_{j,h}^{n+1}-\bu_{j,h}^n\|\nonumber\\   &\le\frac{\alpha_j}{4}\|\nabla\bhu_{j,h}^{n+1}\|^2+\|l^{n}\nabla 			\bhu_{j,h}^{n+1}\|\|\bu_{j,h}^{n+1}-\bu_{j,h}^n\|+\frac{C}{\alpha_j}\|\nabla\cdot\bu_{j,h}^{'n}\|_{L^\infty}^2\|\bu_{j,h}^{n+1}-\bu_{j,h}^n\|^2\nonumber\\&\le\frac{\alpha_j}{4}\|\nabla\bu_{j,h}^{n+1}\|^2+\Delta t\|l^{n}\nabla 			\bu_{j,h}^{n+1}\|^2+\left(\frac{1}{4\Delta t}+\frac{C}{\alpha_j}\|\nabla\cdot\bu_{j,h}^{'n}\|_{L^\infty}^2\right)\|\bu_{j,h}^{n+1}-\bu_{j,h}^n\|^2.
  \end{align*}
Use of H\"older's and Young's inequalities, we have
\begin{align*}
    -(\nu_j^{'}\nabla \bu_{j,h}^{n},\nabla\bu_{j,h}^{n+1})\le\|\nu_j^{'}\|_{\infty}\|\nabla \bu_{j,h}^{n}\|\|\nabla\bu_{j,h}^{n+1}\|\le\frac{\|\nu_j^{'}\|_{\infty}}{2}\|\nabla \bu_{j,h}^{n}\|^2+\frac{\|\nu_j^{'}\|_{\infty}}{2}\|\nabla \bu_{j,h}^{n+1}\|^2.
\end{align*}
Using the above bounds, and reducing the equation \eqref{before-nl-bound}, becomes
	\begin{align}
		&\frac{1}{2\Delta t}\Big(\|\bu_{j,h}^{n+1}\|^2-\|\bu_{j,h}^n\|^2\Big)+\left(\frac{1}{4\Delta t}-\frac{C}{\alpha_j}\|\nabla\cdot\bu_{j,h}^{'n}\|_{L^\infty}^2\right)\|\bu_{j,h}^{n+1}-\bu_{j,h}^n\|^2+\frac{\Bar{\nu}_{\min}}{2}\|\nabla \bu_{j,h}^{n+1}\|^2\nonumber\\&+(2\mu-1)\Delta t\|l^n\nabla \bu_{j,h}^{n+1}\|^2\le \frac{1}{\alpha_j}\|\bif_{j}(t^{n+1})\|_{-1}^2+\frac{\|\nu_j^{'}\|_{\infty}}{2}\|\nabla \bu_{j,h}^{n}\|^2.\label{before-young}
	\end{align}
 Choose the calibration constant $\mu>\frac12$, set the time-step restriction \begin{align}
     \Delta t<\frac{C\alpha_j}{\|\nabla\cdot\bu_{j,h}^{'n}\|^2_{L^\infty}},
 \end{align}
 dropping non-negative term from left-hand-side, and rearranging
 \begin{align}
		&\frac{1}{2\Delta t}\Big(\|\bu_{j,h}^{n+1}\|^2-\|\bu_{j,h}^n\|^2\Big)+\frac{\Bar{\nu}_{\min}}{2}\left(\|\nabla \bu_{j,h}^{n+1}\|^2-\|\nabla \bu_{j,h}^{n}\|^2\right)+\frac{\alpha_j}{2}\|\nabla \bu_{j,h}^{n}\|^2\le \frac{1}{\alpha_j}\|\bif_{j}(t^{n+1})\|_{-1}^2.
	\end{align}
Multiply both sides by $2\Delta t$, and sum over the time-steps $n=0,1,\cdots,M-1$, which will finish the proof.
\end{proof}
\section{Convergence Proof of BESCoupled Algorithm}\label{appendix-convergence}
\begin{proof}
	We start our proof by forming the error equation. Testing \eqref{gov1}-\eqref{gov2} at the time level $t^{n+1}$, the continuous variational formulations can be written
	as
	\begin{align}
		\bigg(&\frac{\bu_{j}(t^{n+1})-\bu_{j}(t^n)}{\Delta t},\bchi_{h}\bigg)+\Big(\Bar{\nu}\nabla \bu_{j}(t^{n+1}),\nabla\bchi_{h}\Big)+\Big(\bu_j(t^{n+1})\cdot\nabla \bu_{j}(t^{n+1}),\bchi_{h}\Big)\nonumber\\&= \Big(\bif_{j}(t^{n+1}),\bchi_{h}\Big)-\Big(\nu_j^{'}\nabla \bu_{j}(t^{n+1}),\nabla\bchi_{h}\Big)-\bigg(\bu_{j,t}-\frac{\bu_{j}(t^{n+1})-\bu_{j}(t^n)}{\Delta t},\bchi_{h}\bigg).\label{before-error-eqn-regular}
	\end{align}
	Denote $\be_j^n:=\bu_j(t^n)-\bu_{j,h}^n.$ Subtract \eqref{couple-eqn-1} from \eqref{before-error-eqn-regular}, gives
	\begin{align}
		&\bigg(\frac{\be_{j}^{n+1}-\be_{j}^n}{\Delta t},\bchi_{h}\bigg)+\big(\Bar{\nu}\nabla \be_{j}^{n+1},\nabla\bchi_{h}\big)+\left(\nu_j^{'}\nabla\be_{j}^n,\nabla\bchi_{h}\right)+b^*\left(<\be>^n,\bu_j(t^{n+1})-\bu_{j}(t^n),\bchi_{h}\right)\nonumber\\&+b^*\big(\hspace{-1.4mm}<\bu_h>^n,\be_j^{n+1},\bchi_{h}\big)+b^*\left(\bu_{j,h}^{'n}, \be_{j}^n,\bchi_{h}\right)+b^*\big(\be_j^n,\bu_{j}(t^n),\bchi_{h}\big)+\left(2\mu\Delta t(l^n)^2\nabla \be_{j}^{n+1},\nabla \bchi_{h}\right)\nonumber\\&-\left(2\mu\Delta t(l^n)^2\nabla \bu_{j}(t^{n+1}),\nabla \bchi_{h}\right) =-G(t,\bu_j,\bchi_{h}),
	\end{align}
	where 
	\begin{align}
		&G(t,\bu_j,\bchi_{h}):=\bigg(\bu_{j,t}-\frac{\bu_{j}(t^{n+1})-\bu_{j}(t^n)}{\Delta t},\bchi_{h}\bigg)+\left((\bu_j(t^{n+1})-\bu_j(t^{n}))\cdot\nabla \bu_{j}(t^{n+1}),\bchi_{h}\right)\nonumber\\&+\left((\bu_j(t^n)-<\bu>(t^n))\cdot\nabla( \bu_{j}(t^{n+1})- \bu_{j}(t^{n})),\bchi_{h}\right)+\left(\nu_j^{'}\nabla(\bu_{j}(t^{n+1})-\bu_{j}(t^n)),\nabla\bchi_{h}\right).
	\end{align}
	Now, we decompose the error as the interpolation error and approximation term:
	\begin{align*}
		\be_{j}^n:& = \bu_j(t^n)-\bu_{j,h}^n=(\bu_j(t^n)-\tilde{\bu}_j^n)-(\bu_{j,h}^n-\tilde{\bu}_j^n):=\bfeta_{j}^n-\bphi_{j,h}^n,
	\end{align*}
	where $\tilde{\bu}_j^n: =P_{\bX_h}^{L^2}(\bu_j(t^n))\in \bX_h$ is the $L^2$ projections of $\bu_j(t^n)$ into $\bX_h$. Note that $(\bfeta_{j}^n,\bv_{j,h})=0\hspace{2mm} \forall \bv_{j,h}\in \bX_h,$  
	we then have
	\begin{align}
&\bigg(\frac{\bphi_{j,h}^{n+1}-\bphi_{j,h}^n}{\Delta t},\bchi_{h}\bigg)+\big(\Bar{\nu}\nabla \bphi_{j,h}^{n+1},\nabla\bchi_{h}\big)+\left(\nu_j^{'}\nabla\bphi_{j,h}^n,\nabla\bchi_{h}\right)+b^*\big(\hspace{-1.4mm}<\bu_h>^n,\bphi_{j,h}^{n+1},\bchi_{h}\big)\nonumber\\&+b^*\left(<\bphi_h>^n,\bu_j(t^{n+1})-\bu_{j}(t^n),\bchi_{h}\right)+b^*\left(\bu_{j,h}^{'n}, \bphi_{j,h}^n,\bchi_{h}\right)+b^*\big(\bphi_{j,h}^n,\bu_{j}(t^n),\bchi_{h}\big)\nonumber\\&=\left(2\mu\Delta t(l^n)^2\nabla \bu_{j}(t^{n+1}),\nabla \bchi_{h}\right)-\left(2\mu\Delta t(l^n)^2\nabla \bphi_{j,h}^{n+1},\nabla \bchi_{h}\right) +b^*\big(\hspace{-1.4mm}<\bu_h>^n,\bfeta_j^{n+1},\bchi_{h}\big)\nonumber\\&+b^*\left(<\bfeta>^n,\bu_j(t^{n+1})-\bu_{j}(t^n),\bchi_{h}\right)+b^*\left(\bu_{j,h}^{'n}, \bfeta_{j}^n,\bchi_{h}\right)+b^*\big(\bfeta_j^n,\bu_{j}(t^n),\bchi_{h}\big)\nonumber\\&+\big(\Bar{\nu}\nabla \bfeta_{j}^{n+1},\nabla\bchi_{h}\big)+\left(\nu_j^{'}\nabla\bfeta_{j}^n,\nabla\bchi_{h}\right)+\left(2\mu\Delta t(l^n)^2\nabla \bfeta_{j}^{n+1},\nabla \bchi_{h}\right)+G(t,\bu_j,\bchi_{h}).\label{phi-equn}
	\end{align}
	Choose $\bchi_{h}=\bphi_{j,h}^{n+1}$, use the polarization identity in \eqref{phi-equn}, and rearrange, to get
	\begin{align}
		&\frac{1}{2\Delta t}\lp\|\bphi_{j,h}^{n+1}\|^2-\|\bphi_{j,h}^{n}\|^2+\|\bphi_{j,h}^{n+1}-\bphi_{j,h}^{n}\|^2\rp+\big\|\Bar{\nu}^{\frac{1}{2}}\nabla\bphi_{j,h}^{n+1}\|^2+2\mu\Delta t\|l^n\nabla \bphi_{j,h}^{n+1}\|^2\nonumber\\&=-\left(\nu_j^{'}\nabla\bphi_{j,h}^n,\nabla\bphi_{j,h}^{n+1}\right)+\big(\Bar{\nu}\nabla \bfeta_{j}^{n+1},\nabla\bphi_{j,h}^{n+1}\big)+\left(\nu_j^{'}\nabla\bfeta_{j}^n,\nabla\bphi_{j,h}^{n+1}\right)-b^*\left(\bu_{j,h}^{'n}, \bphi_{j,h}^n,\bphi_{j,h}^{n+1}\right)\nonumber\\&-b^*\left(<\bphi_h>^n,\bu_j(t^{n+1})-\bu_{j}(t^n),\bphi_{j,h}^{n+1}\right)+b^*\left(<\bfeta>^n,\bu_j(t^{n+1})-\bu_{j}(t^n),\bphi_{j,h}^{n+1}\right)\nonumber\\&-b^*\big(\bphi_{j,h}^n,\bu_{j}(t^n),\bphi_{j,h}^{n+1}\big)+b^*\big(\hspace{-1.4mm}<\bu_h>^n,\bfeta_j^{n+1},\bphi_{j,h}^{n+1}\big)+b^*\left(\bu_{j,h}^{'n}, \bfeta_{j}^n,\bphi_{j,h}^{n+1}\right)+b^*\big(\bfeta_j^n,\bu_{j}(t^n),\bphi_{j,h}^{n+1}\big)\nonumber\\&+\left(2\mu\Delta t(l^n)^2\nabla \bu_{j}(t^{n+1}),\nabla \bphi_{j,h}^{n+1}\right)+\left(2\mu\Delta t(l^n)^2\nabla \bfeta_{j}^{n+1},\nabla\bphi_{j,h}^{n+1}\right)+G\left(t,\bu_j,\bphi_{j,h}^{n+1}\right).\label{phibd}
	\end{align}
	Note that $b^*\big(\hspace{-1.4mm}<\bu_h>^n,\bphi_{j,h}^{n+1},\bphi_{j,h}^{n+1}\big)=0$. Now, turn our attention to finding bounds on the right side terms of \eqref{phibd}. Apply H\"older's and Young’s inequalities to obtain the following bounds
	\begin{align*}
		-\left(\nu_j^{'}\nabla\bphi_{j,h}^n,\nabla\bphi_{j,h}^{n+1}\right)\le\|\nu_j^{'}\|_\infty\|\nabla\bphi_{j,h}^n\|\|\nabla\bphi_{j,h}^{n+1}\|\le\frac{\|\nu_j^{'}\|_\infty}{2}\|\nabla\bphi_{j,h}^{n+1}\|^2+\frac{\|\nu_j^{'}\|_\infty}{2}\|\nabla\bphi_{j,h}^n\|^2,\\
		\big(\Bar{\nu}\nabla \bfeta_{j}^{n+1},\nabla\bphi_{j,h}^{n+1}\big)\le\|\Bar{\nu}\|_\infty\|\nabla\bfeta_{j}^{n+1}\|\|\nabla\bphi_{j,h}^{n+1}\|\le\frac{\alpha_j}{22}\|\nabla\bphi_{j,h}^{n+1}\|^2+\frac{11\|\Bar{\nu}\|_\infty^2}{2\alpha_j}\|\nabla \bfeta_{j}^{n+1}\|^2,\\\left(\nu_j^{'}\nabla\bfeta_{j}^n,\nabla\bphi_{j,h}^{n+1}\right)\le\|\nu_j^{'}\|_\infty\|\bfeta_{j}^n\|\|\nabla\bphi_{j,h}^{n+1}\|\le\frac{\alpha_j}{22}\|\nabla\bphi_{j,h}^{n+1}\|^2+\frac{11\|\nu_j^{'}\|_\infty^2}{2\alpha_j}\|\nabla \bfeta_{j}^{n+1}\|^2.
	\end{align*}
	For the first trilinear form, we rearrange, apply the identity \eqref{trilinear-identitiy}, Cauchy-Schwarz, H\"older's, Poincar\'e,  \eqref{basic-ineq}, and Young's inequalities, to have
	\begin{align}
		&b^*\lp \bu_{j,h}^{'n},\bphi_{j,h} ^n,\bphi_{j,h}^{n+1}\rp=b^*\lp \bu_{j,h}^{'n},\bphi_{j,h} ^{n+1},\bphi_{j,h}^{n+1}-\bphi_{j,h}^{n}\rp\nonumber\\&=\lp \bu_{j,h}^{'n}\cdot\nabla\bphi_{j,h} ^{n+1},\bphi_{j,h}^{n+1}-\bphi_{j,h}^{n}\rp+\frac12\lp\nabla\cdot \bu_{j,h}^{'n},\bphi_{j,h} ^{n+1}\cdot\left(\bphi_{j,h}^{n+1}-\bphi_{j,h}^{n}\right)\rp\nonumber\\&\le\|\bu_{j,h}^{'n}\cdot \nabla\bphi_{j,h} ^{n+1}\|\|\bphi_{j,h}^{n+1}-\bphi_{j,h}^{n}\|+\frac12\|\nabla\cdot \bu_{j,h}^{'n}\|_{L^\infty}\|\bphi_{j,h} ^{n+1}\|\|\bphi_{j,h}^{n+1}-\bphi_{j,h}^{n}\|\nonumber\\&\le \|l^n \nabla\bphi_{j,h} ^{n+1}\|\|\bphi_{j,h}^{n+1}-\bphi_{j,h}^{n}\|+C\|\nabla\cdot \bu_{j,h}^{'n}\|_{L^\infty}\|\nabla\bphi_{j,h} ^{n+1}\|\|\bphi_{j,h}^{n+1}-\bphi_{j,h}^{n}\|\nonumber\\&\le\frac{\alpha_j}{22}\|\nabla\bphi_{j,h}^{n+1}\|^2+\Delta t\|l^n \nabla\bphi_{j,h} ^{n+1}\|^2+\left(\frac{1}{4\Delta t}+\frac{C}{\alpha_j}\|\nabla\cdot\bu_{j,h}^{'n}\|_{L^\infty}^2\right)\|\bphi_{j,h}^{n+1}-\bphi_{j,h}^{n}\|^2.\label{fluc-bound-first-con}
	\end{align}
 For the second trilinear term, use the bound in \eqref{nonlinearbound3}, triangle inequality, Agmon’s \cite{Robinson2016Three-Dimensional} inequality, Sobolev embedding theorem, regularity assumption of the true solution $\bu_j\in L^\infty(0,T;\bH^3(\cD))$, and Young’s inequalities, to obtain
 \begin{align*}
-&b^*\Big(\hspace{-1mm}<\hspace{-1mm}\bphi_{h}\hspace{-1mm}>^n,\bu_j(t^{n+1})-\bu_j(t^{n}),\bphi_{j,h}^{n+1}\Big)\\&\le  C\|<\hspace{-1mm}\bphi_{h}\hspace{-1mm}>^n\|\left(\|\nabla\left(\bu_j(t^{n+1})-\bu_j(t^{n})\right)\|_{L^3}+\|\bu_j(t^{n+1})-\bu_j(t^{n})\|_{L^\infty}\right)\|\nabla \bphi_{j,h}^{n+1}\|^2  \\&\le    \frac{\alpha_j}{22}\|\nabla \bphi_{j,h}^{n+1}\|^2+\frac{C}{\alpha_j}\|\hspace{-1mm}<\hspace{-1mm}\bphi_{h}\hspace{-1mm}>^n\hspace{-1mm}\|^2.
 \end{align*}
For the third trilinear term, apply the bound in \eqref{nonlinearbound}, triangle inequality, regularity assumption of the true solution, and Young’s inequality, to obtain 
	\begin{align*}	
 b^*\Big(\hspace{-1mm}<\hspace{-1mm}\bfeta\hspace{-1mm}>^n,\bu_j(t^{n+1})-\bu_j(t^{n}),\bphi_{j,h}^{n+1}\Big)&\le C\|\nabla \hspace{-1mm}<\hspace{-1mm}\bfeta\hspace{-1mm}>^n\hspace{-1mm}\|\|\|\nabla\left(\bu_j(t^{n+1})-\bu_j(t^{n})\right)\|\|\nabla \bphi_{j,h}^{n+1}\|\\&\le \frac{\alpha_j}{22}\|\nabla \bphi_{j,h}^{n+1}\|^2+\frac{C}{\alpha_j}\|\nabla\hspace{-1mm} <\hspace{-1mm}\bfeta\hspace{-1mm}>^n\hspace{-1mm}\|^2.
 \end{align*}
 For the fourth trilinear term, we use the bound in \eqref{nonlinearbound3}, Agmon’s \cite{Robinson2016Three-Dimensional} inequality, Sobolev embedding theorem, regularity assumption of the true solution, and Young’s inequalities, to reveal
	\begin{align*}
		-b^*\lp\bphi_{j,h}^n, \bu_j(t^n),\bphi_{j,h}^{n+1}\rp&\le C\|\bphi_{j,h}^n\|\left(\|\nabla\bu_j(t^n)\|_{L^3}+\|\bu_j(t^n)\|_{L^\infty}\right)\|\nabla\bphi_{j,h}^{n+1}\|\\&\le C\|\bphi_{j,h}^n\|\|\nabla\bphi_{j,h}^{n+1}\|\\&\le \frac{\alpha_j}{22}\|\nabla\bphi_{j,h}^{n+1}\|^2+\frac{C}{\alpha_j}\|\bphi_{j,h}^n\|^2.
	\end{align*}
For the fifth, and sixth trilinear terms, apply Young’s inequalities with \eqref{nonlinearbound}, to obtain 
	\begin{align*}	
b^*\lp<\hspace{-1mm}\bu_h\hspace{-1mm}>^n, \bfeta_{j}^{n+1},\bphi_{j,h}^{n+1}\rp&\le C\|\nabla\hspace{-1mm}<\hspace{-1mm}\bu_h\hspace{-1mm}>^n\hspace{-1mm}\|\|\nabla \bfeta_{j}^{n+1}\|\|\nabla\bphi_{j,h}^{n+1}\|\\&\le \frac{\alpha_j}{22}\|\nabla \bphi_{j,h}^{n+1}\|^2+\frac{C}{\alpha_j}\|\nabla\hspace{-1mm}<\hspace{-1mm}\bu_h\hspace{-1mm}>^n\hspace{-1mm}\|^2\|\nabla \bfeta_{j}^{n+1}\|^2,\\
	b^*\lp \bu_{j,h}^{'n},\bfeta_{j} ^n,\bphi_{j,h}^{n+1}\rp&\le C\|\nabla \bu_{j,h}^{'n}\|\|\nabla\bfeta_{j} ^n\|\|\nabla\bphi_{j,h}^{n+1}\|\\&\le \frac{\alpha_j}{22}\|\nabla \bphi_{j,h}^{n+1}\|^2+\frac{C}{\alpha_j}\|\nabla \bu_{j,h}^{'n}\|^2\|\nabla\bfeta_{j} ^n\|^2.
	\end{align*} 
 For the seventh nonlinear term, apply the bound in \eqref{nonlinearbound}, regularity assumption of the true solution, and Young’s inequality, to obtain
 \begin{align*}	
 b^*\lp\bfeta^n_{j}, \bu_j(t^n),\bphi_{j,h}^{n+1}\rp&\le C\|\nabla \bfeta^n_{j}\|\|\nabla \bu_j(t^n)\|\|\nabla\bphi_{j,h}^{n+1}\|\\&\le \frac{\alpha_j}{22}\|\nabla \bphi_{j,h}^{n+1}\|^2+\frac{C}{\alpha_j}\|\nabla \bfeta^n_{j}\|^2.
	\end{align*} 
 For the first eddy-viscosity term on the right hand side of \eqref{phibd}, we apply H\"older's inequality, Agmon’s \cite{Robinson2016Three-Dimensional} inequality for both 2D and 3D, Poincar\'e inequality, the regularity assumption of the true solution, and Young’s inequality, to get
	\begin{align*}
		2\mu\Delta t\lp(l^{n})^2\nabla \bu_{j}(t^{n+1}),\nabla\bphi_{j,h}^{n+1}\rp&\le C\mu\Delta t\|l^{n}\|_{L^4}^2\|\nabla \bu_{j}(t^{n+1})\|_{L^\infty} \|\nabla\bphi_{j,h}^{n+1}\|\\&\le\frac{\alpha_j}{22}\|\nabla \bphi_{j,h}^{n+1}\|^2+C\frac{\mu^2\Delta t^2}{\alpha_j}\|l^{n}\|_{L^4}^4.
	\end{align*}
	For the second eddy-viscosity term, we rearrange, and apply Cauchy-Schwarz, and Young’s inequalities assuming $\mu>1$,
	to obtain
	\begin{align*}
		2\mu\Delta t\lp (l^{n})^2\nabla \bfeta_{j}^{n+1},\nabla\bphi_{j,h}^{n+1}\rp&= 2\mu\Delta t\left(l^{n}\nabla \bfeta_{j}^{n+1},l^{n}\nabla\bphi_{j,h}^{n+1}\right)\\&\le 2\mu\Delta t\|l^{n}\nabla \bfeta_{j}^{n+1}\|\|l^{n}\nabla\bphi_{j,h}^{n+1}\|\\&\le \mu\Delta t\|l^{n}\nabla\bphi_{j,h}^{n+1}\|^2+\mu\Delta t\|l^{n}\nabla \bfeta_{j}^{n+1}\|^2.
	\end{align*}
	Using Taylor’s series expansion, Cauchy-Schwarz and Young's inequalities, the last term of \eqref{phibd} is bounded above as
	\begin{align*}
		\left|G(t,\bu_j, \bphi_{j,h}^{n+1})\right|\le\frac{\alpha_j}{22}\|\nabla\bphi_{j,h}^{n+1}\|^2+C\Delta t^2\Big(\|\bu_{j,tt}(t_1^*)\|^2+\|\nabla \bu_{j,t}(t_2^*)\|^2+\|\nabla \bu_{j,t}(t_3^{*})\|^2\|\nabla \bu_j(t^{n+1})\|^2\\+\|\nabla \big(\bu_j(t^n)-<\hspace{-1mm}\bu(t^n)\hspace{-1mm}>\hspace{-1mm}\big)\|^2\|\nabla \bu_{j,t}(t_4^{*})\|^2\Big),
	\end{align*}
	for some $t_i^*\in [t^n,t^{n+1}]$, $i=\overline{1,4}$.
	Using these estimates in \eqref{phibd} and reducing, produces
	\begin{align}
		\frac{1}{2\Delta t}\Big(&\|\bphi_{j,h}^{n+1}\|^2-\|\bphi_{j,h}^{n}\|^2\Big)+\left(\frac{1}{4\Delta t}-\frac{C}{\alpha_j}\|\nabla\cdot\bu_{j,h}^{'n}\|_{L^\infty}^2\right)\|\bphi_{j,h}^{n+1}-\bphi_{j,h}^{n}\|^2+\frac{\Bar{\nu}_{\min}}{2}\|\nabla\bphi_{j,h}^{n+1}\|^2\nonumber\\&+(2\mu-1)\Delta t\|l^n\nabla \bphi_{j,h}^{n+1}\|^2\le\frac{\|\nu_j^{'}\|_\infty}{2}\|\nabla\bphi_{j,h}^n\|^2+\left(\frac{11\Bar{\nu}^2}{2\alpha_j}+\frac{11\|\nu_j^{'}\|_\infty^2}{2\alpha_j}\right)\|\nabla \bfeta_{j}^{n+1}\|^2\nonumber\\&+\frac{C}{\alpha_j}\|\bphi_{j,h}^n\|^2+\frac{C}{\alpha_j}\|\hspace{-1mm}<\hspace{-1mm}\bphi_{h}\hspace{-1mm}>^n\hspace{-1mm}\|^2+\frac{C}{\alpha_j}\|\nabla\hspace{-1mm}<\hspace{-1mm}\bu_h\hspace{-1mm}>^n\hspace{-1mm}\|^2\|\nabla \bfeta_{j}^{n+1}\|^2+\frac{C}{\alpha_j}\|\nabla\hspace{-1mm} <\hspace{-1mm}\bfeta\hspace{-1mm}>^n\hspace{-1mm}\|^2\nonumber\\&+\frac{C}{\alpha_j}\|\nabla \bu_{j,h}^{'n}\|^2\|\nabla\bfeta_{j} ^n\|^2+\frac{C}{\alpha_j}\|\nabla \bfeta^n_{j}\|^2+C\frac{\mu^2\Delta t^2}{\alpha_j}\|l^{n}\|_{L^4}^4+\mu\Delta t\|l^{n}\nabla \bfeta_{j}^{n+1}\|^2\nonumber\\& +C\Delta t^2\Big(\|\bu_{j,tt}(t_3^*)\|^2+\|\nabla \bu_{j,t}(t_4^*)\|^2+\|\nabla \bu_{j,t}(t_5^{*})\|^2\|\nabla \bu_j(t^{n+1})\|^2\nonumber\\&+\|\nabla \big(\bu_j(t^n)-<\hspace{-1mm}\bu(t^n)\hspace{-1mm}>\hspace{-1mm}\big)\|^2\|\nabla \bu_{j,t}(t_6^{*})\|^2\Big).
	\end{align}
	Assuming $\mu>\frac12$, and time-step size $\Delta t<\frac{C\alpha_j}{\|\nabla\cdot\bu_{j,h}^{'n}\|^2_{L^\infty}}$, dropping non-negative terms from left, and rearranging
 \begin{align}
		\frac{1}{2\Delta t}\Big(&\|\bphi_{j,h}^{n+1}\|^2-\|\bphi_{j,h}^{n}\|^2\Big)+\frac{\Bar{\nu}_{\min}}{2}\left(\|\nabla\bphi_{j,h}^{n+1}\|^2-\|\nabla\bphi_{j,h}^{n}\|^2\right)+\frac{\alpha_j}{2}\|\nabla\bphi_{j,h}^{n}\|^2\nonumber\\&\le\left(\frac{11\Bar{\nu}^2}{2\alpha_j}+\frac{11\|\nu_j^{'}\|_\infty^2}{2\alpha_j}\right)\|\nabla \bfeta_{j}^{n+1}\|^2+\frac{C}{\alpha_j}\|\bphi_{j,h}^n\|^2+\frac{C}{\alpha_j}\|\hspace{-1mm}<\hspace{-1mm}\bphi_{h}\hspace{-1mm}>^n\hspace{-1mm}\|^2\nonumber\\&+\frac{C}{\alpha_j}\|\nabla\hspace{-1mm}<\hspace{-1mm}\bu_h\hspace{-1mm}>^n\hspace{-1mm}\|^2\|\nabla \bfeta_{j}^{n+1}\|^2+\frac{C}{\alpha_j}\|\nabla\hspace{-1mm} <\hspace{-1mm}\bfeta\hspace{-1mm}>^n\hspace{-1mm}\|^2+\frac{C}{\alpha_j}\|\nabla \bu_{j,h}^{'n}\|^2\|\nabla\bfeta_{j} ^n\|^2+\frac{C}{\alpha_j}\|\nabla \bfeta^n_{j}\|^2\nonumber\\&+C\frac{\mu^2\Delta t^2}{\alpha_j}\|l^{n}\|_{L^4}^4+\mu\Delta t\|l^{n}\nabla \bfeta_{j}^{n+1}\|^2 +C\Delta t^2\Big(\|\bu_{j,tt}(t_3^*)\|^2+\|\nabla \bu_{j,t}(t_4^*)\|^2\nonumber\\&+\|\nabla \bu_{j,t}(t_5^{*})\|^2\|\nabla \bu_j(t^{n+1})\|^2+\|\nabla \big(\bu_j(t^n)-<\hspace{-1mm}\bu(t^n)\hspace{-1mm}>\hspace{-1mm}\big)\|^2\|\nabla \bu_{j,t}(t_6^{*})\|^2\Big).
	\end{align}
 Multiplying both sides  by $2\Delta t$, sum over the time-steps $n=0,1,\cdots,M-1$, using $\|\bphi_{j,h}^0\|=\|\nabla\bphi_{j,h}^0\|=0$, $\Delta tM=T$, and using stability estimate and regularity assumptions, to find
	\begin{align}   &\|\bphi_{j,h}^{M}\|^2+\alpha_j\Delta t\sum_{n=1}^{M}\|\nabla\bphi_{j,h}^n\|^2\le C\Big(h^{2k}+\Delta t^2+\Delta t\sum_{n=1}^{M-1}\|\bphi_{j,h}^n\|^2\nonumber\\&+\Delta t\sum_{n=1}^{M-1}\|\hspace{-1mm}<\hspace{-1mm}\bphi_{h}\hspace{-1mm}>^n\hspace{-1mm}\|^2+\Delta t^3\sum_{n=0}^{M-1}\|l^{n}\|_{L^4}^4+\Delta t^2\sum_{n=0}^{M-1}\|l^{n}\nabla \bfeta_{j}^{n+1}\|^2\Big).\label{after-drop-non-negative-terms}
	\end{align}
	For the third sum on the right-hand-side of \eqref{after-drop-non-negative-terms}, using Young's inequality, we write
	\begin{align*}	\|l^{n}\|_{L^4}^4=\int_\cD (l^n)^4 d\cD=\int_\cD\left(\sum_{j=1}^J|\bu_{j,h}^{'n}|^2\right)^2d\cD\le 2\sum_{j=1}^J\int_\cD|\bu_{j,h}^{'n}|^4d\cD=2\sum_{j=1}^J\|\bu_{j,h}^{'n}\|^4_{L^4},
	\end{align*}
	and get different bounds for 2D and 3D due to different
	Sobolev embedding (Ladyzhenskaya's inequalities \cite{ladyzhenskaya1958solution,L08,MR17}) as below:
	\begin{align*}
		2D:\hspace{3mm} \|\bu_{j,h}^{'n}\|^4_{L^4}&\le C\|\bu_{j,h}^{'n}\|^2\|\nabla \bu_{j,h}^{'n}\|^2,\\
		3D:\hspace{3mm} \|\bu_{j,h}^{'n}\|^4_{L^4}&\le C\|\bu_{j,h}^{'n}\|\|\nabla \bu_{j,h}^{'n}\|^3.
	\end{align*}
	With the inverse inequality and the stability bound (used on the $L^2$ norm), we obtain
	the bounds for both 2D or 3D:
	\begin{align*}
		\|\bu_{j,h}^{'n}\|^4_{L^4}&\le Ch^{2-d}\|\nabla \bu_{j,h}^{'n}\|^2.
	\end{align*}
	Using the above bound and the stability bound, the third sum on the right-hand-side of \eqref{after-drop-non-negative-terms} is bounded as
	\begin{align*}
		C\Delta t^3\sum_{n=0}^{M-1}\|l^{n}\|_{L^4}^4\le Ch^{2-d}\Delta t^2\sum_{n=0}^{M-1}\sum_{j=1}^J\Delta t\|\nabla \bu_{j,h}^{'n}\|^2\le Ch^{2-d}\Delta t^2.
	\end{align*}
	For the last sum on the right-hand-side of \eqref{after-drop-non-negative-terms}, we use triangle, Agmon's \cite{Robinson2016Three-Dimensional}, and the inverse \cite{BS08} inequalities, standard estimates of the $L^2$ projection error in the $H^1$ norm for the finite element functions, and the stability estimate, to obtain
	\begin{align*}
		C\Delta t^2\sum_{n=0}^{M-1}\|l^{n}\nabla \bfeta_{j}^{n+1}\|^2&\le C\Delta t^2\sum_{n=0}^{M-1}\|(l^{n})^2\|_{\infty}\|\nabla \bfeta_{j}^{n+1}\|^2\\&\le Ch^{-1}\Delta t^2\sum_{n=0}^{M-1}\sum_{j=1}^J\|\nabla \bu_{j,h}^{'n}\|^2\|\nabla \bfeta_{j}^{n+1}\|^2\\&\le Ch^{2k-1}\Delta t\sum_{n=0}^{M-1}\lp\Delta t\|\nabla \bu_{j,h}^{'n}\|^2\rp| \bu_j^{n+1}|_{k+1}^2\\
		&\le Ch^{2k-1}\Delta t.
	\end{align*}
	Using the above bounds, and Young's inequality in \eqref{after-drop-non-negative-terms}, we have
	\begin{align}   \|\bphi_{j,h}^{M}\|^2+\alpha_j\Delta t\sum_{n=1}^{M}\|\nabla\bphi_{j,h}^n\|^2\nonumber\\\le C\Big(h^{2k}+\Delta t^2+\Delta t\left(1+\frac{2}{J^2}\right)\sum_{n=1}^{M-1}\|\bphi_{j,h}^n\|^2+h^{2-d}\Delta t^2+h^{2k-1}\Delta t\Big).
	\end{align}
	Sum over $j=1,\cdots\hspace{-0.35mm},J$, apply triangle, and Young's inequalities, to get
	\begin{align}   \sum_{j=1}^J\|\bphi_{j,h}^{M}\|^2+\Delta t\sum_{n=1}^{M}\sum_{j=1}^J\alpha_j\|\nabla\bphi_{j,h}^n\|^2&\le C\Delta t\left(1+\frac{2}{J^2}\right)\sum_{n=1}^{M-1}\sum_{j=1}^J\|\bphi_{j,h}^n\|^2 \nonumber\\&+C\Big(h^{2k}+\Delta t^2+h^{2-d}\Delta t^2+h^{2k-1}\Delta t\Big).
	\end{align}
	Applying the discrete Gr\"onwall Lemma \ref{dgl}, we have
	\begin{align}   \sum_{j=1}^J\|\bphi_{j,h}^{M}\|^2+\Delta t\sum_{n=1}^{M}\sum_{j=1}^J\alpha_j\|\nabla\bphi_{j,h}^n\|^2&\le C\Big(h^{2k}+\Delta t^2+h^{2-d}\Delta t^2+h^{2k-1}\Delta t\Big).
	\end{align}
	Now, using the triangle and Young's inequalities, we can write
	\begin{align}
\sum_{j=1}^J\|\be_{j}^{M}\|^2+\Delta t\sum_{n=1}^{M}\sum_{j=1}^J\alpha_j\|\nabla \be_{j}^{n}\|^2\le C\Big(h^{2k}+\Delta t^2+h^{2-d}\Delta t^2+h^{2k-1}\Delta t\Big).\label{error-bounds-j-level}
	\end{align}
	Finally, again use the triangle and Young's inequalities to complete the proof.
\end{proof}
\bibliographystyle{plain}
% mike
\bibliography{Penalty-Pro}
%traian
%\bibliography{../../../../bibliography/traian/traian,../../../../bibliography/schneier/mike}
\end{document}

%% file: notation-siam.tex
\newcommand{\lp}{\left(}
\newcommand{\rp}{\right)}

\newcommand{\lab}{<\hspace{-1mm}}
\newcommand{\rab}{\hspace{-1mm}>}

\newtheorem{remark}{Remark}[section]
\newtheorem{assumption}{Assumption}[section]

\def\PP{{{\rm l}\kern - .15em {\rm P} }}
\def\PN2{{\PP_{N}-\PP_{N-2}}}

\newcommand{\cD}{\mathcal{D}}

\newcommand{\bfeta}{\boldsymbol{\eta}}

\newcommand{\bphi}{\boldsymbol{\varphi}}

\newcommand{\bif}{\textbf{\textit{f}}\hspace{-0.5mm}}

\newcommand{\bb}{\boldsymbol{b}}

\newcommand{\bchi}{\pmb{\chi}}

\newcommand{\be}{\boldsymbol{e}}

\newcommand{\bH}{\boldsymbol{H}}

\newcommand{\bL}{\boldsymbol{L}}

\newcommand{\hl}{{\hat{l}}}
\newcommand{\hp}{{\hat{p}}}
\newcommand{\hq}{{\hat{q}}}

\newcommand{\bR}{\boldsymbol{R}}

\newcommand{\bu}{\boldsymbol{u}}

\newcommand{\bv}{\boldsymbol{v}}

\newcommand{\bhu}{\hat{\boldsymbol{u}}}
\newcommand{\bhv}{\hat{\boldsymbol{v}}}
\newcommand{\bnh}{\hat{\textbf{\textit{n}}}}
\newcommand{\hlam}{\hat{\lambda}}

\newcommand{\bV}{\boldsymbol{V}}

\newcommand{\bw}{\boldsymbol{w}}
\newcommand{\bW}{\boldsymbol{W}}
\newcommand{\bhw}{\hat{\boldsymbol{w}}}

\newcommand{\btu}{\tilde{\boldsymbol{u}}}

\newcommand{\bx}{\boldsymbol{x}}
\newcommand{\bX}{\boldsymbol{X}}

\newcommand{\bY}{\boldsymbol{Y}}
\newcommand{\by}{\boldsymbol{y}}

\newcommand{\bGamma}{\boldsymbol{\Gamma}}

\newcommand{\deleted}[1]{{}}%\color{red}{#1}}}

%% file: main.bbl
\begin{thebibliography}{10}

\bibitem{AKMR15}
M.~Akbas, S.~Kaya, M.~Mohebujjaman, and L.~Rebholz.
\newblock Numerical analysis and testing of a fully discrete, decoupled
  penalty-projection algorithm for {MHD} in {E}ls{\"{a}}sser variable.
\newblock {\em International Journal of Numerical Analysis $\&$ Modeling},
  13(1):90--113, 2016.

\bibitem{arnold1992quadratic}
D.~Arnold and J.~Qin.
\newblock Quadratic velocity/linear pressure {S}tokes elements.
\newblock {\em Advances in computer methods for partial differential
  equations}, 7:28--34, 1992.

\bibitem{babuvska2007stochastic}
I.~Babu{\v{s}}ka, F.~Nobile, and R.~Tempone.
\newblock A stochastic collocation method for elliptic partial differential
  equations with random input data.
\newblock {\em SIAM Journal on Numerical Analysis}, 45(3):1005--1034, 2007.

\bibitem{balajewicz2013low}
M.~J. Balajewicz, E.~H Dowell, and Bernd~R. N.
\newblock Low-dimensional modelling of high-reynolds-number shear flows
  incorporating constraints from the navier-stokes equation.
\newblock {\em Journal of Fluid Mechanics}, 729:285, 2013.

\bibitem{BS08}
S.~C. Brenner and L.~R. Scott.
\newblock {\em The Mathematical Theory of Finite Element Methods}, volume~15 of
  {\em Texts in Applied Mathematics}.
\newblock Springer Science+Business Media, LLC, 2008.

\bibitem{erkmen2020second}
D.~Erkmen, S.~Kaya, and A.~{\c{C}}{\i}b{\i}k.
\newblock A second order decoupled penalty projection method based on deferred
  correction for {MHD} in {E}ls{\"a}sser variable.
\newblock {\em Journal of Computational and Applied Mathematics}, 371:112694,
  2020.

\bibitem{fick2018stabilized}
L.~Fick, Y.~Maday, A.~T. Patera, and T.~Taddei.
\newblock A stabilized {POD} model for turbulent flows over a range of
  {R}eynolds numbers: {O}ptimal parameter sampling and constrained projection.
\newblock {\em Journal of Computational Physics}, 371:214--243, 2018.

\bibitem{Unconditionally2021Fiordilino}
J.~A. Fiordilino and M.~Winger.
\newblock Unconditionally energy stable and first-order accurate numerical
  schemes for the heat equation with uncertain temperature-dependent
  conductivity.
\newblock {\em International Journal of Numerical Analysis and Modeling},
  20:805--831, 2023.

\bibitem{fujita2007surface}
T.~Fujita, D.~J. Stensrud, and D.~C. Dowell.
\newblock Surface data assimilation using an ensemble {K}alman filter approach
  with initial condition and model physics uncertainties.
\newblock {\em Monthly weather review}, 135(5):1846--1868, 2007.

\bibitem{GR86}
V.~Girault and P.-A.Raviart.
\newblock {\em Finite element methods for Navier-Stokes equations: Theory and
  Algorithms}.
\newblock Springer-Verlag, 1986.

\bibitem{gunzburger2019evolve}
M.~Gunzburger, T.~Iliescu, M.~Mohebujjaman, and M.~Schneier.
\newblock An evolve-filter-relax stabilized reduced order stochastic
  collocation method for the time-dependent {N}avier--{S}tokes equations.
\newblock {\em SIAM/ASA Journal on Uncertainty Quantification},
  7(4):1162--1184, 2019.

\bibitem{GJW18}
M.~Gunzburger, N.~Jiang, and Z.~Wang.
\newblock A second-order time-stepping scheme for simulating ensembles of
  parameterized flow problems.
\newblock {\em Computational Methods in Applied Mathematics}, to appear, 2018.

\bibitem{HR90}
J.~G. Heywood and R.~Rannacher.
\newblock Finite-element approximation of the nonstationary {N}avier-{S}tokes
  problem part iv: error analysis for second-order time discretization.
\newblock {\em SIAM Journal on Numerical Analysis}, 27:353--384, 1990.

\bibitem{jiang2015higher}
N.~Jiang.
\newblock A higher order ensemble simulation algorithm for fluid flows.
\newblock {\em Journal of Scientific Computing}, 64:264--288, 2015.

\bibitem{jiang2017second}
N.~Jiang.
\newblock A second order ensemble method based on a blended {BDF} timestepping
  scheme for time dependent {Navier-Stokes} equations.
\newblock {\em Numerical Methods for Partial Differential Equations}, to
  appear, 2016.

\bibitem{jiang2015analysis}
N.~Jiang, S.~Kaya, and W.~Layton.
\newblock Analysis of model variance for ensemble based turbulence modeling.
\newblock {\em Computational Methods in Applied Mathematics}, 15(2):173--188,
  2015.

\bibitem{JKL15}
N.~Jiang, S.~Kaya, and W.~Layton.
\newblock Analysis of model variance for ensemble based turbulence modeling.
\newblock {\em Computational Methods in Applied Mathematics}, 15:173--188,
  2015.

\bibitem{JL14}
N.~Jiang and W.~Layton.
\newblock An algorithm for fast calculation of flow ensembles.
\newblock {\em International Journal for Uncertainty Quantification},
  4:273--301, 2014.

\bibitem{jiang2015numerical}
N.~Jiang and W.~Layton.
\newblock Numerical analysis of two ensemble eddy viscosity numerical
  regularizations of fluid motion.
\newblock {\em Numerical Methods for Partial Differential Equations},
  31:630--651, 2015.

\bibitem{jiang2021artificial}
N.~Jiang, Y.~Li, and H.~Yang.
\newblock An artificial compressibility {C}rank--{N}icolson leap-frog method
  for the {S}tokes--{D}arcy model and application in ensemble simulations.
\newblock {\em SIAM Journal on Numerical Analysis}, 59(1):401--428, 2021.

\bibitem{jiang2018efficient}
N.~Jiang and M.~Schneier.
\newblock An efficient, partitioned ensemble algorithm for simulating ensembles
  of evolutionary {MHD} flows at low magnetic reynolds number.
\newblock {\em Numerical Methods for Partial Differential Equations},
  34(6):2129--2152, 2018.

\bibitem{ladyzhenskaya1958solution}
O.~A. Ladyzhenskaya.
\newblock Solution ``in the large'' to the boundary value problem for the
  {N}avier-{S}tokes equations in two space variables.
\newblock 123(3):427--429, 1958.

\bibitem{L08}
W.~Layton.
\newblock {\em Introduction to the Numerical Analysis of Incompressible Viscous
  Flows}.
\newblock Computational Science and Engineering. Society for Industrial and
  Applied Mathematics, 2008.

\bibitem{layton2008numerical}
W.~Layton, C.~C. Manica, M.~Neda, and L.~G. Rebholz.
\newblock Numerical analysis and computational testing of a high accuracy
  leray-deconvolution model of turbulence.
\newblock {\em Numerical Methods for Partial Differential Equations: An
  International Journal}, 24(2):555--582, 2008.

\bibitem{lee2011error}
H.~K. Lee, M.~A. Olshanskii, and L.~G. Rebholz.
\newblock On error analysis for the 3d navier--stokes equations in
  velocity-vorticity-helicity form.
\newblock {\em SIAM Journal on Numerical Analysis}, 49(2):711--732, 2011.

\bibitem{lee2019study}
M.~W. Lee, E.~H. Dowell, and M.~J. Balajewicz.
\newblock A study of the regularized lid-driven cavity’s progression to
  chaos.
\newblock {\em Communications in Nonlinear Science and Numerical Simulation},
  71:50--72, 2019.

\bibitem{L05}
J.~M. Lewis.
\newblock Roots of ensemble forecasting.
\newblock {\em Monthly Weather Review}, 133:1865 -- 1885, 2005.

\bibitem{linke2017connection}
A.~Linke, M.~Neilan, L.~G. Rebholz, and N.~E. Wilson.
\newblock A connection between coupled and penalty projection timestepping
  schemes with {FE} spatial discretization for the {N}avier--{S}tokes
  equations.
\newblock {\em Journal of Numerical Mathematics}, 25(4):229--248, 2017.

\bibitem{GJW17}
N.~Jiang M.~Gunzburger and Z.~Wang.
\newblock A second-order time-stepping scheme for simulating ensembles of
  parameterized flow problems.
\newblock {\em Computational Methods in Applied Mathematics}, 1(4):349--364,
  1988.

\bibitem{LP08}
T.~N.~Palmer M.~Leutbecher.
\newblock Ensemble forecasting.
\newblock {\em Journal of Computational Physics}, 227:3515--3539, 2008.

\bibitem{LK10}
O.~P.~L. Ma\^{i}tre and O.~M. Knio.
\newblock {\em Spectral methods for uncertainty quantification}.
\newblock Springer, 2010.

\bibitem{MX06}
W.~J. Martin and M.~Xue.
\newblock Sensitivity analysis of convection of the 24 {M}ay 2002 {IHOP} case
  using very large ensembles.
\newblock {\em Monthly Weather Review}, 134(1):192--207, 2006.

\bibitem{Mohebujjaman2022High}
M.~Mohebujjaman.
\newblock High order efficient algorithm for computation of {MHD} flow
  ensembles.
\newblock {\em Advances in {A}pplied {M}athematics and {M}echanics},
  14(5):1111--1137, 2022.

\bibitem{mohebujjaman2024decoupled}
M.~Mohebujjaman, C.~Buenrostro, M.~Kamrujjaman, and T.~Khan.
\newblock Decoupled algorithms for non-linearly coupled reaction--diffusion
  competition model with harvesting and stocking.
\newblock {\em Journal of Computational and Applied Mathematics}, 436:115421,
  2024.

\bibitem{mohebujjaman2024efficient}
M.~Mohebujjaman, J.~Miranda, M.~A.~A. Mahbub, and M.~Xiao.
\newblock An efficient and accurate penalty-projection eddy viscosity algorithm
  for stochastic magnetohydrodynamic flow problems.
\newblock {\em Journal of Scientific Computing}, 101(1):2, 2024.

\bibitem{MR17}
M.~Mohebujjaman and L.~G. Rebholz.
\newblock An efficient algorithm for computation of {MHD} flow ensembles.
\newblock {\em Computational Methods in Applied Mathematics}, 17:121--137,
  2017.

\bibitem{mohebujjaman2022efficient}
M.~Mohebujjaman, H.~Wang, L.~G. Rebholz, and M.~A. A.~Mahbub.
\newblock An efficient algorithm for parameterized magnetohydrodynamic flow
  ensembles simulation.
\newblock {\em Computers \& Mathematics with Applications}, 112:167--180, 2022.

\bibitem{neda2016ensemble}
M.~Neda, A.~Takhirov, and J.~Waters.
\newblock Ensemble calculations for time relaxation fluid flow models.
\newblock {\em Numerical Methods for Partial Differential Equations},
  32(3):757--777, 2016.

\bibitem{FTW2008}
F.~Nobile, R.~Tempone, and C.~G Webster.
\newblock An anisotropic sparse grid stochastic collocation method for partial
  differential equations with random input data.
\newblock {\em SIAM J. Num. Anal.}, 46(5):2411--2442, 2008.

\bibitem{GG11}
J.~D.~Giraldo Osorio and S.~G.~Garcia Galiano.
\newblock Building hazard maps of extreme daily rainy events from {PDF}
  ensemble, via {REA} method, on {Senegal} river basin.
\newblock {\em Hydrology and Earth System Sciences}, 15:3605 -- 3615, 2011.

\bibitem{Robinson2016Three-Dimensional}
J.~C. Robinson, J.~L. Rodrigo, and W.~Sadowski.
\newblock {\em The Three-{D}imensional Navier-{S}tokes Equations}.
\newblock Cambridge University Press, 2016.

\bibitem{stoyanov2015tasmanian}
M.~Stoyanov.
\newblock User manual: Tasmanian sparse grids.
\newblock Technical Report ORNL/TM-2015/596, Oak Ridge National Laboratory, One
  Bethel Valley Road, Oak Ridge, TN, 2015.

\bibitem{doecode_6305}
M.~Stoyanov, D.~Lebrun-Grandie, J.~Burkardt, and D.~Munster.
\newblock Tasmanian, 9 2013.

\bibitem{taylor1937mechanism}
G.~I. Taylor and A.~E. Green.
\newblock Mechanism of the production of small eddies from large ones.
\newblock {\em Proceedings of the Royal Society of London. Series
  A-Mathematical and Physical Sciences}, 158(895):499--521, 1937.

\bibitem{Z05}
S.~Zhang.
\newblock A new family of stable mixed finite elements for the 3{D} {S}tokes
  equations.
\newblock {\em Mathematics of Computation}, 74:543--554, 2005.

\end{thebibliography}
